\documentclass[11pt,thmsa]{amsart}

\usepackage{amssymb,latexsym}
\usepackage{amsmath,amscd}
\usepackage[pdftex,all]{xy}
\usepackage{color}
\usepackage{hyperref}
\usepackage{esvect}
\usepackage{framed}
\usepackage{comment}
\usepackage{amsmath}
\usepackage{amsthm}
\usepackage{amsfonts}
\usepackage{amssymb}
\usepackage[pdftex]{graphicx}
\usepackage{wrapfig}
\usepackage{layout}
\usepackage{verbatim}
\usepackage{alltt}
\usepackage{ mathrsfs }
\usepackage{tabularx}

\newtheorem{theorem}{Theorem}[section]
\newtheorem{lemma}[theorem]{Lemma}
\newtheorem{proposition}[theorem]{Proposition}

\newtheorem{corollary}[theorem]{Corollary}
\newtheorem{definition}[theorem]{Definition}
\theoremstyle{definition}
\newtheorem{example}[theorem]{Example}
\newtheorem{remark}[theorem]{Remark}

\newtheorem*{acknowledgements}{Acknowledgements}

\newcommand{\g}{\mathfrak{g}}
\newcommand{\id}{\mathrm{id}}

\newcommand{\U}{\mathcal{U}}

\newcommand{\pre}{\vartriangleright}
\newcommand{\sym}{\mathrm{sym}}
\newcommand{\Coder}{\mathrm{Coder}}
\newcommand{\End}{\mathrm{End}}
\newcommand{\Lie}{\mathrm{L}}
\newcommand{\K}{\mathbb{K}}
\newcommand{\bul}{\bullet}
\newcommand{\Aut}{\operatorname{Aut}}
\newcommand{\T}{\mathcal{T}}
\newcommand{\Adm}{\mathscr{D}}
\newcommand{\pbT}{\mathcal{T}_{\mathrm{pb}}}
\newcommand{\pblT}{\mathcal{T}_{\mathrm{pbl}}}
\renewcommand{\Lie}{\mathcal{L}}
\newcommand{\basis}{\mathcal{B}}
\newcommand{\specular}{\Sigma}
\newcommand{\pT}{\mathcal{T}_{\mathrm{p}}}

\newcommand{\pol}{\mathrm{alt}}

\newcommand{\slab}[1]{\text{\footnotesize{$#1$}}}

\newcommand{\altbin}[2]{
	P\left(#2\right)(#1)  }

\def\aaltbin#1#2{\ensuremath{\left(\kern-.35em\left(\genfrac{}{}{0pt}{}{#1}{#2}\right)\kern-.35em\right)} }

\begin{document}
	
	\title[Eulerian idempotent, pre-Lie logarithm and combinatorics of trees]{Eulerian idempotent, pre-Lie logarithm and combinatorics of trees}

	\author{Ruggero Bandiera}
	\address{Universit\`a degli studi di Roma La Sapienza,
		Dipartimento di Matematica ``Guido Castelnuovo",
		P.le Aldo Moro 5, I-00185 Roma, Italy.}
	\email{bandiera@mat.uniroma1.it}

	\author{Florian Sch\"atz}
	\address{University of Luxembourg, Mathematics Research Unit, 
		Maison du Nombre,
6, avenue de la Fonte,
L-4364 Esch-sur-Alzette, Luxembourg.
	}
	\email{florian.schaetz@gmail.com}

	\maketitle

	\begin{abstract} The aim of this paper is to bring together the three objects in the title. Recall that, given a Lie algebra $\mathfrak{g}$, the Eulerian idempotent is a canonical projection from the enveloping algebra $\mathcal{U}(\mathfrak{g})$ to $\mathfrak{g}$. The Baker-Campbell-Hausdorff product and the Magnus expansion can both be expressed in terms of the Eulerian idempotent, which makes it interesting to establish explicit formulas for the latter. We show how to reduce the computation of the Eulerian idempotent to the computation of a logarithm in a certain pre-Lie algebra of planar, binary, rooted trees. The problem of finding formulas for the pre-Lie logarithm, which is interesting in its own right -- being related to operad theory, numerical analysis and renormalization -- is addressed using techniques inspired by umbral calculus. As a consequence of our analysis, we find formulas both for the Eulerian idempotent and the pre-Lie logarithm in terms of the combinatorics of trees.
	\end{abstract}

	\setcounter{tocdepth}{2} 
	
	\tableofcontents
	
	\newpage

	\section{{\bf Introduction}}\label{section: Eulerian}
	
	\subsection{The problem...}\label{subsection: problem}
	
	Given a Lie algebra $\g$ over a filed $\K$ of characteristic zero, the Poincar\'e-Birkhoff-Witt Theorem states that the symmetrization map
	$$ \sym\colon S(\g) \to \U(\g)\colon \quad x_1\odot \cdots \odot x_n \mapsto \frac{1}{n!}\sum_{\sigma \in S_n}x_{\sigma(1)}\cdots x_{\sigma(n)}, $$
	going from the symmetric coalgebra $S(\g)$ over $\g$ to its universal enveloping algebra $\U(\g)$, is a natural isomorphism of filtered coalgebras. One can therefore consider the map
	$$ E\colon \U(\g) \xrightarrow{\sym^{-1}} S(\g) \xrightarrow{p} \g \hookrightarrow \U(\g), $$
	where the second map $p$ is the natural projection and the last map $\g \hookrightarrow \U(\g)$ is the natural inclusion. The following explicit formula for $E$ was found by Solomon \cite{Solomon}
	\begin{equation}\label{equation: Eulerian}
	E(x_{1}\cdots x_{n}) = \frac{1}{n} \sum_{\sigma \in S_n}\frac{(-1)^{d_\sigma}}{{n-1 \choose d_\sigma}} x_{\sigma(1)}\cdots x_{\sigma(n)},\end{equation}
	where $d_\sigma$ is the {\em descent number} of the permutation $\sigma$, i.e.
	$ d_\sigma := | \{ 1 \le i \le n-1 \, | \, \sigma(i) > \sigma(i+1)\}|$.
	
	A bit improperly, we shall call the projection $E$ the \emph{Eulerian idempotent} on $\U(\g)$\footnote{More properly, the name usually refers to a corresponding idempotent in the group algebra $\K[S_n]$ of the symmetric group, see \cite{Loday}.}. We point out that neither the fact that $E$ is an idempotent, that is, $E\circ E = E$, nor the fact that $E$ takes values in $\g\subset\U(\g)$, is apparent from the previous formula. To solve the second problem, we may compose $E$ with another well-known Lie idempotent, namely, the \emph{Dynkin idempotent} $D\colon \U(\g)\to\g, \, x_1\cdots x_n\mapsto \frac{1}{n}[x_1,\ldots[x_{n-1},x_n]\ldots]$. Since both $E$ and $D$ are projectors onto $\g\subset\U(\g)$, we have $E=D\circ E$, hence
	\begin{equation}\label{equation: Eulerian2}
	E(x_{1}\cdots x_{n}) = \frac{1}{n^2} \sum_{\sigma \in S_n}\frac{(-1)^{d_\sigma}}{{n-1 \choose d_\sigma}}[ x_{\sigma(1)},\cdots[x_{\sigma(n-1)}, x_{\sigma(n)}]\cdots].\end{equation}
	Now the problem is that the iterated brackets appearing on the right hand side of \eqref{equation: Eulerian2} are not linearly independent among each other, due to the antisimmetry of the Lie bracket and the Jacobi identity. 
	For instance, for $n=3$ formula \eqref{equation: Eulerian2} becomes
	\begin{eqnarray*}  E(x_1 x_2 x_3) &=& \frac{1}{9}\Big([ x_1,[x_2, x_3]] - \frac{1}{2}[ x_1,[x_3, x_2]] - \frac{1}{2}[ x_2,[x_1, x_3]]+\\ \quad && \quad \quad- \frac{1}{2}[ x_2,[x_3, x_1]]- \frac{1}{2}[ x_3,[x_1, x_2]] + [ x_3,[x_2, x_1]]\Big), \end{eqnarray*}
	but, after some manipulation, this can be simplified to 
	\[ E(x_1x_2x_3) = \frac{1}{3}[ x_1,[x_2, x_3]]-\frac{1}{6}[x_2, [x_1,x_3]]=\frac{1}{6}[ x_1,[x_2, x_3]]+\frac{1}{6}[ [x_1,x_2], x_3]]. \] 
	
	From here on we shall focus on the universal case, that is, in the previous discussion we set $\g=L_n:= L(x_1,\ldots,x_n)$, the free Lie algebra over $x_1,\ldots,x_n$, and $\U(\g)=A_n:=A(x_1,\ldots,x_n)$, the free associative algebra over $x_1,\ldots,x_n$.
	
	\begin{definition}\label{definition: Lie_n}
		We denote by $\Lie_n \subset L_n$ the vector subspace spanned by those Lie words
		in which each generator appears {\em exactly} once. 
	\end{definition}		
	
	We remark that the vector space $\Lie_n$ is finite dimensional, of dimension $\operatorname{dim}(\Lie_n)=(n-1)!$. It is apparent from formula \eqref{equation: Eulerian2} that $E(x_1\cdots x_n)\in\Lie_n$. We shall study the following
	\begin{center}
		\begin{framed}
			{\em \underline{Problem A}: find the expansion of $E(x_1\cdots x_n)$ with respect to a basis of $\Lie_n$.}
		\end{framed}
	\end{center}
	
	Before we go further, let us provide some motivation why one might be interested in the above problem.
	
	\begin{remark}\label{remark: meaning of E}
		The Eulerian idempotent $E$ is related to the following topics (among others):
		\begin{itemize}
			\item {\em Baker-Campbell-Hausdorff product}: given a pro-nilpotent Lie algebra $\g$, an element $x\in\g$ and the corresponding group-like element $e^x\in\widehat{\U(\g)}$, it follows from the definitions that $E(e^x)=x$. In particular, given $x_1,\ldots,x_n\in\g$, we get the following formula for their Baker-Campbell-Hausdorff product $x_1\bullet\cdots\bullet x_n$ (cf.~\cite{Loday}):
			\[ x_1\bullet\cdots\bullet x_n = E(e^{x_1\bullet\cdots\bullet x_n}) = E(e^{x_1}\cdots e^{x_n})=\sum_{i_1,\ldots,i_n\geq0}\frac{1}{i_1!\cdots i_n!} E(x_1^{i_1}\cdots x_n^{i_n}).  \]
			\item {\em Magnus expansion}: given a matrix Lie group $G\subset GL(k)$ with Lie algebra $\g\subset \mathfrak{gl}(k)$, together with a Lipschitz continuous function $a(-)\colon \mathbb{R}^+\to \g$, consider the following ordinary differential equation in the space of $(k\times k)$ matrices (where $\cdot$ is matrix multiplication):
			\begin{equation}\label{eq:Magnus differential equation}
			\left\{\begin{array}{l} X'(t) = a(t) \cdot X(t) \\
			X(0) = \operatorname{Id_k}
			\end{array}  \right.
			\end{equation}
			It is well-known that the solution $X(t)$ to \eqref{eq:Magnus differential equation} satisfies $X(t)\in G$ at all times. In particular, using the exponential map $\exp\colon\g\to G$, we can write it (at least near $t=0$) in the form $X(t)=\exp(\omega(t))$, where $\omega(t)$ is a function with values in $\g$. When $G=GL(1)$ (or, more in general, when $G$ is abelian) $\omega(t)=\int_0^t a(\tau)d\tau$, which recovers the well-known formula $X(t)=e^{\int_0^t a(\tau)d\tau}$. The general case was studied by Magnus \cite{Magnus}, who found a series expansion for $\omega(t)$ which has since then been called the {\em Magnus expansion}, see \cite{Iserles-Norsett} for further details. In the paper \cite{Mielnik-Plebanski},  Mielnik and Pleb\'anski discovered the following formula
			for $\omega(t)$ in terms of the Eulerian idempotent
			$$\omega(t) = \sum_{n\ge 1} \int\limits_{\Delta_n(t)} E(a(t_1)\cdots a(t_n))dt_1\cdots dt_n,$$
			where the domains of integration $\Delta_n(t)$ are the $n$-simplices of size $t$,
			i.e.,
			$$ \Delta_n(t) := \{(t_1,\dots,t_n)\in \mathbb{R}^n \, | \, 0 \le t_n \le \cdots \le t_1\le t\}.$$
			Let us mention that the Magnus expansion is related to a problem concerning
			the rational homotopy theory of one-dimensional CW-complexes, investigated by us in \cite{Bandiera-Schaetz}. This provided our original motivation to study the previously stated Problem A.
			\item {\em Hochschild homology}: Given a commutative algebra $A$,
			the complex of Hochschild chains $HC_\bullet(A)$ can be split into
			several subcomplexes in terms of the Eulerian idempotent, see \cite{Barr,Gerstenhaber-Schack}. One of these pieces is the image of $E$, and the cohomology of this subcomplex
			identifies with the Harrison homology of $A$.
		\end{itemize}
	\end{remark}

	\subsection{...and our solution}\label{subsection: solution}
	
	We shall provide an answer to our problem for two particular bases of $\Lie_n$. Our main focus throughout the paper will be the study of the expansion of $E(x_1\cdots x_n)$ with respect to a certain basis $\mathcal{B}_n$ of $\Lie_n$, which we introduce next:

	\begin{definition}\label{definition: basis} We denote by $\Lie_{\le n}\subset L_n$ the subspace spanned by those Lie words in which each generator appears {\em at most} once (in particular, we can regard $\Lie_n$ as the subspace of $\Lie_{\leq n}$ spanned by Lie words of lenght $n$). Given a Lie word $w$ in $\Lie_{\le n}$,
		we refer to
		$$ \max\{i \, \vert \, x_i \textrm{ appears in } w\}$$
		as the maximum of $w$, and to
		$$\min\{i \, \vert \, x_i \textrm{ appears in }w\}$$
		as the minimum of $w$.
		
		The \emph{PBW basis} $\basis_{n}$ of $\Lie_{n}$ consists of those Lie words $w$ in $\Lie_n$ such that, for any Lie sub-word $v\subset w$ inside $w$, either $v$ is one of the generators $x_1,\ldots,x_n$, or $v=[v',v'']$, where the minimum of $v'$ is smaller than the minimum of $v''$, and the maximum of $v'$ is smaller than the maximum of $v''$ (cf.~\cite[\S 13.2.5.2]{Loday-Vallette}, and references therein).
	\end{definition}
	
	\begin{example}
		For $n\leq4$, the PBW basis $\basis_n$ is as follows:
		\begin{eqnarray*}
			&\basis_1 = \{x_1\}, \qquad \basis_2 = \{[x_1,x_2]\},&\\ 
			&\basis_3 = \{[x_1,[x_2,x_3]],[[x_1,x_2],x_3]\},& \\
			&\basis_4 = \{ [x_1,[x_2,[x_3,x_4]]], [x_1,[[x_2,x_3],x_4]], [[x_1,x_2],[x_3,x_4]],& \\ & \quad \qquad [[x_1,x_3],[x_2,x_4]], [[x_1, [x_2,x_3]],x_4], [[[x_1,x_2],x_3],x_4] \}&
		\end{eqnarray*}
	\end{example}
	
	\begin{remark} The space $\Lie_{n}$, with its obvious structure of an $S_n$-module, is the $n$-ary component of the operad $\mathcal{L}ie$ encoding Lie algebras. The basis $\basis_{n}$ arises naturally in the study of this operad:
		in particular, the name PBW basis is borrowed from \cite{Loday-Vallette}.
	\end{remark}
	
	\begin{definition}\label{definition: Eulerian coefficient Lie}
		For every element $b\in \basis_n$, we shall denote by $E_b$ the coefficient of $b$ in the expansion of $E(x_1\cdots x_n)$ with respect to the PBW basis, and we shall call it the {\em Eulerian coefficient} of $b$. Thus
		$$ E(x_1\cdots x_n) = \sum_{b\in \basis_n} E_b  \, b.$$
	\end{definition}
	In this context, we may reformulate Problem A from the previous subsection more explicitly as follows:
	\begin{center}
		\begin{framed}
			{\em \underline{Problem B}: given an element $b\in \basis_n$, how can we compute its Eulerian coefficient $E_b$?}
		\end{framed}
	\end{center}
	
The second basis of $\Lie_n$, with respect to which we will address Problem A,
is Dynkin's basis.
	
	\begin{definition} The \emph{Dynkin's basis} $\mathcal{D}_n$ of $\Lie_n$ is
		\[  \mathcal{D}_n := \big\{ [x_{\sigma(1)},\cdots[x_{\sigma(n-1)},x_n]\cdots ] \big\}_{\sigma\in S_{n-1}}.  \]
	\end{definition}
	At the very end of this paper (see Subsection \ref{subsection: eulerian in dynkin}) we shall prove the following as a byproduct of our previous analysis:
	\begin{theorem}\label{th:eulerian in Dynkin basis} The expansion of $E(x_1\cdots x_n)$ with respect to the Dynkin's basis $\mathcal{D}_n$ is
		\[ E(x_1\cdots x_n) = \frac{1}{n}\sum_{\sigma\in S_{n-1}} \frac{(-1)^{d_\sigma}}{\binom{n-1}{d_\sigma}} [x_{\sigma(1)},\cdots[x_{\sigma(n-1)},x_n]\cdots ]. \]
	\end{theorem}
	
	In the remainder of this introduction we shall sketch our answer to Problem B, how Theorem \ref{th:eulerian in Dynkin basis} follows from it, and the most significant results we establish along the way. \\

	\begin{center}
		{\em Step 1) From Lie words to planar, binary rooted trees}
	\end{center}
	
	A basic and familiar observation is that we can depict iterated brackets in $L_n$ as planar, binary, rooted trees with leaves labeled in the set $\{1,\ldots, n\}$, as illustrated in the following picture.
	\[ [[x_1,x_3],[[x_2,x_4],x_5]] \qquad \leftrightarrow \qquad 
	\xy {\ar@{-}(0,-5);(0,-2)};{\ar@{-}(0,-2);(-10,8)}; {(-10,10)*{\slab{1}}};
	{\ar@{-}(-7,5);(-4,8)};{(-4,10)*{\slab{3}}};
	{\ar@{-}(0,-2);(10,8) }; {(10,10)*{\slab{5}}};
	{\ar@{-}(4,2);(-2,8)}; {(-1.5,10)*{\slab{2}}};
	{\ar@{-}(1,5);(4,8)}; {(4,10)*{\slab{4}}};
	\endxy \]

	The Definition \ref{definition: basis} of the PBW basis $\basis_n$ has a clear graphical analog, cf.~\cite{MReut} or \cite[\S 13.2.5.2]{Loday-Vallette}. Given a leaf $l$ and an inner vertex $v$ of a given planar, binary, rooted tree $T$, we say that $l$ is a descendant of $v$ if $v$ lies in the unique directed path from $l$ to the root.
	
 We say that a labeling $\ell\colon\{\mbox{leaves of $T$}\}\to\{1,\ldots,n\}$ is \emph{admissible} if:
	\begin{itemize} \item the labeling, seen as a function from the set of leaves of $T$ to the set $\{1,\ldots,n\}$, is injective (in particular, $T$ has at most $n$ leaves); and
		\item for every inner vertex $v$ of $T$, the smallest (largest) label among all the leaves which are descendants of $v$ is located at the left-most (right-most) position.  
	\end{itemize}
	It follows directly from the definitions that there is a bijective correspondence between the set $\basis_n$ and the set $\pblT(n)$ of planar, binary, rooted trees with $n$ leaves and an admissible labeling $\ell\colon\{\mbox{leaves of $T$}\}\to\{1,\ldots,n\}$. In particular, given $(T,\ell)\in\pblT(n)$, we can define an associated Eulerian coefficient $E_{(T,\ell)}$. 
	
	An important observation -- see Corollary \ref{corollary: Eulerian} of Subsection \ref{subsection: from pbT to pT} -- is that after the passage 
	to trees, the Eulerian coefficient $E_{(T,\ell)}$ turns out to be independent of the labeling $\ell$. In order to prove this fact, we shall consider the vector space $\pbT$ spanned by
	all (un-labeled) planar, binary, rooted trees. In Subsection \ref{subsection: from pbT to pT}, we equip this space with a bilinear operation $\rhd\colon \pbT\otimes\pbT\to\pbT$, defined in terms of graftings of one tree onto another, and show that $(\pbT, \rhd)$ is a left pre-Lie algebra.
	
	Inside the pre-Lie algebra $(\pbT, \rhd)$, we can consider the pre-Lie logarithm $X:=-\log_\rhd\left(1-|\,\right)$, where $|$ is the (planar, binary, rooted) tree with only one leaf (and $1$ is a fictitious unit). By definition, the element $X$ is the unique solution to the equation
	\[ 1- e^{-X}_{\rhd}:= X - \frac{1}{2}X\rhd X + \frac{1}{6}X\rhd(X \rhd X) + \cdots+\frac{(-1)^{n+1}}{n!} \overbrace{X\rhd(\cdots(X\rhd X}^n)\cdots)+\cdots \,\stackrel{!}{=}\,|\,. \]
	We can expand $X$ with respect to the canonical basis of $\pbT$ given by planar, binary, rooted trees: we shall denote by $E_T$ the coefficient of the tree $T$ in the expansion of $X$, i.e., $X=\sum E_T\,T$. This is consistent with the previous notation: given a (planar, binary, rooted) tree $T$ with $|T|$ leaves and an admissible labeling $\ell\colon\{\mbox{leaves of $T$}\}\to \{1,\ldots, |T|\}$, we show that the identity $E_{(T,\ell)}=E_T$ holds.
	
	More precisely, we prove the following fact: We consider the map $\operatorname{lab}_n:\pbT\to \Lie_{\leq n}$ sending a tree $T$ to the sum $\sum_{\ell\, \mbox{\tiny is amissible}}(T,\ell)$ (and to zero if $T$ has more than $n$ leaves), where the sum runs over the admissible labelings (as defined above) of the leaves of $T$ by the set $\{1,\ldots,n\}$, and we regard the labeled tree $(T,\ell)$ as an iterated bracket inside $\Lie_{\leq n}$ in the usual way. We also denote by $\pi:\Lie_{\leq n}\to\Lie_n$ the obvious projection. The crucial Proposition \ref{proposition: Eulerian = log} from Subsection \ref{subsection: from pbT to pT} asserts that $\operatorname{lab}_n$ is a mophism of Lie algebras, and that $\pi\circ\operatorname{lab}_n(X)=E(x_1\cdots x_n)\in\Lie_n$. The independence of the Eulerian coefficient on the labeling follows a posteriori. In fact, we will prove stronger invariance properties for the Eulerian coefficients, which shall be explained at Step 3) below.
	\\
	
	\begin{center}
		{\em Step 2) Computing pre-Lie logarithms via umbral calculus.}
	\end{center}
	In the previous step we explained how to reduce the computation of $E(x_1\cdots x_n)$ with respect to the PBW basis, for all $n\ge 1$, to the computation of a pre-Lie logarithm in a certain pre-Lie algebra. In Section \ref{section: pre-Lie} (in particular, in Subsection \ref{subsection: log in sT}) we study the problem of how to compute pre-Lie logarithms in general. This problem is interesting in its own right, with connections to operad theory \cite{chapoton,pLDF}, numerical analysis \cite{Hairer,alg B-series} and renormalization \cite{Connes-Kreimer,Brouder}. Our approach is inspired by umbral calculus, in its modern formulation by G.-C. Rota and S. Roman \cite{Roman}.
	
	Given a pre-Lie algebra $(L,\rhd)$ and an element $y\in L$, the usual way to compute $x=\log_\rhd(1+y)$ is to observe that the equation
	\[ e_\rhd^x-1 :=  \sum_{n\geq1} \frac{1}{n!} \overbrace{x\rhd(\cdots(x\rhd x}^n)\cdots) = y\]
	is equivalent to
	\begin{equation}\label{rec} x = \sum_{n\geq0} \frac{B_n}{n!} \overbrace{x\rhd(\cdots(x}^n\rhd y)\cdots), \end{equation}
	where $B_n$ is the $n$-th Bernoulli number, and the latter can be solved recursively in $x$. On the other hand, this recursion becomes rapidly unwieldy. 
	
	We denote by $(L[t],\rhd)$ the pre-Lie algebra of polynomials with coefficient in $L$, with the pre-Lie product $\rhd$ induced by scalar extension. We also denote by $\left<\left. \frac{D}{e^D-1}\right| \# \right>\colon \K[t]\to\K$ the linear operator on polynomials defined by $\left<\left. \frac{D}{e^D-1}\right| t^n \right>=B_n$, as well as its extension to a linear operator $\left<\left. \frac{D}{e^D-1}\right| \# \right>\colon L[t]\to L$.
	
	Our idea is to consider the differential equation
	\begin{equation}\label{diffeq} \left\{ \begin{array}{l} P'(t) = \left<\left.\frac{D}{e^D-1}  \right| P(t)\right>\triangleright P(t),\\ P(0)= y,
	\end{array} \right. \end{equation}
	in the pre-Lie algebra $L[t]$, where $P'(t)$ denotes the ordinary derivative with respect to $t$. It is not hard to prove -- see Proposition \ref{prop:recvsdiffeq} -- that if $P(t)$ solves the above differential equation, then $x:=\left<\left. \frac{D}{e^D-1}\right| P(t) \right>\in L$ solves the recursion \eqref{rec}, hence $x=\log_\rhd(1+y)$. 
	
	\newcommand{\rT}{\mathcal{T}}
	
	We study the differential equation \eqref{diffeq} in two particular cases. In Section \ref{section: Eulerian coefficients} we shall consider the case where $L$ is the aforementioned pre-Lie algebra $(\pbT,\rhd)$ of planar, binary, rooted trees, and $y$ is the tree $|$ with only one leaf (in order to compute $-\log_\rhd(1-|)$, as in the previous step, in this case we shall replace the operator $\left<\left. \frac{D}{e^D-1}\right| \# \right>$ by the one $\left<\left. \frac{D}{1-e^{-D}}\right| \# \right>\colon t^n\mapsto(-1)^nB_n$). The second case is the universal one. As is well-known, the vector space $\rT$ spanned by (non-planar, not necessarily binary) rooted trees, together with a certain pre-Lie product $\curvearrowright\colon\rT\otimes\rT\to\rT$ defined in terms of graftings of one tree onto another, is the free pre-Lie algebra generated by $\bul$, the tree with just the root, cf.~\cite{Chapoton-Livernet}. In Section \ref{section: pre-Lie}, we study the differential equation \eqref{diffeq} when $L$ is the free pre-Lie algebra $(\rT,\curvearrowright)$ and $y$ is the generator $\bul$.
	
	In both cases we show how to solve \eqref{diffeq} recursively, see Theorem \ref{th: product rule} and Theorems \ref{th: product rule planar} and \ref{th:recursion planar}. The recursion turns out to be much more amenable to computations than the one implicit in equation \eqref{rec}. Moreover, by studying this recursion we are able to determine formulas for the coefficient of a given tree in the expansion of the pre-Lie logarithm (more precisely, in the expansion of $-\log_\rhd(1-|)$ in one case, and of $\log_\curvearrowright(1+\bul)$ in the other) in terms of purely combinatiorial data associated to the tree: see Theorem \ref{th: combinatorial}, Remark \ref{rem: solution for Eulerian} and Proposition \ref{prop:combinatorial eulerian}. \\

	\begin{center}
		{\em Step 3) From planar, binary, rooted trees to trees.}
	\end{center}
	There is a standard bijective correspondence between the set of planar, binary, rooted
	trees with $n$ leaves and the set of planar, (not necessarily binary) rooted trees with $n$ vertices, that goes under the name
	of \emph{Knuth's rotation correspondence}, see \cite{Knuth,Ebrahimi-Fard-Manchon}. We denote the vector space spanned by planar, (not necessarily binary) rooted trees by $\pT$. We consider a slight variation $\Phi\colon\pbT\to\pT$ of the rotation correspondence, which also involves the specular involution $\specular\colon\pbT\to\pbT$ sending a planar, binary, rooted tree to its mirror image (see Definition \ref{def: specular on trees} and the following remark). This allows us to observe some surprising invariance properties for the Eulerian coefficients. First of all, it follows directly from the explicit formula \eqref{equation: Eulerian} that:
	\begin{itemize} \item the Eulerian coefficients are invariant under the specular involution, that is, for every planar, binary, rooted tree $T$ we have $E_T=E_{\specular(T)}$; cf.~Lemma \ref{corollary: specularity in pbT}.
	\end{itemize}
	
	On the other hand, as explained in the previous step, we can compute the Eulerian coefficients by associating recursively a certain polynomial, which we denote by $\altbin{t}{T}\in\K[t]$, to every $T\in\pbT$, and then applying the linear operator $\left<\left. \frac{D}{1-e^{-D}}\right| \# \right>\colon t^n\mapsto(-1)^nB_n$ to this polynomial. Using the correspondence $\Phi$, we can equivalently consider both the polynomial and the coefficient as being associated to the corresponding planar rooted tree $\Phi(T)$.
	When we do this, the recursion for the polynomial becomes more transparent, and in particular it implies that
	
	\begin{itemize}
		\item The polynomial -- hence, also the Eulerian coefficient -- associated to a planar roooted tree is independent of the planar structure, that is, it only depends on the underlying (non-planar) rooted tree; cf.~Lemma \ref{lemma: Euler planar structure}.
	\end{itemize}
	
	Finally, putting these two facts together, we obtain the following surprising result.
	\begin{itemize}
		\item The Eulerian coefficient associated to a planar rooted tree is furthermore independent of the location of the root, that is, it only depends on the underlying tree (in the sense of graph theory, i.e., a connected graph with no cycles); cf.~Corollary \ref{cor: independence on root}.
	\end{itemize}
	
	We stress that the number of trees (in the sense of graph theory)
	with precisely $n$ vertices (sequence A000055 in the OEIS\footnote{Available at the following link: ~\url{https://oeis.org/A000055}.}) is much smaller than the cardinality of the PBW basis $\basis_n$, which is $(n-1)!$. A table of Eulerian coefficients for trees with $n\leq 8$ vertices can be found in Appendix \ref{appendix: table}. We also remark that the polynomial \emph{is not} independent of the choice of root: in particular, the last invariance result also implies a plethora of identities involving Bernoulli numbers, the simplest example being the classical Euler's identity, cf.~Remark \ref{rem:bernoulli identities} for more details.
	\\
	
	\begin{center}
		{\em Step 4) From the PBW basis to Dynkin's basis}
	\end{center}
	We conclude this long introduction by sketching how Theorem \ref{th:eulerian in Dynkin basis} follows from the previous results. 
	
	To an  element $b\in\basis_n$ in the PBW basis, or equivalently, to the corresponding planar binary rooted tree $T$ with $n$ leaves and an admissible labeling $\ell\colon\{\mbox{leaves of $T$}\}\to\{1,\ldots,n\}$, we associate a subset of the symmetric group $S(b)=S(T,\ell)\subset S_{n-1}$. This may be defined as the set of permutations $\sigma$ such that $b$ appears, with a non-zero coefficient, in the expansion of $[x_{\sigma(1)},\cdots[x_{\sigma(n-1)},x_n]\cdots]\in\mathcal{D}_n$ in the PBW basis $\mathcal{B}_n$. More precisely, given $[x_{\sigma(1)},\cdots[x_{\sigma(n-1)},x_n]\cdots]\in\mathcal{D}_n$, and writing its expansion in the PBW basis \[ [x_{\sigma(1)},\cdots[x_{\sigma(n-1)},x_n]\cdots]=\sum_{b\in\basis_n}c_{\sigma,b}\,b,\] we have (cf.~Proposition \ref{prop: from dynkin to pbw}): 
	\begin{itemize} \item $c_{\sigma,b}=0$ if $\sigma\not\in S(b)$;
		\item $c_{\sigma,b}=(-1)^{r_T-1}$ if $\sigma\in S(b)=S(T,\ell)$, where $(T,\ell)$ is the labeled tree corresponding to $b$, and $r_T$ is the number of right pointing leaves of $T$.\end{itemize}
	On the other hand, we will see that the set $S(T,\ell)$ can be described in terms of the combinatorics of the labeled tree $(T,\ell)$ alone, see Remark \ref{rem:removable}. We use the above observation to switch between the PBW basis and the Dynkin's basis.
	
	We define the \emph{Eulerian numbers associated to $(T,\ell)$}, and we denote them by $E(T,\ell,d)$, where $0\leq d\leq n-2$, as the number of permutations in $S(T,\ell)$ having descent number $d$, that is, 
	\[ E(T,\ell,d) =\left|\left\{ \sigma\in S(T,\ell)\,\vert\, d_\sigma=d   \right\}\right|.  \]
	We recover the usual Eulerian numbers (cf.~\cite{Petersen}) when $T$ is a left pointing comb (equivalently, when $\Phi(T)$ is a corolla). Using Stanley's shuffling Theorem \cite{Stanley, Goulden}, we prove in Proposition \ref{prop: eulerian independet of labeling} that these numbers are independent of the labeling $\ell$, which accordingly will be dropped from the notations. 
	
	In Theorem \ref{th:worpitzki} we prove a generalization of the classical Worpitzki's identity (cf.~\cite[\S 1.5]{Petersen}), relating the numbers $E(T,d)$ and the polynomial $\altbin{t}{T}$ from the previous step. More precisely, we will prove the following identity\footnote{It might appear that this identity already implies the independence of the numbers $E(T,\ell,d)$ from the labeling, but in fact we use this result in the proof of \eqref{worpitzki-intro}. On the other hand, once we have proved both facts, \eqref{worpitzki-intro} implies the stronger result that the numbers $E(T,d)$ only depend on $\Phi(T)$ as a (non-planar) rooted tree.}, where  $|T|$ is the number of leaves of $T$,
	\begin{equation}\label{worpitzki-intro} (-1)^{r_T-1}\altbin{k+1}{T} = \sum_{d=0}^{|T|-2}\binom{|T|-1+k-d}{|T|-1}E(T,d).   \end{equation}
	Once again, this recovers the usual Worpitzki's identity when $T$ is a left pointing comb. In fact, our argument follows a bijective proof of the latter, that we found in Knuth's book \cite{Knuth}.
	
We apply the identity \eqref{worpitzki-intro} to deduce the following formula, relating the Eulerian numbers $E(T,d)$ and the Eulerian coefficient $E_T$ associated to $T$:
\[ E_T=\frac{(-1)^{r_T-1}}{|T|}\sum_{d=0}^{|T|-2}\frac{(-1)^d}{\binom{|T|-1}{d}}E(T,d). \]

Finally, after these preparations, the proof of Theorem \ref{th:eulerian in Dynkin basis} becomes straightforward: it follows by combining the above formula for $E_T$ and the aforementioned rule to switch between the bases $\mathcal{D}_n$ and $\basis_n$, see Subsection \ref{subsection: eulerian in dynkin}. Of course, a more direct proof than the one given here should be possible: nonetheless, we believe that the various results we establish along the way, and in particular the previous identity \eqref{worpitzki-intro}, are of independent interest.

	\begin{acknowledgements}
		We are grateful to the Erwin Schr\"odinger Institute in Vienna (Austria) for the excellent working conditions during our stay there in July and August 2016, as well as for the financial support that we received through its ``Research in Teams" program. 
		Moreover, we thank the authors of the Wikipedia page on
		Bernoulli numbers \url{https://en.wikipedia.org/wiki/Bernoulli_number} -- it was of great help to us at certain stages of this project.
	\end{acknowledgements}

	\section{{\bf Computing pre-Lie logarithms via umbral calculus}}\label{section: pre-Lie}
	
	In this section we develop tools to compute logarithms in a pre-Lie algebra: these will be used in Section \ref{section: Eulerian coefficients} to solve our main problem outlined in the Introduction, that is, find formulas for the Eulerian idempotent. On the other hand, the computation of  pre-Lie logarithms is an interesting problem in its own right, with connection to operad theory \cite{chapoton,pLDF}, numerical analysis \cite{Hairer,alg B-series} and renormalization \cite{Connes-Kreimer,Brouder}. Our approach is inspired by umbral calculus, in its modern formulation by G.-C. Rota and S. Roman \cite{Roman}. 
	
	As always, we work over a field $\K$ of characteristic zero.

	\subsection{Basics on pre-Lie algebras}\label{subsection: pre-Lie}

	\begin{definition}\label{definition: pre-Lie}
		A (left) pre-Lie algebra is a vector space $L$, equipped with a bilinear operation
		$$ \pre: L\otimes L \to L,$$
		such that its associator $A_\pre(x,y,z):= x \pre (y \pre z) - (x \pre y) \pre z$
		is symmetric in the first two arguments. When this happens, the corresponding commutator $[x,y]:= x\rhd y - y\rhd x$ is a Lie bracket on $L$ (hence the name pre-Lie algebras).
	\end{definition}
	
	\begin{example}
		Any associative algebra $(A,\cdot)$ is in particular a pre-Lie algebra, as in this case the corresponding associator is identically zero.
		In particular, both the polynomial ring $\K[t]$ and the ring of formal power
		series $\K[[t]]$ come equipped with the structure of a pre-Lie algebra.
	\end{example}

	\begin{definition}\label{def: sym braces} Given a left pre-Lie algebra $(L,\pre)$, the associated symmetric brace operations $\{\#,\ldots,\#|\# \}:L^{\odot n}\otimes L\to L$ are defined recursively by (see \cite{oudom-guin,pLDF})
		\[ \{ x \} = x, \]
		\[ \{ y|x \} = y\pre x, \]
		\[  \{y_1,\ldots,y_k|x\} = y_1\pre\{y_2,\ldots,y_{k} | x \}-\sum_{j=2}^{k}\{y_2,\ldots,y_1\pre y_j,\ldots y_{k}|x\} . \]\end{definition}
	
	For instance, $\{ y_1,y_2| x\} = y_1\pre (y_2\pre x)- (y_1\pre y_2) \pre x$ is the usual associator. Starting from this observation, it is not hard to prove inductively that for $k\ge 2$ the brace $\{ y_1,\ldots,y_k| x  \}$ is symmetric in the arguments $y_1,\ldots,y_k$. It is well-known that the operations $\{\#,\ldots,\#|\# \}$ make $L$ into a symmetric brace algebra. In fact, this construction establishes an isomorphism between the categories of (right) symmetric brace  algebras and (left) pre-Lie algebras  (cf.~\cite{oudom-guin} and \cite[\S 13.4.9]{Loday-Vallette} for further references on symmetric brace algebras).
	
	\begin{definition}
		A complete  pre-Lie algebra is a pre-Lie algebra $(L,\rhd)$ equipped with a filtration 
		$$ \cdots \subset F^pL \subset F^{p-1}L \subset \cdots F^2L \subset F^1L = L$$
		such that 
		
		\begin{itemize} \item the filtration is complete, that is, the natural morphism of vector spaces $L\to\underleftarrow{\lim}\, L/F^pL$ is an isomorphism; and \item the filtration is compatible with the pre-Lie product, that is, $F^kL \rhd F^l L\subset F^{k+l}L$ for all $k,l\ge 1$.
		\end{itemize}
	\end{definition}

	We next turn to a special example of a complete pre-Lie algebra, defined in terms of trees.
	
	\begin{remark}\label{remark: tree terminilogy}
		Let us fix some basic terminology concerning trees.
		A rooted tree is a tree with a distinguished vertex, the root.
		Given any vertex $v$ of a rooted tree $T$,
		there is a unique shortest path from $v$ to the root. The distance of $v$ from the root is the length of this path, i.e., the number of edges in it.
		The set $V(T)$ of vertices of $T$ inherits a partial order
		by declaring $v\ge v'$ whenever $v$ lies on the shortest path from $v'$ to the root.
		In this situation, we also call $v'$ a descendant of $v$.
		The height $h(T)$ of a tree is the maximal distance
		of a vertex from the root.
		The order of $T$, denoted by $|T|$, is the number of vertices (including the root). 
	\end{remark}
	
	\begin{definition}
		Let $\T(n)$ be the vector space spanned by all rooted trees with $n$ vertices (including the root).
		We denote by $\T$ the direct product $\T := \prod_{k\ge 1}\T(k)$.
		\begin{itemize}
			
			\item We define a complete filtration on $\T$ by setting $F^p\T := \prod_{k\ge p}\T(k)$.
			
			\item 	We define a bilinear operation $\curvearrowright$ on $\T$ by 
			\[ T \curvearrowright T' := \sum_{v\in V(T')} T \searrow_v T',\]
			where $T \searrow_v T'$ is the rooted tree obtained by taking the disjoint union of $T$ and $T'$, drawing an edge from the root of $T$ to $v$, and finally taking the root of $T'$ as the new root.
		\end{itemize}
	\end{definition}
	
	\begin{example}
		An example of $\curvearrowright$ is shown below:
		
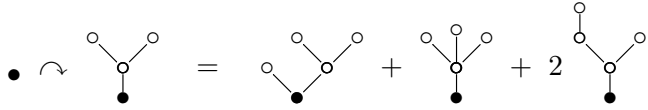
\begin{eqnarray*}
			\xy
			{(0,0)*{\bullet}};
			\endxy \, \, \curvearrowright \, \,
			\xy
			{\ar@{-}(4,-4)*{\bullet};(4,0)*{\circ}};
			{\ar@{-}(4,0)*{\circ};(0,4)*{\circ}};
			{\ar@{-}(4,0)*{\circ};(8,4)*{\circ}};
			\endxy  \,\, 
			&=&
			\,\,\, 	\xy
			{\ar@{-}(0,-4)*{\bullet};(-4,0)*{\circ}};
			{\ar@{-}(0,-4)*{\bullet};(4,0)*{\circ}};
			{\ar@{-}(4,0)*{\circ};(8,4)*{\circ}};
			{\ar@{-}(4,0)*{\circ};(0,4)*{\circ}};
			\endxy
			\,\, +  \,\,
			\xy 
			{\ar@{-}(4,-4)*{\bullet};(4,0)*{\circ}};
			{\ar@{-}(4,0)*{\circ};(0,4)*{\circ}};
			{\ar@{-}(4,0)*{\circ};(8,4)*{\circ}};
			{\ar@{-}(4,0)*{\circ};(4,5)*{\circ}};
			\endxy \,\, + \,\, 2 \,\, \xy
			{\ar@{-}(4,-4)*{\bullet};(4,0)*{\circ}};
			{\ar@{-}(4,0)*{\circ};(0,4)*{\circ}};
			{\ar@{-}(4,0)*{\circ};(8,4)*{\circ}};
			{\ar@{-}(0,4)*{\circ};(0,8)*{\circ}};
			\endxy
		\end{eqnarray*}
	\end{example}

	\begin{remark}\label{rem: sym braces in free pl}
		It is well-known that $(\T,\curvearrowright)$ is a complete pre-Lie algebra. In fact, as shown by Chapoton and Livernet \cite{Chapoton-Livernet},  it is the free complete pre-Lie algebra
		generated by $\bullet$, the tree with only the root.
		It therefore satisfies the following universal property: given a complete pre-Lie algebra $(L,\pre)$ and an element $x\in L$, there is a unique morphism of pre-Lie algebras $\Psi_x:\T\to L$ such that $\Psi_x(\bul)=x$. This morphism is more easily described in terms of the symmetric brace operations. First of all, it is easy to show by induction that the symmetric brace operations $\{\#,\ldots,\#|\bul \}:\T^{\odot n}\to\T$ are as follows: given trees $T_1,\ldots,T_k \in \T$, the tree $\{T_1,\ldots,T_k|\bul \}$ is obtained by taking the disjoint union $\bullet,T_1,\ldots,T_k$, drawing an edge from $\bullet$ to the root of each one of the trees $T_1,\ldots,T_k$, and finally making $\bullet$ into the root of this new tree. For instance, $\{\bullet|\bullet\}= \xy 
		{\ar@{-}(0,-2)*{\bullet};(0,2)*{\circ}}\endxy$ is the tree with only the root and one leaf, and, more in general, the corolla with $n$ leaves can be written as follows:
		$$ \overbrace
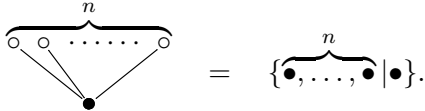
{\xy 
			{\ar@{-}(0,-4)*{\bul};(-10,4)*{\circ}};
			{\ar@{-}(0,-4)*{\bul};(-6,4)*{\circ}};
			{\ar@{}(0,-4)*{\bul};(2,4)*{\cdots\cdots}};
			{\ar@{-}(0,-4)*{\bul};(10,4)*{\circ}}; 
			\endxy}^n
		\quad = \quad
		\{\overbrace{\bullet,\ldots,\bullet}^n | \bullet\}.$$
		Another example is depicted below:
		
		\begin{eqnarray*}
			\{ \{ \bul,\bul,\bul|\bul  \}, \{ \{\bul,\bul|\bul \} |\bul  \} |\bul  \}\: =\:	\left\{ \left.
			
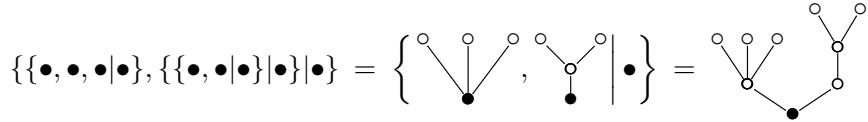
\xy
			{\ar@{-}(0,-4)*{\bul};(0,4)*{\circ}};
			{\ar@{-}(0,-4)*{\bul};(-6,4)*{\circ}};
			{\ar@{-}(0,-4)*{\bul};(6,4)*{\circ}};
			\endxy , 
			\xy
			{\ar@{-}(0,-4)*{\bul};(0,0)*{\circ}};
			{\ar@{-}(0,0)*{\circ};(4,4)*{\circ}};
			{\ar@{-}(0,0)*{\circ};(-4,4)*{\circ}};
			\endxy
			\right| \bul \right\} \: =\:
			\xy
			{\ar@{-}(0,-6)*{\bul};(-6,-2)*{\circ}};
			{\ar@{-}(0,-6)*{\bul};(6,-2)*{\circ}};
			{\ar@{-}(-6,-2)*{\circ};(-6,4)*{\circ}};
			{\ar@{-}(-6,-2)*{\circ};(-10,4)*{\circ}};
			{\ar@{-}(-6,-2)*{\circ};(-2,4)*{\circ}};
			{\ar@{-}(6,-2)*{\circ};(6,3)*{\circ}};
			{\ar@{-}(6,3)*{\circ};(3,8)*{\circ}};
			{\ar@{-}(6,3)*{\circ};(9,8)*{\circ}};
			\endxy
		\end{eqnarray*}
		In particular, it is clear how to generate any tree from $\bul$ via nested symmetric braces. Since any morphism of pre-Lie algebras is automatically compatible with the associated braces, this describes $\Psi_x$ completely. For instance, for the tree $T$ depicted above we get \[ \Psi_x(T)=\{ \{ x,x,x|x  \}, \{ \{x,x|x\} |x \} |x  \}.\]
	\end{remark}

	\subsection{Umbral calculus in pre-Lie algebras}\label{subsection: Umbral calculus}

	We call a formal power series \begin{equation}\label{eq:delta series} \K[[t]]\ni f(t) = \sum_{k\geq0} \frac{c_k}{k!} t^k \end{equation} 
	a {\em $\delta$-series} if $c_0=0$ and $c_1\neq0$. Given a $\delta$-series $f(t)$, the formal power series $\frac{f(t)}{t}$ has a multliplicative inverse, which we shall denote by 
	\[ g(t):= \frac{t}{f(t)} = \sum_{k\geq0} \frac{a_k}{k!} t^k. \]
	Finally, we shall denote by $a_0(t), a_1(t), \ldots, a_k(t),\ldots$ the {\em Appell sequence} of polynomials associated to the sequence of scalars $a_0,a_1,\ldots,a_k,\ldots \in \K$: these are defined recursively by
	\[ a_0(t)=a_0 \quad \textrm{and} \quad \int_0^t a_k(\tau) d\tau = \frac{a_{k+1}(t)-a_{k+1}}{n+1}, \]
	and explicitly by
	\[ a_k(t) = \sum_{j=0}^k\binom{k}{j}a_{k-j}t^j.  \]
	For instance, the first few polynomials are
	\[ a_0(t) = a_0, \]
	\[ a_1(t) = a_0t+a_1, \]
	\[ a_2(t) = a_0 t^2 +2a_1 t + a_2, \]
	\[ a_3(t) = a_0 t^3 + 3a_1t^2 + 3a_2 t +a_3,  \]
	\[ \cdots \]

	Let $(L,\pre)$ be a complete, left pre-Lie algebra. Given $x\in L$, we shall denote by 
	\[ x^{\pre k} := \overbrace{x\pre (\cdots( x\pre x}^k)\cdots).  \]
	Given a $\delta$-series $f(t)\in \K[[t]]$ as in \eqref{eq:delta series}, we can consider the function
	\[ f_{\pre}(-): L \to L : x\to f_{\pre}(x):= \sum_{k\geq1} \frac{c_k}{k!}\,x^{\pre k}.  \]
	The above infinite summation (and the following ones) makes sense since $L$ is complete. This function admits a compositional inverse $f_\rhd^{-1}(-):L\to L$. In fact, given $x,y\in L$, and denoting by $y \pre \#: L \to L : z\to y\pre z$ the operator of left multiplication by $y$, the equation 
	\[ x = f_\pre (y) = \sum_{k\ge1}\frac{c_k}{k!} y^{\pre k} = \sum_{k\geq0} \frac{c_{k+1}}{(k+1)!}(y\pre \#)^k(y)  \]
	is equivalent to 
	\begin{equation}\label{eq:recursion} 
	y =  \sum_{k\geq0} \frac{a_k}{k!}(y\pre \#)^k(x),
	\end{equation}
	and the latter can be solved recursively in $y$.
	
	\begin{example}
		One can compute $y$ up to order three as follows, where $o(n)$ denotes a remainder term $o(n)\in F^n L$. 
		\begin{eqnarray*}
			y &=& a_0\, x + o(2)= \\
			&=& a_0\, x + a_1\, y\pre x + o(3) = a_0\,x + a_1( a_0\,x + o(2))\pre x + o(3) = a_0\,x +a_0a_1\,x\pre x + o(3) =\\
			&=& a_0\, x + a_1\, y\pre x + \frac{a_2}{2} y\pre(y\pre x) +o(4)= \\&=& a_0\,x + a_1\,( a_0\,x + a_0a_1\,x\pre x)\pre x + \frac{a_2}{2} a_0\, x\pre ( a_0\,x\pre x) + o(4) = \\ &=& a_0\,x + a_0a_1\, x\pre x + \frac{a_0^2 a_2}{2}x\pre (x\pre x) + a_0 a_1^2\,(x\pre x)\pre x + o(4).
		\end{eqnarray*}
		It is clear that the above process can be iterated up to any order, and since $L$ is complete, the resulting infinite series converges to a well defined $y=: f_\pre^{-1}(x)\in L$.
	\end{example}
	
	The previous recursion \eqref{eq:recursion} for $f_\pre^{-1}(-):L\to L$ becomes rapidly unwieldy, and the aim of this section is to develop techniques to compute the function $f_\pre^{-1}(-)$ more efficiently.
	In the remainder of this section, we will focus on the universal case, i.e. we shall assume that $(L,\pre)$ is the free complete pre-Lie algebra $(\mathcal{T},\curvearrowright)$ of rooted trees described in the previous section.

	We denote by $(\T[t],\curvearrowright)$ the complete pre-Lie algebra $\T[t]:=\K[t]\otimes \T$ of polynomials with coefficients in $\T$, with the obvious pre-Lie product $\curvearrowright$ induced via scalar extension. We shall denote an element $P\in\T[t]$ by 
	\[ P = \sum_{T} P(T)(t) \otimes T, \] 
	where the sum runs over the set of all rooted trees, and by $P', P'',\ldots, P^{(n)},\ldots$ the ordinary derivatives $P^{(n)} = \sum_{T}\frac{d^nP(T)}{dt^n}(t) \otimes T $ with respect to the variable $t$. Furthermore, given the formal power series $g(t) = \sum_{k\ge 0}\frac{a_k}{k!} t^k \in \K[[t]]$ as before, we shall denote by $\langle g(D)|-\rangle$ the linear functional $\K[t]\to\K$ sending $t^n$ to $a_n$, and by the same symbol its extension to a functional $\langle g(D)|-\rangle:\T[t]\to\T$. This notation is inspired by umbral calculus, cf.~Roman's book \cite{Roman}. 
	
	In order to compute $f^{-1}_\curvearrowright(\bul)$, our main idea is to replace the recursion  \eqref{eq:recursion} by the following differential equation
	\begin{equation}\label{eq:differential equation}
	\left\{\begin{array}{l} P ' = \langle g(D)|P\rangle\curvearrowright P \\
	P(0) = \bullet
	\end{array}  \right.
	\end{equation}
	in the pre-Lie algebra $\T[t]$. 
	
	\begin{remark}\label{rem: notation}
		For the remainder of this subsection, we shall denote by \[ P=\sum_T P(T)(t)\otimes T \]  
		the solution of $\eqref{eq:differential equation}$. To get nicer formulas later on, it will be convenient to introduce the normalized polynomials
		\[  p(T)(t) := |\operatorname{Aut}(T)|\cdot P(T)(t),\qquad P=\sum_T \frac{p(T)(t)}{|\Aut(T)|}\otimes T \]
		where $|\operatorname{Aut}(T)|$ is the number (also called the \emph{symmetry factor} of $T$) of automorphisms of $T$ as a rooted tree. 
		
		We shall also denote by $a:=\left< g(D)|P \right>\in\T$, and respectively by $a_T, \widetilde{a}_T\in\K$ the normalized and unnormalized coefficient of a tree $T$ in the expansion of $a$, that is,  
		\begin{eqnarray*} &\widetilde{a}_T := \left< g(D)|P(T)(t)  \right>,\quad a_T :=\left< g(D) | p(T)(t) \right> =|\Aut(T)|\cdot\widetilde{a}_T,&\\
			&a:= \langle g(D)|P\rangle = \sum_T\widetilde{a}_T\,T =\sum_{T} \frac{a_T}{|\operatorname{Aut}(T)|}\, T.& \end{eqnarray*}
	\end{remark}
	The link between \eqref{eq:recursion} and \eqref{eq:differential equation} is explained by the following proposition.
	\begin{proposition}\label{prop:recvsdiffeq}
		If $P\in \mathcal{T}[t]$ is the solution to the equation \eqref{eq:differential equation}, then $a:=\langle g(D) | P\rangle$ is the solution of the recursion \eqref{eq:recursion}, hence $ a = f^{-1}_{\curvearrowright}(\bul)$.
	\end{proposition}

	\begin{proof} Equation \eqref{eq:differential equation} implies by induction that
		\[ P^{(k)}(0)= \overbrace{a\curvearrowright (\cdots( a}^k\curvearrowright \bul)\cdots),\]
		and therefore the Taylor series of $P$ reads
		\[ P = \sum_{k\geq0} \frac{t^k}{k!} \overbrace{a\curvearrowright (\cdots( a}^k\curvearrowright \bul)\cdots) . \]
		Since by definition $a=\langle g(D)| P\rangle $, and $\langle g(D)| \#\rangle$ is the functional sending $t^k$ to $a_k$, applying $\langle g(D)| \#\rangle$ to both sides yields
		\[ a = \sum_{k\geq0} \frac{a_k}{k!} \overbrace{a\curvearrowright (\cdots( a}^k\curvearrowright \bul)\cdots),  \]
		which is precisely saying that $a$ solves the recursion \eqref{eq:recursion}
		$$ a = \sum_{k\ge 0}\frac{a_k}{k!}(a \curvearrowright \#)^k(\bullet).$$ 
	\end{proof}

	\begin{remark}\label{remark:recucomp}	We notice that equation \eqref{eq:differential equation} could be solved recursively, in a way similar to what we did for \eqref{eq:recursion}. For instance, we can determine $P$ up to order three as follows: first of all, $P' = o(2)$, together with the initial condition $P(0)=\bullet$, implies that $p(\bullet)=1$, $a_\bullet =\left<g(D)|p(\bullet)\right>= a_0$. Next, we have
		\[ P' = a\curvearrowright P = (a_0\bullet + o(2))\curvearrowright (\bullet + \,o(2) ) = a_0 \xy
		{\ar@{-}(0,-2)*{\bul};(0,2)*{\circ}};
		\endxy +o(3), \]
		thus $p(\xy
		{\ar@{-}(0,-2)*{\bul};(0,2)*{\circ}};
		\endxy) =a_0t$ and $a_{\xy
			{\ar@{-}(0,-2)*{\bul};(0,2)*{\circ}};
			\endxy}= a_0a_1$. Continuing like this, we find
		
		\begin{eqnarray*}
			P' &=& (a_0\bullet + a_0a_1\xy
			{\ar@{-}(0,-2)*{\bul};(0,2)*{\circ}};
			\endxy +o(3))\curvearrowright(\bullet + a_0 t\xy
			{\ar@{-}(0,-2)*{\bul};(0,2)*{\circ}};
			\endxy +o(3)) = \\ &=& a_0 \xy
			{\ar@{-}(0,-2)*{\bul};(0,2)*{\circ}};
			\endxy + a_0^2 t \xy
			{\ar@{-}(0,-1)*{\bul};(-2,2)*{\circ}};
			{\ar@{-}(0,-1)*{\bul};(2,2)*{\circ}};
			\endxy + a_0(a_0t +a_1)\xy
			{\ar@{-}(-3,0)*{\bul};(0,0)*{\circ}};
			{\ar@{-}(0,0)*{\circ};(3,0)*{\circ}};
			\endxy + o(4),   
		\end{eqnarray*} 
		thus \[ p(\xy
		{\ar@{-}(0,-1)*{\bul};(-2,2)*{\circ}};
		{\ar@{-}(0,-1)*{\bul};(2,2)*{\circ}};
		\endxy) = 2P(\xy
		{\ar@{-}(0,-1)*{\bul};(-2,2)*{\circ}};
		{\ar@{-}(0,-1)*{\bul};(2,2)*{\circ}};
		\endxy) = a_0^2 t^2,\qquad a_{\xy
			{\ar@{-}(0,-1)*{\bul};(-2,2)*{\circ}};
			{\ar@{-}(0,-1)*{\bul};(2,2)*{\circ}};
			\endxy }= a_0^2a_2,\] \[p(\xy
		{\ar@{-}(-3,0)*{\bul};(0,0)*{\circ}};
		{\ar@{-}(0,0)*{\circ};(3,0)*{\circ}};
		\endxy)= P(\xy
		{\ar@{-}(-3,0)*{\bul};(0,0)*{\circ}};
		{\ar@{-}(0,0)*{\circ};(3,0)*{\circ}};
		\endxy) =\frac{a_0^2}{2}t^2 + a_0a_1 t,\qquad a_{\xy
			{\ar@{-}(-3,0)*{\bul};(0,0)*{\circ}};
			{\ar@{-}(0,0)*{\circ};(3,0)*{\circ}};
			\endxy}= \frac{1}{2}a_0^2a_2 + a_0a_1^2.\]
	\end{remark}
	
	On the other hand, there is a simpler recursive scheme to compute the polynomials $p(T)(t)$, which is why we prefer to work with equation \eqref{eq:differential equation}. To state the result, we need an additional piece of notation: given the formal power series $g(t)=\sum_{k\geq0}\frac{a_k}{k!} t^k$ as before, we shall denote by $g(D):\K[t]\to\K[t]$ the associated operator on polynomials, obtained by formally replacing $t$ by the derivative operator $D=\frac{d}{dt}$. Equivalently, this is the linear operator sending $t^n$ to the $n$-th polynomial $a_n(t)$ in the Appell sequence associated to $a_0,\ldots,a_k,\ldots$ (cf.~the beginning of the subsection, as well as Roman's book \cite[\S 2.2]{Roman})
	$$g(D):= \sum_{k\geq0}\frac{a_k}{k!} D^k:\K[t]\to\K[t]:p \mapsto g(D)p,\quad g(D)t^n= a_n(t).$$

	\begin{theorem}\label{th: product rule}\hspace{0cm} The polynomials $p(T)(t)\in\mathbb{K}[t]$, $T\in\T$, are determined by $p(\bullet)(t)=1$ and the following recursion. 
		\begin{itemize}
			\item[(I)] Given a tree $T=\{T_1,\ldots,T_k|\bul \}$, $k\geq 1$ (see Remark \ref{rem: sym braces in free pl} for the notation), we have
			\begin{equation}\label{eq: product rule} p(T) = p(\{T_1,\ldots,T_k|\bul \})=p(\{T_1|\bullet\})\cdots p(\{T_k|\bullet \}). \end{equation}
			\item[(II)] For every tree $T \in \mathcal{T}$, we have
			\begin{equation}\label{eq: umbral substitution} p(\{T|\bul\})(t) = \int_0^t g(D)p(T)(\tau) d\tau.  \end{equation}
		\end{itemize}
	\end{theorem}
	\begin{remark}\label{remark: diff operator} Given the initial $\delta$-series $f(t)=\sum_{k\geq1} \frac{c_k}{k!} t^k$, we denote by $f(D)$ the associated operator $f(D):=\sum_{k\geq1} \frac{c_k}{k!} D^k:\K[t]\to \K[t]$. 
		Recall that $g(t):= \frac{t}{f(t)}$.
		Then item (II) in the previous theorem can be stated equivalently as follows:
		\begin{align*} p(\{T|\bul\})(t) = \int_0^t g(D)p(T)(\tau) d\tau & \Longleftrightarrow\quad  Dp(\{T|\bul \})=g(D)p(T),\quad p(\{T|\bul \})(0)=0\\
		&\Longleftrightarrow\quad \frac{D}{g(D)} p(\{T|\bul \}) = p(T),\quad p(\{T|\bul \})(0)=0 \\
		& \Longleftrightarrow\quad f(D) p(\{T|\bul \}) = p(T),\quad p(\{T|\bul \})(0) =0.  
		\end{align*} 
		In particular, the latter equivalence and \cite[Theorem 2.4.5]{Roman}, together with a straightforward induction, imply that the polynomial associated to the tall tree		
		\[ T_{n+1} =  
		\left.  \xy
		{\ar@{-}(0,-8)*{\bul};(0,-3)*{\circ}};
		{\ar@{.}(0,-3)*{\circ};(0,2)*{\circ}};
		{\ar@{-}(0,2)*{\circ};(0,7)*{\circ}};
		\endxy \right\}\mbox{\scriptsize $n+1$}
		\qquad\mbox{is}\qquad P(T_{n+1})=\frac{p_n(t)}{n!}, \] 
		where $p_n(t)$ is the $n$-th polynomial in the sequence of binomial type associated to the $\delta$-series $f(t)$, cf.~\cite[\S 2.4]{Roman}.
		
	\end{remark}

	Our proof of Theorem \ref{th: product rule} relies on the following recursion \eqref{eq:coproduct recursion} for the (unnormalized, cf.~Remark \ref{rem: notation}) polynomials $P(T)$. In the statement of the following lemma, we consider the coproduct $\Delta:\mathcal{T}\to \mathcal{T}\otimes\mathcal{T}$ which is the transpose of $\curvearrowright$ with respect to the canonical basis given by trees, for instance 
	
	\begin{eqnarray*}
		\Delta\left(
		\xy
		{\ar@{-}(0,-4)*{\bul};(-4,1)*{\circ}};
		{\ar@{-}(0,-4)*{\bul};(4,1)*{\circ}};
		{\ar@{-}(4,1)*{\circ};(4,6)*{\circ}};
		\endxy \right)
		\,\, = \,\,
		2 \,\, 
		\xy
		{\ar@{-}(0,0)*{\bul};(0,0)*{\bul}};
		\endxy
		\, \otimes \,
		\xy 
		{\ar@{-}(0,-2)*{\bul};(-4,3)*{\circ}};
		{\ar@{-}(0,-2)*{\bul};(4,3)*{\circ}};
		\endxy
		\,\, + \,\,
		\xy
		{\ar@{-}(0,-2)*{\bul};(0,3)*{\circ}}; \endxy \, \otimes \,
		\xy
		{\ar@{-}(0,-2)*{\bul};(0,3)*{\circ}};
		\endxy
		\,\,+ \,\,
		\xy
		{\ar@{-}(0,0)*{\bul};(0,0)*{\bul}};
		\endxy
		\, \otimes \,
		\xy
		{\ar@{-}(0,-4)*{\bul};(0,1)*{\circ}};
		{\ar@{-}(0,1)*{\circ};(0,6)*{\circ}};
		\endxy 
	\end{eqnarray*}
	We shall write in Sweedler's notation
	\[\Delta(T)= \sum T^{(1)}\otimes T^{(2)}.  \]
	\begin{remark} As the above example shows, there might be some integer coefficients in the expansion of $\Delta(T)$, which we hide inside the term $T^{(1)}$ when using Sweedler's notation. In particular, when we write $\widetilde{a}_{T^{(1)}}$ as in the following formula \eqref{eq:coproduct recursion}, we are also taking into account these coefficients. \end{remark}
	\begin{lemma}\label{prop:coproduct recursion} Given a tree $T\neq \bul$, and writing $\Delta(T)$ in Sweedler's notation as above, we have (cf.~Remark \ref{rem: notation} for the definition of $\widetilde{a}_{T^{(1)}}$)                                                          
		\begin{equation}\label{eq:coproduct recursion} P(T)(t) = \sum \widetilde{a}_{T^{(1)}}\int_0^t P(T^{(2)})(\tau)d\tau.  \end{equation}
		\begin{proof} Straightforward consequence of the differential equation \eqref{eq:differential equation}.
		\end{proof} \end{lemma}
		
		\begin{proof}[Proof of Theorem \ref{th: product rule}]
			
			Part (I) of the theorem can be restated, in terms of the unnormalized polynomials $\{P(T)\}_{T\in\T}$, as follows. For any set of distinct trees $T_1,\ldots, T_k$ and positive integers $i_1,\ldots,i_k$, we consider the tree   
			\[ T:=\{\overbrace{T_1,\ldots,T_1}^{i_1},\ldots,\overbrace{T_k,\ldots,T_k}^{i_k}| \bullet  \}.\]
			Then we have to show that the following relation holds
			\begin{equation}\label{eq: unnorm prod rule} P(T) = \frac{P(\{T_1|\bul\})^{i_1}\cdots P(\{T_k|\bul \})^{i_k}}{i_1!\cdots i_k!}.
			\end{equation}
			We work inductively on $|T|$, the number of vertices of $T$ (including its root): when $ k =i_1=1$ (in particular, for $T=\bullet,\xy
			{\ar@{-}(0,-2)*{\bul};(0,2)*{\circ}};
			\endxy$) there is nothing to prove, which gives the basis of our induction. It is easy to see that
			\begin{multline}\label{eq:copr} \Delta(T) = \sum_{j=1}^k \left( T_j\otimes \{\overbrace{T_1,\ldots,T_1}^{i_1},\ldots,\overbrace{T_j,\ldots,T_j}^{i_j-1},\ldots,\overbrace{T_k,\ldots,T_k}^{i_k}|\bul  \}+ \right.\\ \left. +\sum k_j^{(2)}T_j^{(1)}\otimes \{T_j^{(2)},\overbrace{T_1,\ldots,T_1}^{i_1},\ldots,\overbrace{T_j,\ldots,T_j}^{i_j-1},\ldots,\overbrace{T_k,\ldots,T_k}^{i_k} |\bul \}  \right),
			\end{multline}
			where $k_j^{(2)}=i_h +1$ if $T_j^{(2)}= T_h$ for some $h\neq j$, and $k_j^{(2)}=1$ otherwise. Thus, by the previous lemma and the inductive hypothesis
			\[ P'(T)=\sum_{j=1}^k\left( \widetilde{a}_{T_j} +\sum \widetilde{a}_{T_j^{(1)}} P(\{T_j^{(2)}|\bul \}) \right)\frac{P(\{T_1|\bul \})^{i_1}\cdots P(\{ T_j|\bul \})^{i_j-1}\cdots P(\{T_k|\bul \})^{i_k}}{i_1!\cdots (i_j-1)!\cdots i_k!}. \]
			On the other hand, another application of Lemma \ref{prop:coproduct recursion} (together with the previous formula \eqref{eq:copr} for $k=i_1=1$) shows that
			\[ \widetilde{a}_{T_j} +\sum \widetilde{a}_{T_j^{(1)}} P(\{T_j^{(2)}|\bul\}) = P'(\{T_j|\bul \}),\]
			hence the right-hand side of the above identity is precisely the derivative of the right-hand side of \eqref{eq: unnorm prod rule}. Since both sides of \eqref{eq: unnorm prod rule} evaluate to zero for $t=0$, this concludes the proof.
			
			Let us now turn to part (II) of Theorem \ref{th: product rule}. We already checked the claim for $T=\bullet$ (in the computation from Remark \ref{remark:recucomp}), hence we can work inductively on the number of vertices of $T$. Since $|\Aut(\{T|\bul \})|=|\Aut(T)|$, the claim is equivalent to 
			\[ P(\{T|\bul \})(t) = \int_0^t g(D)P(T)(\tau) d\tau. \]
			In the course of the proof, we shall denote the coefficients of the polynomial $P(T)(t)$ by $c(T,k)$, that is,
			\[ P(T)(t) =\sum_{k=0}^{|T|-1} c(T,k) t^k.   \]
			It follows directly from the definitions that 
			\[ \widetilde{a}_T = \sum_{k=0}^{|T|-1} c(T,k) a_k,\qquad g(D)P(T)(t) = \sum_{k=0}^{|T|-1} c(T,k) a_k(t),  \]	
			and it follows immediately from Lemma \ref{prop:coproduct recursion} that if $T\neq\bullet$
			\[ c(T,0)=0,\qquad c(T,k) = \frac{1}{k}\sum \widetilde{a}_{T^{(1)}}c(T^{(2)},k-1)\quad\mbox{if $k\geq1$}.  \]
			By Lemma $\ref{prop:coproduct recursion}$ and the inductive hypothesis, together with the recursive relation $\int_0^t a_k(\tau) d\tau= \frac{1}{k+1}(a_{k+1}(t)-a_{k+1})$ for the Appell sequence $a_0(t),\ldots, a_k(t),\ldots$, we can now compute
			\begin{eqnarray*}
				P'(\{T|\bul \})(t) &=& \widetilde{a}_T +\sum \widetilde{a}_{T^{(1)}} P(\{T^{(2)}|\bul \})(t) = \widetilde{a}_T +\sum \widetilde{a}_{T^{(1)}} \int_0^t g(D)P(T^{(2)})(\tau) d\tau= \\ &=&\sum_{k=1}^{|T|-1}\left( c(T,k) a_k + \sum \widetilde{a}_{T^{(1)}}\int_0^t c(T^{(2)}, k-1) a_{k-1}(\tau)d\tau \right) =\\&=& \sum_{k=1}^{|T|-1}\left( c(T,k) a_k + \frac{1}{k}\sum \widetilde{a}_{T^{(1)}}c(T^{(2)} ,k-1) (a_{k}(t)- a_{k}) \right)=\\&=& \sum_{k=1}^{|T|-1}\left( c(T,k) a_k + c(T,k) (a_k(t)-a_k) \right) = \\
				&=& g(D)P(T)(t).
			\end{eqnarray*}
			This shows that the identity \eqref{eq: umbral substitution} holds after differentiation, and since both sides vanish when $t=0$, this concludes the proof.
		\end{proof}
		
		\begin{definition} Given a rooted tree $T$ and a vertex $v\in V(T)$, the \emph{hook lenght} $hl(v)$ of $v$ is the number of descendants of $v$ (including $v$ itself). The \emph{tree factorial} of $T$ is the number $T!:=\prod_{v\in V(T)}hl(v)$. For instance, 
		\end{definition}
		\[ \xy
		{(0,-8)};
		{(-8,-2)};
		{(-12,4)};
		{(-6,7)};
		{(0,4)};
		{(8,3)};
		{(8,-2)};
		{(3,11)};
		{(9,11)};		
		{\ar@{-}(0,-6)*{\bul};(-6,-2)*{\circ}};
		{\ar@{-}(0,-6)*{\bul};(6,-2)*{\circ}};
		{\ar@{-}(-6,-2)*{\circ};(-6,4)*{\circ}};
		{\ar@{-}(-6,-2)*{\circ};(-10,4)*{\circ}};
		{\ar@{-}(-6,-2)*{\circ};(-2,4)*{\circ}};
		{\ar@{-}(6,-2)*{\circ};(6,3)*{\circ}};
		{\ar@{-}(6,3)*{\circ};(3,8)*{\circ}};
		{\ar@{-}(6,3)*{\circ};(9,8)*{\circ}};
		\endxy \mbox{\Huge !} = 9\cdot 4\cdot 4\cdot1\cdot1\cdot1\cdot3\cdot1\cdot 1= 432.
		\]
		
		\begin{example} We illustrate Theorem \ref{th: product rule} in the simplest possible case, namely, when the initial $\delta$-series is $f(t)=t$. In this case, we also have $g(t)=1$ and $g(D)=\id:\K[t]\to\K[t]$. We show by induction that $p(T)(t) = \frac{|T|}{T!}t^{|T|-1}$. In fact, the formula is evidently true for $p(\bullet)=1$. Given a tree $T=\{T_1,\ldots,T_k|\bullet\}$, $k\geq1$, using Theorem \ref{th: product rule} and the inductive hypothesis, we see that
			\[ p(T)(t) =\prod_{j=1}^k \int_0^t \frac{|T_j|}{T_j!}\tau^{|T_j|-1}d\tau = \prod_{j=1}^k\frac{1}{T_j!} t^{|T_j|} = \frac{|T|}{|T|\cdot T_1!\cdots T_k!} t^{|T_1|+\cdots+|T_j|}=\frac{|T|}{T!}t^{|T|-1},   \]
			as desired.
			
			As an application, we recover the well-known formula
			\begin{equation}\label{eq:appl}  \bullet^{\curvearrowright n} = \sum_{T\in\T(n)} \frac{n!}{|\Aut(T)|\cdot T!} \,T, \end{equation}
			where the sum runs over the set of rooted trees with $n$ vertices. In fact, the same argument as in the proof of Proposition \ref{prop:recvsdiffeq} shows that, when $f(t)=t$, the Taylor series of the solution $P\in\T[t] $ to \eqref{eq:differential equation} reads $P= \sum_{k\geq0}\frac{t^k}{k!}\,\bullet^{\curvearrowright k+1}$ (notice that $f(t)=t$ implies $a=f_\curvearrowright^{-1}(\bullet)=\bullet$). Comparing this expansion for $P$ with the one $P=\sum_T\frac{p(T)(t)}{|\Aut(T)|} T$, together with the previous computation of the polynomials $p(T)(t)$, yelds the desired formula \eqref{eq:appl}. 
		\end{example}	
		\begin{remark}\label{rem:leadcoeff} The same inductive argument as in the previous example shows, more generally, that $p(T)(t)=\frac{|T|}{T!}(a_0 t)^{|T|-1} + \{\mbox{lower degree terms}\}$, for any initial $\delta$-series $f(t)\in\K[[t]]$ and any rooted tree $T$.
		\end{remark}

		\subsection{The pre-Lie logarithm in $\T$}\label{subsection: log in sT}

		We shall apply the results from the previous subsection in the particular case $f(t)= e^t-1$, $g(t)=\frac{t}{e^t-1}=\sum_{k\ge0}\frac{B_k}{k!}t^k$, where \[B_0=1,\quad B_1=-\frac{1}{2},\quad B_2=\frac{1}{6},\quad B_3=0,\quad B_4=-\frac{1}{30},
		\quad \ldots\] 
		are the Bernoulli numbers. The Appell sequence associated to $B_0,B_1,\ldots,B_k,\ldots$ is the usual sequence of Bernoulli polynomials
		\begin{eqnarray*} & B_0(t) = 1,\qquad B_1(t)= t-\frac{1}{2},\qquad B_2(t)= t^2-t+\frac{1}{6},&\\  & B_3(t)= t^3-\frac{3}{2}t^2+\frac{1}{2}t,\qquad 
			B_4(t) = t^4-2t^3+t^2-\frac{1}{30},& \quad \ldots 
		\end{eqnarray*}
		
		\begin{remark} Let us introduce some notation. Having fixed the above choice of $f(t)$, we shall use a special notation for the solution of \eqref{eq:differential equation} in this context, namely, we shall denote the (normalized) polynomial $p(T)(t)$ associated to a tree $T$, as in the previous subsection, by
			\[ p(T)(t) = \binom{t}{T}.  \] 
			Moreover, we denote the corresponding coefficient by
			$$B_T = \left< \frac{D}{e^D-1} \left| {t \choose T}\right. \right> \, \in \K,$$
			and we call it the \emph{Bernoulli coefficent} of $T$.
			
			To justify the notation for the coefficients (cf.~also Remark \ref{rem: Bern poly}), we notice that for the corolla $C_{n+1}$  with $n$ leaves we get (by Theorem \ref{th: product rule}, Part I)
			\[ C_{n+1}\::=\:\overbrace{\xy 
				{\ar@{-}(0,-4)*{\bul};(-10,4)*{\circ}};
				{\ar@{-}(0,-4)*{\bul};(-6,4)*{\circ}};
				{\ar@{}(0,-4)*{\bul};(2,4)*{\cdots\cdots}};
				{\ar@{-}(0,-4)*{\bul};(10,4)*{\circ}}; 
				\endxy}^n,
			\qquad { t\choose C_{n+1} } = t^n,\quad B_{C_{n+1}}=B_n.
			\]

			To justify the notation for the polynomials, we notice that the polynomial sequence $p_n(t)$ of binomial type associated to the $\delta$-series $e^t-1$ is the usual sequence of falling factorials $p_0(t)=(t)_0:=1$, $p_n(t) = (t)_n := t(t-1)\cdots(t-n+1)$ if $n\geq1$, cf.~\cite[\S 4.1.2]{Roman}. We apply Remark \ref{remark: diff operator} to deduce
			
			\begin{proposition}\label{corollary: Bernoulli coefficients} For the tall tree $T_{n+1}$ with $n+1$ vertices, we have \[ \binom{t}{T_{n+1}}=\binom{t}{n},\qquad B_{T_{n+1}}=\frac{(-1)^n}{n+1},\]
				where ${t\choose n}:=\frac{(t)_n}{n!}$ is the generalized binomial coefficient.
			\end{proposition}
			\begin{proof}
				The identity on the left is clear by Remark \ref{remark: diff operator}.
				The coefficient $B_{T_{n+1}}$ is thus given by
				$$ \left< \frac{D}{e^D-1}\left| {t \choose n} \right.\right>= \left< \frac{D}{e^D-1}\left| \frac{1}{n!}\sum_{k=0}^n s(n,k)t^k \right.\right>  = \frac{1}{n!}\sum_{k=0}^n s(n,k) B_k,$$
				where $s(n,k)$ are the Stirling numbers of the first kind.
				This last sum is known to equal $\frac{(-1)^n}{n+1}$, cf.~for instance \cite[\S 4.2.2]{Roman}.
				
				Alternatively, we may consider the unique morphism of pre-Lie algebras $\Psi_t:\T\to\K[[t]]$ sending $\bul$ to $t$. This morphism $\Psi_t$ is easily described according to Remark \ref{rem: sym braces in free pl}: since $\K[[t]]$ is associative, the associated braces $\{\#,\ldots,\#|\# \}\colon\K[[t]]^{\odot n}\otimes\K[[t]]\to\K[[t]]$ vanish for $n\geq2$, hence $\Psi_t(T)=0$ whenever $T$ is not a tall tree, while for the tall tree $T_{n+1}$ we find $\Psi_t(T_{n+1})=t^{n+1}$. The claim about $B_{T_{n+1}}$ follows by comparing the coefficients of $t^{n+1}$ in the identity \[ \Psi_t\left(\sum_{T}\frac{B_T}{|\Aut(T)|} \,T\right) =\Psi_t\left(\log_\curvearrowright(1+\bullet)\right) 
				= \log(1+t) = \sum_{k\geq1} \frac{(-1)^{k-1}}{k} t^k. \]
			\end{proof}

		\end{remark}
		
		We can reformulate the main result from the previous subsection in the present context. 
		
		First, we notice that the operator 
		\[ \K[\tau]\to\K[t]\:\:\colon\:\: p(\tau) \mapsto \int_0^t \frac{D}{e^D-1} p(\tau) d\tau\:\:\colon\:\: \tau^k\mapsto\frac{B_{k+1}(t)-B_{k+1}}{k+1} \]
		appearing in the right-hand side of Equation \eqref{eq: umbral substitution} coincides with the usual indefinite sum operator
		\[ \sum_{\tau=0}^{t-1}\colon \K[\tau]\to\K[t]\colon\:\: p(\tau)\mapsto \sum_{\tau=0}^{t-1} p(\tau)\:\:\colon\:\: \tau^k\mapsto \sum_{\tau=0}^{t-1}\tau^k\: =\:\frac{B_{k+1}(t)-B_{k+1}}{k+1}, \]
		where the identity on the left-hand side is the classical Faulhaber's formula. Equivalently, we may apply Remark \ref{remark: diff operator} to reach the same conclusion, and notice that the operator $e^D-1 = \sum_{k\geq1}\frac{1}{k!} D^k$ associated to $e^t-1$ coincides with the usual forward difference operator
		\[ \vv{\Delta}\colon \K[t]\to\K[t]\colon p\mapsto \vv{\Delta}p, \qquad \vv{\Delta}p(t) := p(t+1)- p(t), \] 
		and that $\sum_{\tau=0}^{t-1}$ is (by definition) the formal inverse to $\vv{\Delta}$.
		
		We can now restate Theorem \ref{th: product rule} in the following form.
		\begin{theorem}\label{th: difference rule} Given a tree $T=\{T_1,\ldots,T_k|\bul\}$, the identity
			\[ \binom{t}{T} = \sum_{\tau_1,\ldots,\tau_k=0}^{t-1} \binom{\tau_1}{T_1}\cdots\binom{\tau_k}{T_k} \]
			holds.
		\end{theorem}
		
		\begin{proof} Using the first part of Theorem \ref{th: product rule} we are reduced to the case $k=1$, which follows from the second part of the same theorem and the above discussion.
		\end{proof}
		
		Our next goal is to apply the previous theorem to get a description of the polynomial ${ t\choose T}$ and the coefficient $B_T$ in terms of purely combinatorial data associated to the tree $T$.		
		Recall that given a tree $T$, we put a partial order $\leq$ on the set $V(T)$ of vertices (including the root) of $T$ by saying that $v'\leq v$ if $v$ lies in the unique path from $v'$ to the root. If this is the case, we call $v'$ a descendant of $v$.
		
		\begin{definition}\label{def:decreasing decoration}
			A {\em decreasing decoration} of $T$
			is a strictly monotone correspondence $d:(V(T),\leq)\to(\mathbb{N},\leq)$, or in other words, it is the association of a natural number $d(v)$ to each vertex $v$ of $T$, such that $d(v)\gneqq d(v')$ whenever $v'\neq v$ is a descendant of $v$.
			We say that a decreasing decoration is \emph{complete} if it maps $V(T)$ surjectively onto a segment $[0,k]=\{0,1,\ldots,k\}\subset\mathbb{N}$, or in other words, if all the numbers from zero to a certain $k\in\mathbb{N}$ -- which is necessarily associated to the root -- appear as the label of some vertex of $T$. 
			
			We shall denote by $\Adm(T,i)$ (resp.: $\Adm^c(T,i)$) the set of those (resp.: complete) decreasing decorations of $T$ which associate the number $i$ to the root of $T$.
		\end{definition}
		
		\begin{example}
			Two examples of decreasing decorations are
			$$
			
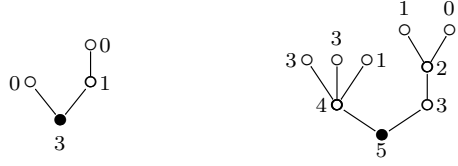
\xy     {(-6,1)*{_0}};
			{(0,-7)*{_3}};
			{(6,1)*{_1}};
			{(6,6)*{_0}};
			{\ar@{-}(0,-4)*{\bul};(-4,1)*{\circ}};
			{\ar@{-}(0,-4)*{\circ};(4,1)*{\circ}};
			{\ar@{-}(4,1)*{\circ};(4,6)*{\circ}};
			\endxy \qquad \qquad \qquad
			\xy
			{(0,-8)*{_5}};
			{(-8,-2)*{_4}};
			{(-12,4)*{_3}};
			{(-6,7)*{_3}};
			{(0,4)*{_1}};
			{(8,3)*{_2}};
			{(8,-2)*{_3}};
			{(3,11)*{_1}};
			{(9,11)*{_0}};		
			{\ar@{-}(0,-6)*{\bul};(-6,-2)*{\circ}};
			{\ar@{-}(0,-6)*{\bul};(6,-2)*{\circ}};
			{\ar@{-}(-6,-2)*{\circ};(-6,4)*{\circ}};
			{\ar@{-}(-6,-2)*{\circ};(-10,4)*{\circ}};
			{\ar@{-}(-6,-2)*{\circ};(-2,4)*{\circ}};
			{\ar@{-}(6,-2)*{\circ};(6,3)*{\circ}};
			{\ar@{-}(6,3)*{\circ};(3,8)*{\circ}};
			{\ar@{-}(6,3)*{\circ};(9,8)*{\circ}};
			\endxy
			$$
			The one on the right-hand side is complete, while the one on the left-hand side is not. The following labelings are {\em not} decreasing
			$$
			
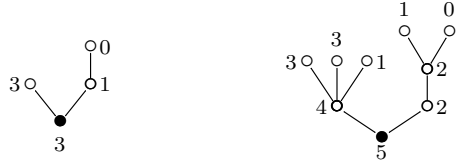
\xy     {(-6,1)*{_3}};
			{(0,-7)*{_3}};
			{(6,1)*{_1}};
			{(6,6)*{_0}};
			{\ar@{-}(0,-4)*{\bul};(-4,1)*{\circ}};
			{\ar@{-}(0,-4)*{\circ};(4,1)*{\circ}};
			{\ar@{-}(4,1)*{\circ};(4,6)*{\circ}};
			\endxy \qquad \qquad \qquad
			\xy
			{(0,-8)*{_5}};
			{(-8,-2)*{_4}};
			{(-12,4)*{_3}};
			{(-6,7)*{_3}};
			{(0,4)*{_1}};
			{(8,3)*{_2}};
			{(8,-2)*{_2}};
			{(3,11)*{_1}};
			{(9,11)*{_0}};		
			{\ar@{-}(0,-6)*{\bul};(-6,-2)*{\circ}};
			{\ar@{-}(0,-6)*{\bul};(6,-2)*{\circ}};
			{\ar@{-}(-6,-2)*{\circ};(-6,4)*{\circ}};
			{\ar@{-}(-6,-2)*{\circ};(-10,4)*{\circ}};
			{\ar@{-}(-6,-2)*{\circ};(-2,4)*{\circ}};
			{\ar@{-}(6,-2)*{\circ};(6,3)*{\circ}};
			{\ar@{-}(6,3)*{\circ};(3,8)*{\circ}};
			{\ar@{-}(6,3)*{\circ};(9,8)*{\circ}};
			\endxy
			$$
		\end{example}
		
		We need the following lemma.
		
		\begin{lemma}\label{lemma: complete admissible trees}
			Given a tree $T \in \T$, the identity
			$$ |\Adm^c(T,i) | = \sum_{j=0}^i (-1)^{i-j}{i \choose j}|\Adm(T,j)|$$
			holds.
		\end{lemma}
		\begin{proof}
			Let us denote by $\Adm^{(k)}(T,i)\subset\Adm(T,i)$, $1\le k \le i$, the set of decreasing
			decorations of $T$ such that $k-1$ does {\em not} appear as a label.
			The set of complete decorations can be written as
			$$ \Adm^c(T,i) = \Adm(T,i) \setminus \bigcup_{k=1}^i \Adm^{(k)}(T,i).$$
			Moreover, for a collection of distinct numbers $1\leq k_1,\dots,k_r \le i$, we have
			$$ \left| \bigcap_{s=1}^r \Adm^{(k_s)}(T,i) \right| = |\Adm(T,i-r)|.$$
			The inclusion-exclusion principle now yields
			$$
			|\Adm^c(T,i)| = \sum_{j=0}^i(-1)^j{i \choose j} |\Adm(T,i-j)| =
			\sum_{j=0}^i (-1)^{i-j}{i\choose j} |\Adm(T,j)|.
			$$
		\end{proof}

		As a consequence of Theorem \ref{th: difference rule}, we obtain an interpretation of the polynomial $\binom{t}{T}$ and the coefficient $B_T$
		in terms of the decreasing decorations of $T$.
		
		\begin{theorem}\label{th: combinatorial} The polynomial $\binom{t}{T}$ is the unique polynomial that, when evaluated at a natural number $i\in \mathbb{N}$, yields the number of decreasing decorations of $T$ associating $i$ to the root, i.e.,
			\[ \binom{i}{T} = |\Adm(T,i)|,\qquad \forall i\in \mathbb{N}.  \]
			In particular, it has the following expansion in the basis $\{\binom{t}{k}\}_{k\in\mathbb{N}}$ of $\K[t]$ given by the generalized binomial coefficients:
			\[ \binom{t}{T} = \sum_{k= h(T)}^{|T|-1}|\Adm^c(T,i)|\binom{t}{k},\]
			where $h(T)$ is the \emph{height} of $T$. Finally, we get the following combinatorial formula for the Bernoulli coefficient $B_T$
			\[ B_T = \sum_{k= h(T)}^{|T|-1}\frac{(-1)^{k} }{k+1}|\Adm^c(T,i)|.  \]
		\end{theorem}
		\begin{proof}
			The first formula follows from the fact that the numbers $|\Adm(T,i)|$
			obey the same recursions as ${i \choose T}$, that is, we have:
			\begin{itemize}
				\item $|\Adm(\bul,i)|=1$,
				\item $|\Adm(\{T_1,\dots,T_k | \bul\},i)|= |\Adm(\{T_1 | \bul\},i)|\cdots |\Adm(\{T_k | \bul\},i)|$,
				\item $|\Adm(\{T | \bul\} ,i)| =\sum_{j=0}^{i-1} |\Adm(T ,j)|.$
			\end{itemize}
			To verify the last relation, observe that if a decreasing decoration associates $i$ to the root of $\{T | \bul\}$, it associates a number strictly smaller than $i$ to the root of $T$.

			The second formula is a consequence of the first one, Lemma \ref{lemma: complete admissible trees} and the following general fact: given a polynomial $q(t)$ of degree $d$, the expansion of $q(t)$ in the basis $\{{t \choose n}\}_{n\ge 0}$ of generalized binomial coefficient reads
			$$ q(t) = \sum_{j=0}^d \lambda_j{t \choose j}, \quad \textrm{where} \quad \lambda_j = \sum_{i=0}^j (-1)^{i-j}{j \choose i} q(i).$$
			For a proof of this fact, cf.~\cite[pg. 59]{Roman}. Moreover, notice that $|\Adm^c(T,k)|=|\Adm(T,k)|=0$ whenever $k$ is strictly smaller than $h(T)$, the height of $T$.
			
			Finally, the claim about the Bernoulli coefficients $B_T$ is a consequence of the established formula for ${t \choose T}$ and the fact that, as observed in Proposition \ref{corollary: Bernoulli coefficients}, 
			$$  \left< \frac{D}{e^D-1}
			\left| {t \choose k}\right.\right> = \frac{(-1)^k}{k+1}.$$
			
		\end{proof}
		
		\begin{remark}\label{rem:} Comparing the previous theorem and Remark \ref{rem:leadcoeff}, we recover the following well-known formula 
			\[ |\Adm^c(T,|T|-1)| = \frac{|T|!}{T!}. \]
		\end{remark}		
		\begin{remark}\label{rem: Bern poly} We may associate to a tree $T$ the polynomial $B_T(t)$ defined by
			\[ B_T(t) := \frac{D}{e^D-1} {t\choose T}. \]
			This is the only polynomial satisfying the identity 
			\[ \sum_{i=m}^{n-1} { i\choose T} = \int_m^n B_T(t)dx,   \]
			where the previous theorem provides a combinatorial interpretation of the left-hand side. It makes some sense to consider $B_T(t)$ as a generalized Bernoulli polynomial associated to the tree $T$, and to consider the previous formula as a generalization of Faulhaber's formula. In particular, we recover the ordinary Bernoulli polynomial $B_n(t)$ as the polynomial associated to the corolla $C_{n+1}$ with $n$ leaves. This justifies the terminology \emph{Bernoulli coefficient} for the number $B_T=B_T(0)$.
		\end{remark}
		
		We conclude this subsection by presenting a simple umbral proof of \cite[Proposition 4.3]{alg B-series}. To state the result, we recall the following notation from \cite{alg B-series}: Given two trees $T=\{T_1,\ldots,T_j|\bul \}$ and $T' = \{  T'_1,\ldots T'_k|\bul \}$, their Butcher products $T\circ T'$, $T'\circ T$, and their merging product $T\ast T'$ are defined respectively by (cf.~\cite[fig. 3.1 at pg. 75]{Hairer})  
		\[ T\circ T' = \{  T_1,\ldots,T_j,T' |\bul \},\quad T'\circ T = \{ T'_1,\ldots,T'_k,T|\bul  \},\quad T\ast T' =\{ T_1,\ldots,T_j,T'_1,\ldots, T'_k |\bul\}.  \]
		\begin{proposition} For every pair of trees $T,T'$, we have 
			\[ B_{T\circ T'}+B_{T'\circ T}+B_{T\ast T'} =0. \]\end{proposition}
		\begin{proof} By the product rule for the difference operator
			\[ \vv{\Delta}( pq ) = \vv{\Delta}(p)q+p\vv{\Delta}(q)+\vv{\Delta}(p)\vv{\Delta}(q),\]
			together with Theorems \ref{th: product rule} and \ref{th: difference rule}, we see that
			\[ \binom{t}{T\circ T'} + \binom{t}{T'\circ T} + \binom{t}{T\ast T'} =\vv{\Delta}\binom{t}{\{ T,T'|\bul \}}= (e^D-1)\binom{t}{\{ T,T'|\bul \}}. \]
			Using \cite[Theorem 2.2.5]{Roman}, we get
			\[ B_{T\circ T'}+B_{T'\circ T}+B_{T\ast T'} = \left< \frac{D}{e^D-1} \left| (e^D-1) \binom{t}{\{ T,T'|\bul \}} \right.\right> =\left< D\left| \binom{t}{\{ T,T'|\bul \}}\right.\right>  =0,   \] 
			where the last identity follows from the fact that $\langle D|q(t)\rangle=q'(0)$ for all $q(t)\in\K[t]$, together with $\binom{t}{\{ T,T'|\bul \}}=\binom{t}{\{ T|\bul \}}\binom{t}{\{ T'|\bul \}}$ and $\binom{0}{\{ T|\bul \}}=\binom{0}{\{ T'|\bul \}}=0$. 
			
		\end{proof}

		\section{{\bf The Eulerian idempotent in the PBW basis}}\label{section: Eulerian coefficients}
		
		In this section, we bring together the problem outlined in the Introduction and the results from  the previous section. More precisely,
		we show that expressing $E(x_1\cdots x_n)$ in the PBW basis
		$\basis_n$, for all $n\ge 1$, see Subsection \ref{subsection: problem}, 
		is equivalent to computing a logarithm in a certain pre-Lie algebra $(\pbT,\rhd)$
		of planar binary rooted trees. We then address the latter problem with the methods developed in Section \ref{section: pre-Lie}.
		
		\subsection{A recursion for the Eulerian idempotent}\label{subsection: recursion for Eulerian}
		
		Let $\g$ be a pro-nilpotent Lie algebra over a field $\K$ of characteristic zero.
		We denote its universal enveloping algebra by $\U(\g)$. This is a biaugmented cocommutative bialgebra. We denote the product in $\U(\g)$ by $\cdot$ and the coproduct by $\Delta_{\U(\g)}$.
		Let $S(\g)$ be the symmetric coalgebra over $\g$. We denote the image of $x_1\otimes \cdots \otimes x_n \in T(\g)$ under the canonical projection
		$T(\g) \to S(\g)$ by $x_1\odot \cdots \odot x_n$.
		
		The Poincar\'e-Birkhoff-Witt Theorem asserts that the symmetrization map $\sym: S(\g) \to \U(\g)$, defined by
		\begin{eqnarray*}
			\sym(1)=1 \qquad \textrm{and} \qquad \sym(x_1\odot \cdots \odot x_n)= \frac{1}{n!}\sum_{\sigma \in S_n} x_{\sigma(1)}\cdots x_{\sigma(n)},
		\end{eqnarray*}
		is an isomorphism of coaugmented coalgebras.
		
		The following lemma is an easy consequence of \cite[Theorem 1.2]{Bandiera_nonabelian}, see also \cite[Equation (3.12)]{Bandiera_Kapranov} and the discussion therein. 
		
		We denote by $\Coder(S(\g))$ the Lie algebra of coderivations of the symmetric coalgebra $S(\g)$. Recall that, since $S(\g)$ is cofree, every coderivation $Q\in\Coder(S(\g))$ is completely described by the family of its Taylor coefficients $Q_n:\g^{\odot n}\to \g$, which are defined as the composition $Q_n:\g^{\odot n}\hookrightarrow S(\g)\xrightarrow{Q}S(\g)\xrightarrow{p}\g$, where the last map $p$ is the natural projection.
		\begin{lemma}\label{lemmaHDB} The linear map $\Phi^\perp: \g \to \Coder(S(\g))$, given in Taylor coefficients by
			\[ \Phi^\perp(x)_0(1)=x, \qquad \Phi^\perp(x)_k(y_1\odot \cdots \odot y_k)=(-1)^k\frac{B_k}{k!}\sum_{\sigma \in S_k}[y_{\sigma(1)},\cdots [y_{\sigma(k)},x]\cdots],\]
			is an anti-morphism of Lie algebras, i.e., we have $\Phi^\perp([x,y])=-[\Phi^\perp(x),\Phi^\perp(y)]$ for all $x, y \in \g$.
		\end{lemma}
		
		As a consequence, the map $\Phi^\perp: \g \to \Coder(S(\g)) \subset \End(S(\g))$ extends 
		uniquely to an anti-morphism of algebras
		$$ \Phi^\perp: \U(\g) \to \End(S(\g)).$$
		
		\begin{definition}
			We define a map $\eta: \U(\g) \to S(\g)$
			by setting
			$$\eta(1)=1 \quad \textrm{and} \quad \eta(x_1\cdots x_n):=\Phi^\perp(x_1 \cdots x_n)(1) = \Phi^\perp(x_n)\circ\cdots\circ \Phi^\perp(x_1)(1)$$
			for all $x_1,\ldots,x_n\in\g$ (where we denote by $1$ both the unit in $S(\g)$ and the one in $\U(\g)$).
		\end{definition}
		
		\begin{lemma}
			The map $\eta: \U(\g) \to S(\g)$ is the inverse to $\sym: S(\g) \to \U(\g)$.
		\end{lemma}

		\begin{proof}
			Following the proof of \cite[Theorem 3.3]{Bandiera_Kapranov}, one shows that $\eta$ is an isomorphism of coaugmented coalgebras.
			We claim that the composition $$S(\g) \stackrel{\sym}{\longrightarrow} \U(\g) \stackrel{\eta}{\longrightarrow} S(\g),$$
			which is an automorphism of the coaugmented coalgebra $S(\g)$, is the identity.
			Since $S(\g)$ is spanned by elements of the form
			$x^{\odot n}$, $x\in \g$, it suffices to verify that $\eta\circ \sym$ is the identity on such elements. By the compatibility with the comultiplication, and the cofreeness of $S(\g)$, it therefore suffices to verify that
			
			$$ p\circ \eta \circ \sym(x^{\odot n}) = \begin{cases} x & \textrm{ if } n=1 \\
			0 & \textrm{ otherwise}\end{cases},$$
			where $p$ denote the canonical projection $S(\g) \to \g$.
			This is straightforward for $n=0,1$. Assuming by induction that $\eta \circ \sym(x^{\odot n})=\eta(x^n)=x^{\odot n}$, we compute
			\begin{eqnarray*}
				p\circ \eta \circ \sym(x^{\odot (n+1)}) &=& p\circ \eta(x^{n+1}) = p\circ\Phi^\perp(x^{n+1})(1) = p\circ\Phi^\perp(x)\circ\Phi^\perp(x^{n})(1)\\
				&=&p\circ\Phi^\perp(x)\circ\eta(x^{n}) =
				p\circ\Phi^\perp(x)(x^{\odot n})=\Phi^\perp_n(x)(x^{\odot n}) = 0 .
			\end{eqnarray*}
		\end{proof}
		
		\begin{definition} The Eulerian idempotent $E$ on $\U(\g)$ is the composition
			$$ E: \U(\g) \xrightarrow{\eta} S(\g) \stackrel{p}{\longrightarrow} \g\hookrightarrow\U(\g).$$
		\end{definition}
		\begin{remark} $E$ may be equivalently defined as the logarithm $E=\log_\star(\id)$ of the identity in $\End(\U(\g))$ with respect to the convolution product $\star$, cf.~\cite{Loday}.
		\end{remark}
		Since $\eta$ is a morphism of coalgebras, and since $S(\g)$
		is cofree, we can write $\eta$ in terms of its corestriction $E:\U(\g)\to\g$ as follows:
		\begin{equation*}\label{eqn: sym-1}
		\eta(x_1\cdots x_n) =\sum_{k=1}^n \frac{1}{k!}\sum_{i_1+\cdots+i_k=n}\sum_{\sigma \in S(i_1,\dots,i_k)}E(x_{\sigma(1)}\cdots x_{\sigma(i_1)})\odot \cdots \odot E(x_{\sigma(n-i_k+1)}\cdots x_{\sigma(n)}).
		\end{equation*}
		
		As a consequence of these considerations, we obtain the following recursion for the Eulerian idempotent:
		
		\begin{proposition}\label{corollary: Eulerian}
			The map $E$ is determined by the following recursion:
			\begin{itemize}
				\item $E(x) = x$ for all $x\in\g$.
				\item For all $n\ge 1$ and $x_1,\ldots,x_{n+1}\in\g$,
				\begin{multline*}
				E(x_1\cdots x_{n+1}) = \\
				=\sum_{k=1}^n  (-1)^k\frac{B_k}{k!}\sum_{i_1+\cdots +i_k=n} \sum_{\sigma \in S(i_1,\dots,i_k)}[E(x_{\sigma(1)}\cdots x_{\sigma(i_1)}),\cdots[E(x_{\sigma(n-i_k+1)}\cdots x_{\sigma(n)}),x_{n+1}]\cdots]
				\end{multline*}
			\end{itemize}
		\end{proposition}
		
		\begin{proof}
			We have
			\begin{equation*}
			E(x_1\cdots x_{n+1})= p\circ\eta(x_1\cdots x_{n+1})=
			p\circ\Phi^\perp(x_{n+1})(\eta(x_1\cdots x_n)).
			\end{equation*}
			Using the above formula for $\eta(x_1\cdots x_n)$, together with the formulas for the Taylor coefficients of $\Phi^\perp(x_{n+1})$ in Lemma \ref{lemmaHDB}, we arrive at the claimed formula for $E(x_1\cdots x_{n+1})$.
		\end{proof}

		\subsection{The Eulerian idempotent as a pre-Lie logarithm}
		\label{subsection: pbT}
		
		In this subsection we introduce a (left) pre-Lie algebra structure on the vector space $\pbT$ spanned by binary planar rooted trees. Moreover, we show how the computation of $E(x_1\cdots x_n)\in \Lie_n$,  $n\ge 1$, in the PBW basis $\basis_n$ of $\Lie_n$ (cf.~the Introduction) can be reduced to the computation of a pre-Lie logarithm in the pre-Lie algebra $\pbT$.

		\begin{definition}\label{def:pbT}
			Let $\pbT(n)$ be the vector space spanned by all planar, binary, rooted trees with $n$ leaves. By convention, the root of such a tree is always univalent.
			
			We denote by $\pbT$ the direct product $\pbT=\prod_{n\geq1}\pbT(n)$. 
		\end{definition}	
		
		\begin{remark} As we will need to switch between binary and non-binary rooted trees, we shall use different pictures to avoid confusion. Namely, we shall depict binary, planar, rooted trees as in the following example 
			
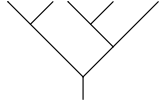
\[  \xy {\ar@{-}(0,-5);(0,-2)};
			{\ar@{-}(0,-2);(-10,8)};
			{\ar@{-}(-7,5);(-4,8)};
			{\ar@{-}(0,-2);(10,8) };
			{\ar@{-}(4,2);(-2,8)}; 	
			{\ar@{-}(1,5);(4,8)};
			\endxy \]	
			while we shall depict non-binary (or rather, not necessarily binary) rooted trees, both in the planar and the non-planar cases, by the same kind of pictures we used in the previous Section \ref{section: pre-Lie}, as in the following example
			
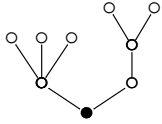
\[  \xy
			{(0,-8)};
			{(-8,-2)};
			{(-12,4)};
			{(-6,7)};
			{(0,4)};
			{(8,3)};
			{(8,-2)};
			{(3,11)};
			{(9,11)};		
			{\ar@{-}(0,-6)*{\bul};(-6,-2)*{\circ}};
			{\ar@{-}(0,-6)*{\bul};(6,-2)*{\circ}};
			{\ar@{-}(-6,-2)*{\circ};(-6,4)*{\circ}};
			{\ar@{-}(-6,-2)*{\circ};(-10,4)*{\circ}};
			{\ar@{-}(-6,-2)*{\circ};(-2,4)*{\circ}};
			{\ar@{-}(6,-2)*{\circ};(6,3)*{\circ}};
			{\ar@{-}(6,3)*{\circ};(3,8)*{\circ}};
			{\ar@{-}(6,3)*{\circ};(9,8)*{\circ}};
			\endxy \]	
		\end{remark}
		
		\begin{definition}\label{def:rhd} Given an edge $e$ of a planar, binary, rooted tree $T$, we say that $e$ is \emph{right pointing} if it is oriented along the south-west/north-east diagonal, and we say that $e$ is \emph{left pointing} if it is oriented along the south-east/north-west diagonal. By convention, although we do not show this in the pictures, we think of the edge connected to the root as a right-pointing edge. We denote by $E(T)$ (resp.: $E_r(T)$, $E_l(T)$)  the set of (resp.: right pointing, left pointing) edges of $T$.
			
			Given trees $T,T'\in\pbT$, and an edge $e$ of $T'$, we denote by $T\searrow_e T'$ the tree obtained by grafting the root of $T$ onto the edge $e$. 
			
			Finally, we define a bilinear product $\rhd\colon\pbT\otimes\pbT\to\pbT$ on $\pbT$ by the formula \begin{equation}\label{eq:rhd} T \rhd T'=\sum_{e\in E_r(T')} T\searrow_e T' - \sum_{e\in E_l(T')} T\searrow_e T'.
			\end{equation}
		\end{definition}
		
		\begin{example}	
			
			Let us provide a few examples:
			\[ 	\xy
			{\ar@{-}(0,-2);(0,2)};
			\endxy\,\,\rhd\,\,	\xy
			{\ar@{-}(0,-2);(0,2)};
			\endxy \quad = \quad\xy
			{\ar@{-}(0,-2);(0,0)};
			{\ar@{-}(0,0);(-2,2)};
			{\ar@{-}(0,0);(2,2)};
			\endxy,\qquad
			\xy
			{\ar@{-}(0,-2);(0,0)};
			{\ar@{-}(0,0);(-2,2)};
			{\ar@{-}(0,0);(2,2)};
			\endxy
			\, \, \rhd \, \,
			\xy
			{\ar@{-}(0,-2);(0,2)};
			\endxy \quad = \quad
			\xy
			{\ar@{-}(0,-2);(0,0)};
			{\ar@{-}(0,0);(4,4)};
			{\ar@{-}(0,0);(-4,4)};
			{\ar@{-}(-2,2);(0,4)};
			\endxy,\qquad\xy
			{\ar@{-}(0,-2);(0,2)};
			\endxy	\, \, \rhd \, \,
			\xy
			{\ar@{-}(0,-2);(0,0)};
			{\ar@{-}(0,0);(-2,2)};
			{\ar@{-}(0,0);(2,2)};
			\endxy
			\quad = \quad 2\,\,\xy
			{\ar@{-}(0,-2);(0,0)};
			{\ar@{-}(0,0);(4,4)};
			{\ar@{-}(0,0);(-4,4)};
			{\ar@{-}(2,2);(0,4)};
			\endxy
			\,\, -\,\,
			\xy
			{\ar@{-}(0,-2);(0,0)};
			{\ar@{-}(0,0);(4,4)};
			{\ar@{-}(0,0);(-4,4)};
			{\ar@{-}(-2,2);(0,4)};
			\endxy,
			\] 
			\[  	\xy
			{\ar@{-}(0,-2);(0,0)};
			{\ar@{-}(0,0);(-2,2)};
			{\ar@{-}(0,0);(2,2)};
			\endxy
			\, \, \rhd \, \,
			\xy
			{\ar@{-}(0,-2);(0,0)};
			{\ar@{-}(0,0);(-2,2)};
			{\ar@{-}(0,0);(2,2)};
			\endxy\quad = \quad  \xy {\ar@{-}(0,-4);(0,-2)};
			{\ar@{-}(0,-2);(-6,4)};
			{\ar@{-}(-4,2);(-2,4)};
			{\ar@{-}(0,-2);(6,4) };
			{\ar@{-}(4,2);(2,4)}; 	
			\endxy\,\,+\,\,\xy {\ar@{-}(0,-4);(0,-2)};
			{\ar@{-}(0,-2);(-6,4)};
			{\ar@{-}(0,2);(2,4)};
			{\ar@{-}(0,-2);(6,4) };
			{\ar@{-}(2,0);(-2,4)}; 	
			\endxy\,\,-\,\, \xy {\ar@{-}(0,-4);(0,-2)};
			{\ar@{-}(0,-2);(-6,4)};
			{\ar@{-}(0,2);(-2,4)};
			{\ar@{-}(0,-2);(6,4) };
			{\ar@{-}(-2,0);(2,4)}; 	
			\endxy.  \]
		\end{example}
		
		\begin{proposition}
			The operation $\rhd$ equips $\pbT$ with the structure of a (complete) left pre-Lie algebra.
		\end{proposition}
		\begin{proof} Given three planar, binary, rooted trees $T_1, T_2, T_3$, a direct computation shows that the associator
			$$A_\rhd(T_1,T_2,T_3) = T_1 \rhd (T_2 \rhd T_3) - (T_1 \rhd T_2) \rhd T_3$$
			equals the (signed, with signs as in Definition \ref{def:rhd}) sum of all possible graftings of $T_1$ and $T_2$ onto two (possibly coinciding) edges of $T_3$, minus the (unsigned) sum of all graftings of the element
			$$ 	\xy
			{\ar@{-}(0,-4);(0,-2)};
			{\ar@{-}(0,-2);(-4,2)};{(-4,4)*{\slab{T_1}}};
			{\ar@{-}(0,-2);(4,2)};{(4,4)*{\slab{T_2}}};
			\endxy 
			\quad  + \quad
			\xy
			{\ar@{-}(0,-4);(0,-2)};
			{\ar@{-}(0,-2);(-4,2)};{(-4,4)*{\slab{T_2}}};
			{\ar@{-}(0,-2);(4,2)};{(4,4)*{\slab{T_1}}};
			\endxy 
			$$	
			onto a right pointing edge of $T_3$. This element of $\pbT$ is clearly symmetric with respect to
			$T_1$ and $T_2$. The pre-Lie algebra is complete with respect to the filtration
			$$F^p\pbT := \prod_{k\ge p}\pbT(k).$$
		\end{proof}
		\begin{example} For instance,
			\[ A_\rhd\left(\,\xy
			{\ar@{-}(0,-2);(0,0)};
			{\ar@{-}(0,0);(-2,2)};
			{\ar@{-}(0,0);(2,2)};
			\endxy
			\,  , \, 
			\xy
			{\ar@{-}(0,-2);(0,2)};
			\endxy\, , \, \xy
			{\ar@{-}(0,-2);(0,2)};
			\endxy \,\right) = A_\rhd\left(\,
			\xy
			{\ar@{-}(0,-2);(0,2)};
			\endxy\, , \xy
			{\ar@{-}(0,-2);(0,0)};
			{\ar@{-}(0,0);(-2,2)};
			{\ar@{-}(0,0);(2,2)};
			\endxy
			\,  , \, \xy
			{\ar@{-}(0,-2);(0,2)};
			\endxy \,\right) = \,\,\xy {\ar@{-}(0,-4);(0,-2)};
			{\ar@{-}(0,-2);(-6,4)};
			{\ar@{-}(-4,2);(-2,4)};
			{\ar@{-}(0,-2);(6,4) };
			{\ar@{-}(4,2);(2,4)}; 	
			\endxy\,\,+\,\,\xy {\ar@{-}(0,-4);(0,-2)};
			{\ar@{-}(0,-2);(-6,4)};
			{\ar@{-}(0,2);(2,4)};
			{\ar@{-}(0,-2);(6,4) };
			{\ar@{-}(2,0);(-2,4)}; 	
			\endxy\,\,-\,\, \xy {\ar@{-}(0,-4);(0,-2)};
			{\ar@{-}(0,-2);(-6,4)};
			{\ar@{-}(0,2);(-2,4)};
			{\ar@{-}(0,-2);(6,4) };
			{\ar@{-}(-2,0);(2,4)}; 	
			\endxy\,\,-\,\, \xy {\ar@{-}(0,-4);(0,-2)};
			{\ar@{-}(0,-2);(-6,4)};
			{\ar@{-}(-4,2);(-2,4)};
			{\ar@{-}(0,-2);(6,4) };
			{\ar@{-}(-2,0);(2,4)}; 	
			\endxy. \]
		\end{example} 	
		\begin{definition}
			We refer to $(\pbT, \rhd)$ as the pre-Lie algebra of planar, binary, rooted trees.
		\end{definition}	
		
		\begin{remark}\label{rem: pre-lie copr}
			It will be convenient to describe explicitly the pre-Lie coproduct on $\pbT$ dual to $\rhd$, which we denote by $\Delta:\pbT\to \pbT\otimes\pbT$. Given a planar, binary, rooted tree $T$, we denote by $V(T)$ the set of inner vertices of $T$. Given $v\in V(T)$, we put $\varepsilon(v)=1$ if $v$ is the vertex connected to the root or if $v$ is a right child, and we put $\varepsilon(v)=-1$ if $v$ is a left child. Moreover, given $v\in V(T)$ such that $\varepsilon(v)=1$ (resp.: $\varepsilon(v)=-1$), we denote by $T_v$ the (planar, binary, rooted at $v$) subtree of $T$ to the left (resp.: right) of $v$, and by $\widehat{T_v}$ the (planar, binary, rooted) tree obtained by removing $T_v$ from $T$. With these notations, the coproduct $\Delta$ is given by the formula
			\begin{equation}\label{pre-lie copr} \Delta(T) = \sum_{v\in V(T)}\varepsilon(v)\,\, T_v\otimes\widehat{T_v}. \end{equation}
			For instance,
			\[ \Delta\left(\xy {\ar@{-}(0,-4);(0,-2)};
			{\ar@{-}(0,-2);(-8,6)};
			{\ar@{-}(-4,2);(0,6)};
			{\ar@{-}(0,-2);(8,6) };
			{\ar@{-}(-2,0);(4,6)}; 
			{\ar@{-}(-2,4);(-4,6)};	
			\endxy\right) = 	\xy {\ar@{-}(0,-4);(0,-2)};
			{\ar@{-}(0,-2);(-6,4)};
			{\ar@{-}(0,2);(-2,4)};
			{\ar@{-}(0,-2);(6,4) };
			{\ar@{-}(-2,0);(2,4)}; 	
			\endxy
			\,  \otimes \,\,
			\xy
			{\ar@{-}(0,-2);(0,2)};
			\endxy\,\,-\,\, \xy
			{\ar@{-}(0,-2);(0,2)};
			\endxy\,\,\otimes\,\xy {\ar@{-}(0,-4);(0,-2)};
			{\ar@{-}(0,-2);(-6,4)};
			{\ar@{-}(0,2);(-2,4)};
			{\ar@{-}(0,-2);(6,4) };
			{\ar@{-}(-2,0);(2,4)}; 	
			\endxy\,\,-\,\,\xy
			{\ar@{-}(0,-2);(0,0)};
			{\ar@{-}(0,0);(-2,2)};
			{\ar@{-}(0,0);(2,2)};
			\endxy\,\,\otimes\,\xy
			{\ar@{-}(0,-2);(0,0)};
			{\ar@{-}(0,0);(4,4)};
			{\ar@{-}(0,0);(-4,4)};
			{\ar@{-}(-2,2);(0,4)};
			\endxy\,\,+\,\,\xy
			{\ar@{-}(0,-2);(0,2)};
			\endxy\,\,\otimes\, \xy {\ar@{-}(0,-4);(0,-2)};
			{\ar@{-}(0,-2);(-6,4)};
			{\ar@{-}(-4,2);(-2,4)};
			{\ar@{-}(0,-2);(6,4) };
			{\ar@{-}(-2,0);(2,4)}; 	
			\endxy. \]
		\end{remark}

		\begin{definition}\label{def:admissible} Let $(X,\le)$ be a totally ordered set.
			A labeling of a planar, binary, rooted tree $T$ by $(X,\le)$ is a map
			$$ \ell : \{\mbox{\emph{leaves of} $T$}\} \to X.$$
			
			We say that a labeling $\ell$ of $T$ is {\em admissible} if 
			
			\begin{itemize} \item the labeling is injective as a function $ \ell : \{\mbox{\emph{leaves of} $T$}\} \to X$; and
				\item for every inner vertex $v$ of $T$, denoting by $\operatorname{Des}(v)$ the set of leaves of $T$ which are descendants of $v$, the rightmost leaf in $\operatorname{Des}(v)$ is labeled by \[\max\{i\in X \, \vert \,  \mbox{$i$ is the label of a leaf in $\operatorname{Des}(v)$ }\},\] and the leftmost leaf in $\operatorname{Des}(v)$ is labeled by \[ \min\{i\in X \, \vert \,  \mbox{$i$ is the label of a leaf in $\operatorname{Des}(v)$ }\}.\] 
			\end{itemize}
			
			We denote by $\pblT(X)$ be the vector spaces spanned by pairs $(T,\ell)$,
			where $T$ is a planar, binary, rooted tree and $\ell$ is an admissible labeling of $T$ by $(X,\le)$.
		\end{definition}
		
		\begin{definition}
			We introduce a bilinear operation $\rhd$ on $\pblT(X)$ as follows:
			\begin{itemize}
				\item$(T_1,\ell_1) \rhd (T_2,\ell_2)=0$ unless the labelings $\ell_1$ and $\ell_2$ have disjoint image.
				\item If $\ell_1$ and $\ell_2$ have disjoint image, then $(T_1,\ell_1) \rhd (T_2,\ell_2)$ is obtained by the element $T_1\rhd T_2\in\pbT$ (as defined in Equation \eqref{eq:rhd}), equipped with the labelings inherited from $\ell_1$ and $\ell_2$, by removing all terms for which the inherited labeling is not admissible.
			\end{itemize}
		\end{definition}
		\begin{example} Here are a few examples of the product $\rhd$ on $\pblT(X)$ (when $(X,\leq)=(\mathbb{N},\leq)$):
			\[ 	\xy
			{\ar@{-}(0,-2);(0,2)};{(0,4)*{\slab{1}}};
			\endxy \,\,\rhd\,\,	
			\xy	{\ar@{-}(0,-2);(0,2)};{(0,4)*{\slab{2}}};
			\endxy \quad = \quad\xy
			{\ar@{-}(0,-2);(0,0)};
			{\ar@{-}(0,0);(-2,2)};{(-2,4)*{\slab{1}}};
			{\ar@{-}(0,0);(2,2)};{(2,4)*{\slab{2}}};
			\endxy,\qquad
			\xy
			{\ar@{-}(0,-2);(0,0)};{(-2,4)*{\slab{1}}};
			{\ar@{-}(0,0);(-2,2)};
			{\ar@{-}(0,0);(2,2)};{(2,4)*{\slab{3}}};
			\endxy
			\, \, \rhd \, \,
			\xy
			{\ar@{-}(0,-2);(0,2)};{(0,4)*{\slab{2}}};
			\endxy \quad = \quad 0,\qquad\xy
			{\ar@{-}(0,-2);(0,2)};{(0,4)*{\slab{2}}};
			\endxy	\, \, \rhd \, \,
			\xy
			{\ar@{-}(0,-2);(0,0)};
			{\ar@{-}(0,0);(-2,2)};{(-2,4)*{\slab{1}}};
			{\ar@{-}(0,0);(2,2)};{(2,4)*{\slab{3}}};
			\endxy
			\quad = \quad \,\,\xy
			{\ar@{-}(0,-2);(0,0)};
			{\ar@{-}(0,0);(4,4)};{(4,6)*{\slab{3}}};
			{\ar@{-}(0,0);(-4,4)};{(-4,6)*{\slab{1}}};
			{\ar@{-}(2,2);(0,4)};{(0,6)*{\slab{2}}};
			\endxy
			\,\, -\,\,
			\xy
			{\ar@{-}(0,-2);(0,0)};
			{\ar@{-}(0,0);(4,4)};{(4,6)*{\slab{3}}};
			{\ar@{-}(0,0);(-4,4)};{(-4,6)*{\slab{1}}};
			{\ar@{-}(-2,2);(0,4)};{(0,6)*{\slab{2}}};
			\endxy,
			\] 
			
			\[  	\xy
			{\ar@{-}(0,-2);(0,0)};
			{\ar@{-}(0,0);(-2,2)};{(-2,4)*{\slab{1}}};
			{\ar@{-}(0,0);(2,2)};{(2,4)*{\slab{3}}};
			\endxy
			\, \, \rhd \, \,
			\xy
			{\ar@{-}(0,-2);(0,0)};
			{\ar@{-}(0,0);(-2,2)};{(-2,4)*{\slab{2}}};
			{\ar@{-}(0,0);(2,2)};{(2,4)*{\slab{4}}};
			\endxy\quad = \quad  \xy {\ar@{-}(0,-4);(0,-2)};
			{\ar@{-}(0,-2);(-6,4)};{(-6,6)*{\slab{1}}};
			{\ar@{-}(-4,2);(-2,4)};{(-2,6)*{\slab{3}}};
			{\ar@{-}(0,-2);(6,4) };{(6,6)*{\slab{4}}};
			{\ar@{-}(4,2);(2,4)}; 	{(2,6)*{\slab{2}}};
			\endxy.  \]
		\end{example}
		
		\begin{remark}\label{remboh} In particular, given $(T_1,\ell_1), (T_2,\ell_2)\in\pblT(X)$ such that $\ell_1,\ell_2$ have disjoint images, we have $(T_1,\ell_1)\rhd(T_2,\ell_2)=\left(\xy
			{\ar@{-}(0,-4);(0,-2)};
			{\ar@{-}(0,-2);(-4,2)};{(-4,4)*{\slab{T_1}}};
			{\ar@{-}(0,-2);(4,2)};{(4,4)*{\slab{T_2}}};
			\endxy  ,\ell\right)$ (where $\ell$ is induced from $\ell_1$ and $\ell_2$) whenever the minimum label of $\ell_1$ is smaller than the minimum label of $\ell_2$, and we have $(T_1,\ell_1)\rhd(T_2,\ell_2)=0$ whenever the maximum label of $\ell_1$ is greater than the maximum label of $\ell_2$.
		\end{remark}

		\begin{lemma}\label{lem:labeling} The product $\rhd$ on $\pblT(X)$ is a pre-Lie product. Furthermore, when the set $X$ is finite,
			the ``labeling" map
			$$ \operatorname{lab}_X: \pbT \to \pblT(X), \quad T \mapsto \sum_{\ell}(T,\ell),$$
			where the sum runs over all admissible labelings of $T$ in $(X,\le)$,
			is a morphism of pre-Lie algebras.
		\end{lemma}
		\begin{proof} The first statement follow from the definition of the product $\rhd$ on $\pblT(X)$. Given $(T_j,\ell_j)\in \pblT(X)$, $j=1,2,3$, by definition both the associators $A_\rhd((T_1,\ell_1),(T_2,\ell_2),(T_3,\ell_3))$ and $A_\rhd((T_2,\ell_2),(T_1,\ell_1),(T_3,\ell_3))$ are obtained by the same element $A_\rhd(T_1,T_2,T_3)=A_\rhd(T_2,T_1,T_3)$ in $\pbT$, equipped with the labelings inherited from $\ell_1,\ell_2,\ell_3$, by removing those terms for which the inherited labeling is not admissible. 
			
			To see that $\operatorname{lab}_X$ commutes with $\rhd$, it is more convenient to consider the dual statement, namely, that the map $\pblT(X)\to\pbT\colon(T,\ell)\to T$ forgetting the labeling is a morphism of pre-Lie coalgebras. In order to check this fact, we observe that the pre-Lie coproduct $\Delta:\pblT(X)\to \pblT(X)\otimes\pblT(X)$ dual to $\rhd$ is given precisely by formula \eqref{pre-lie copr}, where both $T_v$ and $\widehat{T_v}$ are equipped with the inherited labeling (which will be automatically admissible).
			
		\end{proof}
		
		\begin{definition} We denote by $L(X)$ the free Lie algebra over $X$. Given $i\in X$, we shall denote the corresponding generator of $L(X)$ by $x_i$. The Lie algebra $\Lie(X)$ is the quotient of $L(X)$ by the ideal $I$ spanned by Lie words such that one of the generators appears more than once.
			
			Given a Lie word $w$ in $\Lie(X)$, we shall refer to
			\[ \min\{ i\in X\,|\, \mbox{$x_i$ \emph{appears in} $w$}  \}  \]
			as the \emph{minimum} of $w$, and denote it by $\min(w)$, and we shall refer to
			\[ \max\{ i\in X\,|\, \mbox{$x_i$ \emph{appears in} $w$}  \}  \]
			as the \emph{maximum} of $w$, and denote it by $\max(w)$.
			
			The \emph{PBW basis} $\basis(X)$ of $\Lie(X)$ is defined recursively as follows (cf.~\cite{MReut}, \cite[\S 13.2.5.2]{Loday-Vallette}, and references therein):
			\begin{itemize} \item Every generator $x_i$ is in $\basis(X)$.
				\item Given a Lie word $w$ in $\Lie(X)$, which is the bracket $w=[v,v']$ of Lie words $v,v'$, then $w$ is in $\basis(X)$ if both $v$ and $v'$ are elements of $\basis(X)$, and furthermore the relations $\min(v)<\min(v')$ and $\max(v)<\max(v')$ are satisfied. 
			\end{itemize}
		\end{definition}
		
		\begin{remark} It is well-known that $\mathcal{B}(X)$ is indeed a basis of $\mathcal{L}(X)$. The fact that $\mathcal{B}(X)$ spans $\mathcal{L}(X)$ can be seen by induction on the length of words, were the inductive step depends on the Jacobi identity. In order to check linear independence, it is not restrictive to assume $(X,\leq)=(\underline{n},\leq):=\{1\leq\cdots\leq n\}$. In this case $\mathcal{L}(\underline{n})$ is a finite dimensional vector space, whose dimension is known to be $\sum_{k=1}^n\binom{n}{k}(k-1)!$, and an easy count (cf.~\cite[Proposition 3]{operad lie is free}) shows that this is also the cardinality of $\mathcal{B}(\underline{n})$.
			
		\end{remark}
		
		\begin{lemma}\label{lem:realization} The ``realization" map
			$ r: \pblT(X) \to \Lie(X)$, defined recursively by
			$$ r\left(\,\xy {\ar@{-}(0,-3);(0,1)}; {(0,3)*{_i}};\endxy \,\right)=x_i \qquad \textrm{and} \qquad r\left(\xy
			{\ar@{-}(0,-4);(0,-2)};
			{\ar@{-}(0,-2);(-4,2)};{(-4,4)*{\slab{T_1}}};
			{\ar@{-}(0,-2);(4,2)};{(4,4)*{\slab{T_2}}};
			\endxy \right) = [r(T_1),r(T_2)],$$
			is an isomorphism of Lie algebras, where the left-hand side is equipped with the commutator bracket with respect to $\rhd$.\end{lemma}
		\begin{proof} The fact that $r$ is an isomorphism of vector spaces is clear, as it sends the canonical basis of $\pblT(X)$ bijectively onto the PBW basis $\basis(X)$ of $\Lie(X)$. Let $\widetilde{\phi}:L(X)\to\pblT(X)$ be the morphism of Lie algebras defined by $\widetilde{\phi}(x_i)=\xy {\ar@{-}(0,-2);(0,2)}; {(0,4)*{_i}};\endxy$ and the universal property of the free Lie algebra $L(X)$. A straightforward induction shows that $\widetilde{\phi}(w)=0$ whenever some generator appears more than once inside the Lie word $w$. Hence, $\widetilde{\phi}$ descends to a morphism of Lie algebras $\phi:\Lie(X)\to \pblT(X)$. The thesis is proven once we show $\phi\circ r =\id$. This is obvious for trees of the form $\xy {\ar@{-}(0,-2);(0,2)}; {(0,4)*{_i}}\endxy$. Finally, given $(T,\ell)\in\pblT(X)$, with $T\, = \, \xy
			{\ar@{-}(0,-4);(0,-2)};
			{\ar@{-}(0,-2);(-4,2)};{(-4,4)*{\slab{T_1}}};
			{\ar@{-}(0,-2);(4,2)};{(4,4)*{\slab{T_2}}};
			\endxy $, we denote by $\ell_j$, $j=1,2$, the labeling on $T_j$ inherited from $\ell$. According to Remark \ref{remboh}, we have $(T,\ell)=(T_1,\ell_1)\rhd(T_2,\ell_2)=[(T_1,\ell_1),(T_2,\ell_2)]$. This implies \begin{eqnarray*}\phi\circ r(T,\ell) &=&  \phi([ r(T_1,\ell_1),r(T_2,\ell_2)]) \qquad \mbox{(by definition of $r$)} \\ &=&  [(T_1,\ell_1),(T_2,\ell_2)] =(T,\ell),
			\end{eqnarray*} 
			as $\phi$ is a morphism of Lie algebras and by induction on the number of leaves.
		\end{proof}
		
		We now turn our attention to the pre-Lie logarithm
		$$ -\log_{\rhd}\left( 1- \xy {\ar@{-}(0,-2);(0,2)}; \endxy\,\right) \in \pbT.$$
		Recall from Subsection \ref{subsection: Umbral calculus} that this can be characterized as the unique solution to the recursion
		$$ 			y = \sum_{k\geq0} (-1)^k\frac{B_k}{k!} (y\rhd\#)^k \left(\,\xy {\ar@{-}(0,-2);(0,2)}; \endxy\,\right).  $$
		\begin{definition}\label{definition: Eulerian coefficient}
			The Eulerian coefficient $E_T \in \K$ of a planar, binary, rooted tree $T\in \pbT$ is defined by the expansion
			$$-\log_{\rhd}\left( 1- \xy {\ar@{-}(0,-2);(0,2)}; \endxy\,\right) = \sum_{T} E_T\, T.$$
		\end{definition}
		
		\begin{remark}
			By the universal property of $(\T,\curvearrowright)$, there is a unique morphism 
			$$ \Psi: \T \to \pbT, \quad \bul \mapsto \xy {\ar@{-}(0,-2);(0,2)}; \endxy$$
			from the pre-Lie algebra of rooted trees, see Subsection \ref{subsection: pre-Lie}. In principle, this reduces the computation of $-\log_{\rhd}\left( 1- \xy {\ar@{-}(0,-2);(0,2)}; \endxy\,\right)$ to the computation of $-\log_{\curvearrowright}\left( 1- \bul \right)$, which was essentially (up to signs) performed in Subsection \ref{subsection: log in sT}. On the other hand, it is not easy to describe explicitly the morphism $\Psi$, or equivalently, the symmetric brace operations on $(\pbT,\rhd)$ (cf.~Remark \ref{rem: sym braces in free pl}). It will  be more convenient to apply the techniques developed in  Section \ref{section: pre-Lie} directly to this situation: this will be done in the following Subsection \ref{subsection:umbral in planar}.
		\end{remark}
		
		We consider the totally ordered set $\underline{n+1}:=\{1\le \cdots \le n+1\}$.
		As in Subsection \ref{subsection: problem}, in this case we shall denote the Lie algebra $\Lie(\underline{n+1})$ by $\Lie_{\leq n+1}$, and by $\Lie_{n+1}\subset\Lie_{\leq n+1}$ the subspace spanned by Lie words of length $(n+1)$.  We also denote by $\pi: \Lie_{\le n+1}\to \Lie_{n+1}$ the canonical projection.  In the following proposition, we consider the composition $$\pbT\xrightarrow{\operatorname{lab}_{\underline{n+1}}} \pblT(\underline{n+1})\xrightarrow{r} \Lie_{\leq n+1}\xrightarrow{\pi} \Lie_{n+1},$$
		where the first map was introduced in Lemma \ref{lem:labeling}, and the second one in Lemma \ref{lem:realization}.

		\begin{proposition}\label{proposition: Eulerian = log}
			Given $n\ge 0$, the element
			$$ (\pi \circ r\circ\operatorname{lab}_{\underline{n+1}})\left( -\log_{\rhd}\left( 1- \xy {\ar@{-}(0,-2);(0,2)}; \endxy\,\right) \right) \in \Lie_{n+1}$$
			coincides with $E(x_1\cdots x_{n+1})$.
		\end{proposition}

		\begin{proof} The thesis is clear for $n=0$, thus we can proceed by induction. We equip $\Lie_{\leq n+1}$ with the pre-Lie product $\rhd$ induced via the  isomorphism $r:\pblT(\underline{n+1})\to \Lie_{\leq n+1}$. Since $r\circ \operatorname{lab}_{\underline{n+1}}: \pbT \to \Lie_{\le n+1}$
			commutes with $\rhd$, the element
			$$X := (r\circ \operatorname{lab}_{\underline{n+1}})\left(-\log_{\rhd}\left( 1- \xy {\ar@{-}(0,-2);(0,2)}; \endxy\,\right)\right) \in \Lie_{\le n+1}$$
			satisfies the recursion
			$$ X = \sum_{k\ge 0}(-1)^k\frac{B_k}{k!}(X\rhd\#)^k(x_1+\cdots + x_{n+1}).$$
			We decompose $\Lie_{\le n+1}=\oplus_{k=1}^{n+1} \Lie_{k}$, where $\Lie_k\subset\Lie_{\leq n+1}$ is the vector subspace spanned by Lie words of length $k$. Accordingly, $X$ decomposes as $X = \sum_{k= 1}^{n+1} X_{(k)}$. The inductive hypothesis (and naturality of the Eulerian idempotent) implies that, for $k\le n$, we have
			$$ X_{(k)}= \sum_{j:(\underline{k},\le)\hookrightarrow (\underline{n+1},\le)}E(x_{j(1)}\cdots x_{j(k)}).$$
			Putting together this fact, the above recursion for $X$, and the observation that any product of Lie words $v \rhd w$ in $\Lie_{\le n+1}$ vanishes if the generator $x_{n+1}$ appears in $v$ (Remark \ref{remboh}), we finally find
			\begin{multline*}
			X_{(n+1)} =\\= \sum_{k\ge 0}(-1)^k\frac{B_k}{k!} \sum_{i_1+\cdots +i_k=n} \sum_{\sigma\in S(i_1\cdots,i_k)} E(x_{\sigma(1)}\cdots x_{\sigma(i_1)})\rhd (\cdots (E(x_{\sigma(n-i_k+1)}\cdots x_{\sigma(n)})\rhd x_{n+1})\cdots) =\\=
			\sum_{k\ge 0}(-1)^k\frac{B_k}{k!} \sum_{i_1+\cdots +i_k=n} \sum_{\sigma\in S(i_1\cdots,i_k)} [E(x_{\sigma(1)}\cdots x_{\sigma(i_1)}),\cdots [E(x_{\sigma(n-i_k+1)}\cdots x_{\sigma(n)}), x_{n+1}]\cdots].
			\end{multline*}
			Therefore, by Corollary \ref{corollary: Eulerian}, Subsection \ref{subsection: recursion for Eulerian}, $X_{(n+1)}$  coincides with $E_{n+1}(x_1\cdots x_{n+1})$.
		\end{proof}
		
		\begin{corollary}\label{corollary:independence on labeling}
			For $n\ge 1$, the expansion of $E(x_1\cdots x_n)\in \Lie_n$
			with respect to the PBW basis $\basis_n$ reads
			$$ E(x_1\cdots x_n) = \sum_{b\in\basis_n} E_b\,b=\sum_{(T,\ell)}E_T\,r(T,\ell), $$
			where the sum in the right hand side runs over the set of planar, binary, rooted trees with $n$ leaves and an admissible labeling in $(\underline{n},\leq)$.
		\end{corollary}
		
		\begin{remark} In particular, the two definitions of Eulerian coefficients, the one in \ref{definition: Eulerian coefficient} and the one in \ref{definition: Eulerian coefficient Lie}, are consistent with each other. Given a basis element $b\in\basis_n$, corresponding to $(T,\ell)\in\pblT(\underline{n})$ under the realization isomorphism $r$, the Eulerian coefficient $E_b$ coincides with the Eulerian coefficient $E_T$ of the underlying planar, binary, rooted tree.
		\end{remark}
		
		\subsection{Invariance under the specular involution}\label{subsec: specular}
		
		\begin{definition}
			We consider the anti-automorphism of the free Lie algebra $L_n$ over $x_1,\ldots,x_n$, defined on generators by $x_i\mapsto x_{n-i+1}$. This obviously descends to an anti-automorphism of the Lie algebra $\Lie_{\leq n}$, which we call the \emph{specular involution} and denote by $\Sigma: \Lie_{\leq n}\to \Lie_{\leq n}$.
		\end{definition}	
		\begin{example} For instance, $\Sigma([ x_1, [[x_2,x_3],x_4 ] ])=[[x_1,[x_2,x_3]],x_4]$.
		\end{example} 
		\begin{definition}\label{def: specular on permutations} We continue to denote by $\Sigma:S_n\to S_n$ the involution on the set of permutations defined by $\Sigma\sigma(i) = n-\sigma(i)+1$.
		\end{definition}
		The proofs of the following lemmas are easy and left to the reader.
		\begin{lemma} The specular involution $\Sigma: \Lie_{\leq n}\to \Lie_{\leq n}$ sends the PBW basis $\basis_n$ onto itself. 
		\end{lemma} 
		\begin{lemma}\label{lem: specular on permutations} Given a permutation $\sigma\in S_n$, we have $d_{\Sigma\sigma}=n-1-d_\sigma$ (recall that we denote by $d_\sigma,d_{\Sigma\sigma}$ the descent numbers of the permutations $\sigma,\Sigma\sigma$, respectively).
		\end{lemma}
		
		\begin{proposition}\label{lemma: specularity Lie}
			For any $b \in \basis_n$, the Eulerian coefficients $E_b$ and $E_{\specular{(b)}}$ are equal.
		\end{proposition}

		\begin{proof} First, we show that $E(x_1\cdots x_n)$ is a fixed point for $\Sigma$. To do so, we use the formula \eqref{equation: Eulerian2} for $E(x_1\cdots x_n)$ from the Introduction. Using this formula and Lemma \ref{lem: specular on permutations}, we see that 
			\begin{eqnarray} \nonumber \Sigma(E(x_1\cdots x_n)) &=& \Sigma\left(\frac{1}{n^2}\sum_{\sigma\in S_n}\frac{(-1)^{d_\sigma}}{\binom{n-1}{d_\sigma}}[x_{\sigma(1)},\cdots[x_{\sigma(n-1)},x_{\sigma(n)}]\cdots]\right) \\\nonumber &=&\frac{1}{n^2}\sum_{\sigma\in S_n}\frac{(-1)^{d_\sigma}}{\binom{n-1}{d_\sigma}}[\cdots [x_{\Sigma\sigma(n)},x_{\Sigma\sigma(n-1)}]\cdots ,x_{\Sigma\sigma(1)} ]\\\nonumber &=& \frac{1}{n^2}\sum_{\sigma\in S_n}\frac{(-1)^{n-1-d_\sigma}}{\binom{n-1}{d_\sigma}}[x_{\Sigma\sigma(1)},\cdots[x_{\Sigma\sigma(n-1)},x_{\Sigma\sigma(n)}]\cdots] \\ \nonumber &=& \frac{1}{n^2}\sum_{\sigma\in S_n}\frac{(-1)^{d_{\Sigma\sigma}}}{\binom{n-1}{d_{\Sigma\sigma}}}[x_{\Sigma\sigma(1)},\cdots[x_{\Sigma\sigma(n-1)},x_{\Sigma\sigma(n)}]\cdots] \\ \nonumber &=& E(x_1\cdots x_n).
			\end{eqnarray}
			
			This implies
			\[  E(x_1\cdots x_n) = \sum_{b\in\basis_n} E_b\, b = \Sigma\left(\sum_{b\in\basis_n} E_b\, b\right) = \sum_{b\in\basis_n} E_b\,\Sigma(b) = \sum_{b\in\basis_n} E_{\Sigma(b)}\, b,  \]
			and the thesis is proven.
		\end{proof}

		\begin{definition}\label{def: specular on trees}
			The specular involution 
			$\specular: \pbT \stackrel{\cong}{\longrightarrow} \pbT$ on the vector space $\pbT$ 
			is defined inductively by
			$$ \Sigma\left(\, \xy
			{\ar@{-}(0,-2);(0,2)};
			\endxy \, \right)\, =  \, \xy
			{\ar@{-}(0,-2);(0,2)};
			\endxy \qquad \qquad \mathrm{and} \qquad \qquad \Sigma\left(\xy
			{\ar@{-}(0,-4);(0,-2)};
			{\ar@{-}(0,-2);(-4,2)};{(-4,4)*{\slab{T_1}}};
			{\ar@{-}(0,-2);(4,2)};{(4,4)*{\slab{T_2}}};
			\endxy\right)  \, = \, \xy
			{\ar@{-}(0,-4);(0,-2)};
			{\ar@{-}(0,-2);(-4,2)};{(-6,4)*{\slab{\Sigma(T_2)}}};
			{\ar@{-}(0,-2);(4,2)};{(6,4)*{\slab{\Sigma(T_1)}}};
			\endxy,$$
			for any $T_1$, $T_1 \in \pbT$.
		\end{definition}
		\begin{remark} In other words, $\Sigma(T)$ is the reflection of the planar, binary, rooted tree $T$ with respect to a non-intersecting vertical axis (hence the name!), as illustrated in the following picture 
			
			
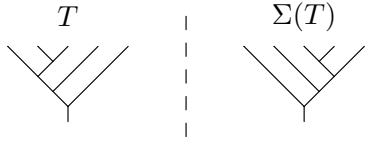
\[  \xy {\ar@{-}(0,-4);(0,-2)};
			{\ar@{-}(0,-2);(-8,6)};
			{\ar@{-}(-4,2);(0,6)};
			{\ar@{-}(0,-2);(8,6) };
			{\ar@{-}(-2,0);(4,6)}; 
			{\ar@{-}(-2,4);(-4,6)};	
			{(0,10)*{T}}
			\endxy \qquad \xy {\ar@{--}(0,-6);(0,10)};\endxy \qquad \xy {\ar@{-}(0,-4);(0,-2)};
			{\ar@{-}(0,-2);(-8,6)};
			{\ar@{-}(4,2);(0,6)};
			{\ar@{-}(0,-2);(8,6) };
			{\ar@{-}(2,0);(-4,6)}; 
			{\ar@{-}(2,4);(4,6)};	{(0,10)*{\Sigma(T)}}
			\endxy\]
			
		\end{remark}
		
		It is straightforward to see that the specular involution on $\pbT$ and the one on $\Lie_{\leq n}$ are compatible under the map 
		$$ r\circ \operatorname{lab}_{\underline{n}}: \pbT \to \Lie_{\le n}.$$
		from Proposition \ref{proposition: Eulerian = log}, that is, $\specular \circ r\circ \operatorname{lab}_{\underline{n}}= r\circ\operatorname{lab}_{\underline{n}} \circ \specular: \, \pbT \to \Lie_{\le n}$. This, together with the above Proposition \ref{lemma: specularity Lie} and Corollary \ref{corollary:independence on labeling} from the previous subsection, proves the following
		\begin{corollary}\label{corollary: specularity in pbT}
			The Eulerian coefficient of a (planar, binary, rooted) tree and the one of its specular coincide, that is, $E_T=E_{\Sigma(T)}$ for all $T\in\pbT$. 
		\end{corollary}

		\subsection{Computing the Eulerian coefficients via umbral calculus}\label{subsection:umbral in planar}
		
		In this subsection, we apply the umbral calculus developed in Section \ref{section: pre-Lie} to compute the pre-Lie logarithm $-\log_\rhd\left(1-\xy{\ar@{-}(0,-2);(0,2)}\endxy\,\right)$ in the pre-Lie algebra of planar, binary, rooted trees. As in Proposition \ref{prop:recvsdiffeq}, this is done by solving the differential equation 
		\begin{equation}\label{eq:differential equation planar}
		\left\{\begin{array}{l} P ' = \left<\left. \frac{D}{1-e^{-D}}\right|P\right>\rhd P \\
		P(0) = \,\xy{\ar@{-}(0,-2);(0,2)}\endxy
		\end{array}  \right.
		\end{equation}
		in the pre-Lie algebra $(\pbT[t],\rhd)$ of polynomials with coefficients in $\pbT$.  In the remainder of this paper, given a planar, binary, rooted tree $T$, we shall denote by $P(T)(t)\in\K[t]$ the coefficient of the solution $P\in\pbT[t]$ to \eqref{eq:differential equation planar} in the expansion with respect to the canonical basis of $\pbT$. We can recover the Eulerian coefficient $E_T$ from the polynomial $P(T)(t)$ via the identity
		\[ E_T = \left<\left. \frac{D}{1-e^{-D}}\right|P(T)(t)\right>  \]	
		Our main result is an analog of Theorem \ref{th: product rule} in this context.
		\begin{theorem}\label{th: product rule planar} The polynomials $P(T)(t)$ are determined by $P \left( \, \xy{\ar@{-}(0,-2);(0,2)}\endxy \,\right)(t) = 1$ and the following recursion:
			\begin{itemize}
				
				\item[(I)] Given a tree of the form $T=\xy  {\ar@{-}(0,-4);(0,-2)};
				{\ar@{-}(0,-2);(-3,1)};{(-3,3)*{\slab{T_1}}};
				{\ar@{-}(0,-2);(10,8)};
				{\ar@{-}(6,4);(3,7) };   {(3,9)*{\slab{T_k}}};
				{\ar@{.}(-1,1);(3,5)}
				\endxy  $ (cf.~Appendix \ref{appendix:magmatic} for the notation), we have
				\[ P(T)(t) =\prod_{i=1}^k P\left(  \xy  {\ar@{-}(0,-4);(0,-2)};
				{\ar@{-}(0,-2);(-3,1)};{(-3,3)*{\slab{T_i}}};
				{\ar@{-}(0,-2);(4,2)};
				\endxy  \right)(t) \]
				\item[(II)] For all trees $T\in\pbT$, we have
				\[ P\left(\xy  {\ar@{-}(0,-4);(0,-2)};
				{\ar@{-}(0,-2);(-3,1)};{(-3,3)*{\slab{T}}};
				{\ar@{-}(0,-2);(4,2)};
				\endxy\right)(t) = \sum_{\tau=0}^{t-1} P(\Sigma(T))(-\tau), \]
				where $\sum_{\tau=0}^{t-1}\colon \K[\tau]\to \K[t]$ is the indefinite sum operator, and $\Sigma:\pbT\to\pbT$ is the specular involution from the previous subsection.
			\end{itemize}
			
		\end{theorem}
		\begin{proof} We follow closely the proof of Theorem \ref{th: product rule}. In particular, our proof will rely on the following key lemma, analog to Lemma \ref{prop:coproduct recursion}, which is a straightforward consequence of the differential equation \eqref{eq:differential equation planar}. 
			\begin{lemma}\label{lem:recuplanar} Using the notations from Remark \ref{rem: pre-lie copr}, the polynomials $P(T)(t)$ are determined recursively by  $P \left( \, \xy{\ar@{-}(0,-2);(0,2)}\endxy \,\right)(t) = 1$, and for $T\neq \,\xy{\ar@{-}(0,-2);(0,2)}\endxy$ by
				\[ P(T)(t) = \sum_{v\in V(T)} \varepsilon(v) E_{T_v}\int_0^t P\left(\widehat{T_v}\right)(\tau)d\tau.\]
			\end{lemma}
			
			For notational convenience, we shall introduce a linear map \[ Z:\pbT\to\pbT:T\to Z(T):= \xy  {\ar@{-}(0,-4);(0,-2)};
			{\ar@{-}(0,-2);(-3,1)};{(-3,3)*{\slab{T}}};
			{\ar@{-}(0,-2);(4,2)};
			\endxy.\] 
			
			\begin{remark}\label{rem:planar} The main difference between the current proof and the one of Theorem \ref{th: product rule} is linked to the asymmetry between left and right in the formula \eqref{pre-lie copr}. More precisely, the problem arises when we consider the vertex $v_\ast$  connected to the root of $T$. In this case, by definition $T_{v_\ast}$ is the subtree to the left of $v_\ast$. On the other hand, in the tree $Z(T)$ the vertex $v_\ast$ has become a left child, hence, always according to the definitions, $Z(T)_{v_\ast}$ will be the subtree to the right of $v$. For any inner vertex $v$ of $T$ different from $v_\ast$, conversely, the subtrees $T_v$ and $Z(T)_v$ coincide. We handle this asymmetry by using the specular involution $\Sigma$. A direct computation shows that, for all $T\in\pbT$, the folowing relation holds
				\[ \Delta\circ Z (T) = T\otimes\,\xy{\ar@{-}(0,-2);(0,2)}\endxy - (\Sigma\otimes (Z\circ\Sigma))\circ \Delta\circ\Sigma(T). \] 
			\end{remark}
			
			We focus on the proof of Item (II) in the claim of the theorem. When $T=\, \xy{\ar@{-}(0,-2);(0,2)}\endxy$ , $Z(T)=\, \xy{\ar@{-}(0,-2);(0,0)};{\ar@{-}(0,0);(-2,2)};{\ar@{-}(0,0);(2,2)}\endxy$, the thesis, that is, $P\left(\xy{\ar@{-}(0,-2);(0,0)};{\ar@{-}(0,0);(-2,2)};{\ar@{-}(0,0);(2,2)}\endxy\right)(t)=t$, can be checked directly using Lemma \ref{lem:recuplanar}, thus we may assume $T\neq\, \xy{\ar@{-}(0,-2);(0,2)}\endxy$ . To simplify the notations, we shall write $\Sigma T$ instead of $\Sigma(T)$. By the previous remark and Lemma \ref{lem:recuplanar}, we have
			\begin{equation*}
			P'(Z(T))(t) = E_T - \sum_{v\in V(\Sigma T)}\varepsilon(v)\, E_{\Sigma(\Sigma T_v)} P\left(Z\left(\Sigma\left(\widehat{\Sigma T_v}\right)\right)\right)(t).
			\end{equation*}
			Using induction on the number of leaves and Corollary \ref{corollary: specularity in pbT}, this becomes
			\begin{equation*}
			P'(Z(T))(t)= E_{\Sigma T} - \sum_{v\in V(\Sigma T)}\varepsilon(v)\,E_{\Sigma T_v} \sum_{\tau=0}^{t-1}P\left(\widehat{\Sigma T_v}\right)(-\tau).
			\end{equation*}
			Now we are in a position to repeat the computation from \ref{th: product rule}. Denoting by $c(\Sigma T, k)$ the coefficients of the polynomial $P(\Sigma T)(t) =\sum_{k=0}^{|T|-1} c(\Sigma T,k) t^k$, we have 
			\[ E_{\Sigma T} = \sum_{k=0}^{|T|-1} (-1)^k\,c(\Sigma T, k)\,B_k, \]
			and furthermore, by another application of Lemma \ref{lem:recuplanar}, we see that 
			\[ c(\Sigma T,0)=0\quad\mbox{for $T\neq\, \xy{\ar@{-}(0,-2);(0,2)}\endxy$ },\qquad c(\Sigma T,k) = \frac{1}{k}\sum_{v\in V(\Sigma T)}\varepsilon(v)\, E_{\Sigma T_v}\,c\left(\widehat{\Sigma T}_v,k-1\right)\quad\mbox{if $k\geq1$}.  \]
			Putting together the above identities, we can finally compute 
			
			\begin{eqnarray*}
				P'(Z(T))(t) &=& E_{\Sigma T} - \sum_{v\in V(\Sigma T)}\varepsilon(v)\,E_{\Sigma T_v} \sum_{\tau=0}^{t-1}P\left(\widehat{\Sigma T_v}\right)(-\tau) \\ &=&\sum_{k=1}^{|T|-1}\left((-1)^k c(\Sigma T,k) B_k - \sum_{v\in V(\Sigma T)}\varepsilon(v) E_{\Sigma T_v}\sum_{\tau=0}^{t-1} c\left(\widehat{\Sigma T}_v,k-1\right) (-\tau)^{k-1} \right) \\&=& \sum_{k=1}^{|T|-1}(-1)^k\left( c(\Sigma T,k) B_k + \frac{1}{k}\sum_{v\in V(\Sigma T)}\varepsilon(v)\, E_{\Sigma T_v}\,c\left(\widehat{\Sigma T}_v,k-1\right) (B_{k}(t)- B_{k}) \right)\\&=& \sum_{k=1}^{|T|-1}(-1)^k\left( c(\Sigma T,k) B_k + c(\Sigma T,k) (B_k(t)-B_k) \right) = \sum_{k=1}^{|T|-1}(-1)^k c(\Sigma T,k) B_k(t)\\
				&=& \frac{d}{dt}\left( \sum_{k=0}^{|T|-1} (-1)^k c(\Sigma T,k)\frac{B_{k+1}(t)-B_{k+1}}{k+1}\right)=\frac{d}{dt}\left(\sum_{\tau=0}^{t-1}P(\Sigma T)(-\tau)\right),
			\end{eqnarray*}
			which concludes the proof of Item (II).
			
			Taking into account the subtleties illustrated in Remark \ref{rem:planar}, the proof of Item (I) is similar to the proof of the corresponding item in Theorem \ref{th: product rule}. Namely, given the tree $T=\xy  {\ar@{-}(0,-4);(0,-2)};
			{\ar@{-}(0,-2);(-3,1)};{(-3,3)*{\slab{T_1}}};
			{\ar@{-}(0,-2);(10,8)};
			{\ar@{-}(6,4);(3,7) };   {(3,9)*{\slab{T_k}}};
			{\ar@{.}(-1,1);(3,5)}
			\endxy  $, one shows  
			\[ P'(T)(t) = \sum_{i=1}^k P'(Z(T_i)) P(Z(T_1))\cdots\widehat{P(Z(T_i))}\cdots P(Z(T_k)), \]
			using induction on the number of leaves and Lemma \ref{lem:recuplanar}. Details are left to the reader.
		\end{proof}
		\subsection{The $\specular$-twisted rotation correspondence}
		\label{subsection: from pbT to pT}

		We can further improve our invariance results for the Eulerian coefficients (Corollary \ref{corollary:independence on labeling} and Corollary \ref{corollary: specularity in pbT}) with the help of an alternative presentation of $\pbT$, which relies on \emph{Knuth's rotation correspondence}, cf.~\cite{Knuth,Ebrahimi-Fard-Manchon}. This is a bijective correspondence between the set of planar, binary, rooted trees with $n$ leaves and the set of planar, non-binary (rather, not necessarily binary) rooted trees with $n$ vertices. We consider a slight variation of Knuth's rotation correspondence, which makes also use of the specular involution $\Sigma$ from Subsection \ref{subsec: specular}.

		\begin{definition}
			We denote by $\pT(n)$ be the vector space spanned by all planar, rooted trees with $n$ vertices, and by $\pT:=\prod_{n\geq1}\pT(n)$. 
			
			The \emph{$\specular$-twisted rotation correspondence} $\Phi:\pbT \stackrel{\cong}{\to} \pT$ is defined recursively by
			\[ \Phi(\, \xy
			{\ar@{-}(0,-2);(0,2)};
			\endxy \, ) = \bul \quad \qquad \textrm{and } \quad \qquad \Phi\left(\xy  {\ar@{-}(0,-7);(0,-5)};
			{\ar@{-}(0,-5);(-3,-2)};{(-3,0)*{\slab{T_1}}};
			{\ar@{-}(0,-5);(10,5)};
			{\ar@{-}(6,1);(3,4) };   {(3,6)*{\slab{T_k}}};
			{\ar@{.}(-1,-2);(3,2)}
			\endxy \right) \,  = \,  
			\xy 
			{\ar@{-}(0,-6)*{\bul};(-10,4)};{(-12,6)*{\slab{\Phi(\Sigma ({T}_1))}}};
			{\ar@{-}(0,-6)*{\bul};(10,4)};{(12,6)*{\slab{\Phi(\Sigma (T_k))} }}; 
			{\ar@{.} (-3.5,6);(3.5,6) };
			{\ar@{.} (-4,0);(4,0) };
			\endxy
			\]
		\end{definition}
		\begin{remark}\label{rem: rotation correspondence} In other words, we may describe $\Phi(T)$ in terms of the branches of the planar, binary rooted tree $T$. The rightmost branch of $T$ corresponds to the root of $\Phi(T)$. Given a branch $b$ of $T$, corresponding to a vertex $v$ in $\Phi(T)$, the branches of $T$ having their root in $b$ correspond to the children of $v$ in $\Phi(T)$, with the branch closest to the root of $b$ corresponding to the leftmost child of $v$, and the one farthest from the root of $b$ corresponding to the rightmost child of $v$.

			Below we depict some examples of how $\Phi$ works:
			\[ \Phi\left( \xy
			{\ar@{-}(0,-2);(0,0)};
			{\ar@{-}(0,0);(-2,2)};
			{\ar@{-}(0,0);(2,2)};
			\endxy \right) = \xy
			{\ar@{-}(0,-2)*{\bul};(0,2)*{\circ}};
			\endxy,\qquad \Phi\left( \xy
			{\ar@{-}(0,-3);(0,-1)};
			{\ar@{-}(0,-1);(4,3)};
			{\ar@{-}(0,-1);(-4,3)};
			{\ar@{-}(-2,1);(0,3)};
			\endxy \right)\,=\,\,\xy
			{\ar@{-}(0,-4)*{\bul};(0,0)*{\circ}};
			{\ar@{-}(0,0)*{\circ};(0,4)*{\circ}};
			\endxy,\qquad\Phi\left( \xy
			{\ar@{-}(0,-3);(0,-1)};
			{\ar@{-}(0,-1);(4,3)};
			{\ar@{-}(0,-1);(-4,3)};
			{\ar@{-}(2,1);(0,3)};
			\endxy \right)\,=\,\xy
			{\ar@{-}(0,-2)*{\bul};(-4,2)*{\circ}};
			{\ar@{-}(0,-2)*{\bul};(4,2)*{\circ}};
			\endxy,\qquad\Phi\left( \xy {\ar@{-}(0,-4);(0,-2)};
			{\ar@{-}(0,-2);(-6,4)};
			{\ar@{-}(-4,2);(-2,4)};
			{\ar@{-}(0,-2);(6,4) };
			{\ar@{-}(4,2);(2,4)}; 	
			\endxy \right)\,=\,\xy
			{\ar@{-}(0,-4)*{\bul};(-4,0)*{\circ}};
			{\ar@{-}(0,-4)*{\bul};(4,0)*{\circ}};
			{\ar@{-}(-4,0)*{\circ};(-4,4)*{\circ}};
			\endxy \]
			\[ \Phi\left( \xy {\ar@{-}(0,-4);(0,-2)};
			{\ar@{-}(0,-2);(-8,6)};
			{\ar@{-}(-4,2);(0,6)};
			{\ar@{-}(0,-2);(8,6) };
			{\ar@{-}(-2,0);(4,6)}; 
			{\ar@{-}(-2,4);(-4,6)};	
			\endxy\right)\,=\,\,\, \xy 
			{\ar@{-} (0,-5)*{\bul}; (0,-1)*{\circ}}; 
			{\ar@{-} (0,-1)*{\circ}; (-4,3)*{\circ}}; 
			{\ar@{-} (0,-1)*{\circ}; (4,3)*{\circ}}; 
			{\ar@{-} (4,3)*{\circ}; (4,7)*{\circ}};
			\endxy,\qquad\qquad\Phi\left(\xy {\ar@{-}(0,-4);(0,-2)};
			{\ar@{-}(0,-2);(-8,6)};
			{\ar@{-}(4,2);(0,6)};
			{\ar@{-}(0,-2);(8,6) };
			{\ar@{-}(2,0);(-4,6)}; 
			{\ar@{-}(2,4);(4,6)};	
			\endxy\right)\,=\, \xy 
			{\ar@{-} (0,-3)*{\bul}; (-4,1)*{\circ}}; 
			{\ar@{-} (0,-3)*{\bul}; (0,1)*{\circ}}; 
			{\ar@{-} (0,-3)*{\bul}; (4,1)*{\circ}}; 
			{\ar@{-} (4,1)*{\circ}; (4,5)*{\circ}};
			\endxy.\]

			In particular, it is clear that $\Phi$ is a bijective correspondence from the set of planar, binary, rooted trees with $n$ leaves to the set of planar rooted trees with $n$ vertices. In fact, using the above description of $\Phi:\pbT(n)\to\pT(n)$, it is straightforward to construct the inverse $\Phi^{-1}:\pT(n)\to\pbT(n)$.   	
			
		\end{remark}
		
		\begin{remark} Using the isomorphism $\Phi$, we can associate an Eulerian coefficient $E_T$ and a polynomial $P(T)(t)$ to every planar rooted tree $T\in\pT$.
		\end{remark}

		In the context of planar rooted trees, the recursion for the polynomial, given in the previous Theorem \ref{th: product rule planar}, becomes very similar to the one from Theorem \ref{th: difference rule} (but with a twist involving signs). More precisely, Theorem \ref{th: product rule planar} translates into the following result.

		\begin{theorem}\label{th:recursion planar} The polynomials $P(T)(t)$, $T\in\pT$, are defined by $P(\bul)(t)=1$  and the following recursion: for all planar rooted trees $T,T_1,\ldots,T_k\in\pT$, $k\geq1$,
			
			\[ P\left(\xy {\ar@{-}(0,-4)*{\bul};(-6,2)};{(-6,4)*{\slab{T_1}}};
			{\ar@{-}(0,-4)*{\bul};(6,2)};{(6,4)*{\slab{T_k}}};{\ar@{.}(-3,4);(3,4)};{\ar@{.}(-2.5,0);(2.5,0)}   \endxy\right)(t) = \prod_{j=1}^k P\left(\xy {\ar@{-}(0,-3)*{\bul};(0,1)};{(0,3)*{\slab{T_j}}};
			\endxy\right)(t),   \] 
			\[ P\left(\xy {\ar@{-}(0,-3)*{\bul};(0,1)};{(0,3)*{\slab{T}}};
			\endxy\right)(t) = \sum_{\tau=0}^{t-1} P(T)(-\tau).\]
		\end{theorem}
		The above recursion, together with a straightforward induction, implies the following
		\begin{corollary}\label{lemma: Euler planar structure}
			Given a planar rooted tree $T$,	the polynomial $P(T)(t)$, thus also the Eulerian coefficient $E_{T}$, is independent of the planar structure on $T$ (that is, it only depends on the underlying rooted tree).
		\end{corollary}		
		
		When we combine the previous corollary with the one \ref{corollary: specularity in pbT}, we reach the following surprising conclusion.
		
		\begin{corollary}\label{cor: independence on root} Given a planar rooted tree $T$, the Eulerian coefficient $E_{T}$ is further independent of the location of the root of $T$, that is, it only depends on the underlying tree (in the sense of graph theory, i.e., a connected graph with no cycles).
		\end{corollary}
		\begin{proof} To prove the corollary, we have to understand how specularity behaves under the rotation correspondence $\Phi$. Let us denote by $\widetilde{\specular}$ the map $\widetilde{\specular}:=\Phi\circ\specular\circ\Phi^{-1}:\pT\to\pT$. Given a planar rooted tree $T$, we claim that $T$ and $\widetilde{\Sigma}(T)$ have the same underlying tree, but the root of $\widetilde{\Sigma}(T)$ corresponds to the leftmost child of the root of $T$.
			More precisely, if the tree $T$ has the following form, for certain planar rooted trees $T_1 = \xy {\ar@{-}(0,-4)*{\bul};(-6,2)};{(-6,4)*{\slab{T_1^1}}};
			{\ar@{-}(0,-4)*{\bul};(6,2)};{(6,4)*{\slab{T_1^j}}};{\ar@{.}(-3,4);(3,4)};{\ar@{.}(-2.5,0);(2.5,0)}   \endxy$, $T_2,\ldots, T_k$, $k\ge1$, $r\ge0$, 
			$$
			T \, \, = \, \xy
			{\ar@{-}(0,-6)*{\bul};(-12,2)*{\circ}};
			{\ar@{-}(-12,2)*{\circ};(-20,8)*{\slab{T_1^1}}};
			{\ar@{-}(-12,2)*{\circ};(-4,8)*{\slab{T_1^j}}};
			{\ar@{-}(0,-6)*{\bul};(0,4)*{\slab{T_2}}};
			{\ar@{-}(0,-6)*{\bul};(18,4)*{\slab{T_k}}}; 
			{\ar@{.}(3,5);(14,5)};
			{\ar@{.}(2,0);(8,0)};
			{\ar@{.}(-17,9);(-8,9)};
			{\ar@{.}(-14.5,5);(-9.5,5)};
			
			\endxy
			$$
			then the tree $\widetilde{\Sigma}(T)$ is
			$$
			\widetilde{\Sigma}(T) \, \, = \, 
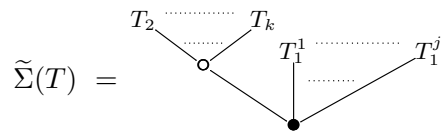
\xy
			{\ar@{-}(0,-6)*{\bul};(-12,2)*{\circ}};
			{\ar@{-}(-12,2)*{\circ};(-20,8)*{\slab{T_2}}};
			{\ar@{-}(-12,2)*{\circ};(-4,8)*{\slab{T_k}}};
			{\ar@{-}(0,-6)*{\bul};(0,4)*{\slab{T_1^1}}};
			{\ar@{-}(0,-6)*{\bul};(18,4)*{\slab{T_1^j}}}; 
			{\ar@{.}(3,5);(14,5)};
			{\ar@{.}(2,0);(8,0)};
			{\ar@{.}(-17,9);(-8,9)};
			{\ar@{.}(-14.5,5);(-9.5,5)};
			\endxy
			$$
			The claim is checked by a direct inspection, using the definition of $\Phi$ (cf.~also the discussion, as well as the figures, in Remark \ref{rem:planar}).	
			
			Given two planar rooted trees having the same underlying tree, it is clear that we can turn one into the other by a sequence of moves which are: \begin{itemize} \item either a change in planar structure, leaving the underlying rooted tree unchanged, or
				\item  an application of the map $\widetilde{\specular}$ described above.
			\end{itemize}	
			By Corollary \ref{lemma: Euler planar structure} and Corollary \ref{corollary: specularity in pbT}, neither of these moves will change the associated Eulerian coefficient.
		\end{proof}
		
		\begin{remark} Let us stress that the number of (non-planar, non-rooted) trees with exactly $n$ vertices (which is sequence A000055
			in the OEIS1\footnote{Available at the following link: \url{https://oeis.org/A000055}.}), is much smaller than the cardinality $|\basis_n| = (n - 1)!$. This greatly reduces the number of necessary computations to obtain the expansion of $E(x_1\cdots x_n)$ in the PBW basis. A table of Eulerian coefficients for trees with $n\leq8$ vertices can be found in Appendix \ref{appendix: table}. \end{remark}
		
		\begin{remark}\label{rem:bernoulli identities} We remark that, conversely, the polynomial associated to a planar rooted tree (is independent on the planar structure, according to Corollary \ref{lemma: Euler planar structure}, but) \emph{depends} on the root. In particular, Corollary \ref{cor: independence on root} implies a plethora of identities involving Bernoulli numbers. The simplest instance of this assertion, happens when we consider a corolla $C_{n+1}$ with $n$ leaves, together with its `` specular '' $\widetilde{\Sigma}(C_{n+1})$ (where the map $\widetilde{\Sigma}:\pT\to\pT$ was introduced in the proof of Corollary \ref{cor: independence on root})
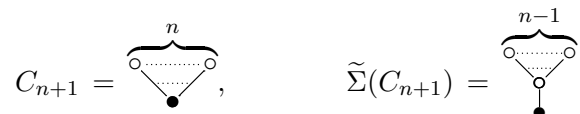
			\[ C_{n+1}\,=\,\overbrace{\xy 
				{\ar@{-}(0,-2.5)*{\bul};(-5,2.5)*{\circ}};
				{\ar@{-}(0,-2.5)*{\bul};(5,2.5)*{\circ}};
				{\ar@{.}(-3.5,2.5);(3.5,2.5)};
				{\ar@{.}(-2,0);(2,0)};
				\endxy}^n,\qquad\qquad \widetilde{\Sigma}(C_{n+1})\,=\, \overbrace{\xy 
				{\ar@{-}(0,-4)*{\bul};(0,0)*{\circ}  };
				{\ar@{-}(0,0)*{\circ};(-4,4)*{\circ}};
				{\ar@{-}(0,0)*{\circ};(4,4)*{\circ}};
				{\ar@{.}(-3,4);(3,4)};
				{\ar@{.}(-1.5,2);(1.5,2)};
				\endxy}^{n-1}  \]
			Using Theorem \ref{th:recursion planar}, one immediately finds that the associated polynomials are $P(C_{n+1})(t)=t^n$ and $P(\widetilde{\Sigma}(C_{n+1}))=(-1)^{n-1}\frac{B_n(t)-B_n}{n}$. Applying $\left<\left. \frac{D}{1-e^{-D}}\right|\#\right>$ to both polynomials, and equating the results, we finally recover the following well-known identity, due to Euler
			\[ - B_n = \frac{1}{n}\sum_{k=1}^n (-1)^k  \binom{n}{k}B_{n-k}B_k.  \]
			As another example, we consider the trees
			\[ T \,=\, 
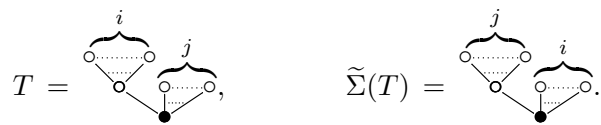
\xy {\ar@{-}(0,-4)*{\bul};(-6,0)*{\circ}};
			{\ar@{-}(-6,0)*{\circ};(-10,4)*{\circ}};
			{\ar@{-}(-6,0)*{\circ};(-2,4)*{\circ}};
			{\ar@{-}(0,-4)*{\bul};(0,0)*{\circ}};
			{\ar@{-}(0,-4)*{\bul};(6,0)*{\circ}}; 
			{\ar@{.}(-9,4);(-3,4)};
			{\ar@{.}(-7.5,2);(-4.5,2)};
			{\ar@{.}(1,0);(5,0)};
			{\ar@{.}(0.5,-2);(2.5,-2)};
			{(-6,8)*{\overbrace{\qquad}^i}};
			{(3,4)*{\overbrace{\qquad}^j}};
			\endxy,\qquad\qquad \widetilde{\specular}(T) \,=\, \xy {\ar@{-}(0,-4)*{\bul};(-6,0)*{\circ}};
			{\ar@{-}(-6,0)*{\circ};(-10,4)*{\circ}};
			{\ar@{-}(-6,0)*{\circ};(-2,4)*{\circ}};
			{\ar@{-}(0,-4)*{\bul};(0,0)*{\circ}};
			{\ar@{-}(0,-4)*{\bul};(6,0)*{\circ}}; 
			{\ar@{.}(-9,4);(-3,4)};
			{\ar@{.}(-7.5,2);(-4.5,2)};
			{\ar@{.}(1,0);(5,0)};
			{\ar@{.}(0.5,-2);(2.5,-2)};
			{(-6,8)*{\overbrace{\qquad}^j}};
			{(3,4)*{\overbrace{\qquad}^i}};
			\endxy.  \]
			Using Theorem \ref{th:recursion planar}, the associated polynomials are $P(T)(t) = \frac{(-1)^i}{i+1} t^j(B_{i+1}(t)-B_{i+1})$ and $P(\widetilde{\Sigma}(T))(t) = \frac{(-1)^j}{j+1} t^i(B_{j+1}(t)-B_{j+1})$. Equating the corresponding Eulerian coefficients, we obtain the identity
			\[ \frac{1}{i+1} \sum_{k=1}^{i+1}(-1)^k \binom{i+1}{k} B_{i+1-k}B_{j+k} =  \frac{1}{j+1} \sum_{k=1}^{j+1}(-1)^k \binom{j+1}{k} B_{j+1-k}B_{i+k}.        \]
		\end{remark}
		As a final corollary to Theorem \ref{th:recursion planar}, we notice that 
		\begin{corollary}\label{cor:coefficients of tall trees} For the tall tree $T_{n+1} =  
			\left.  \xy
			{\ar@{-}(0,-8)*{\bul};(0,-3)*{\circ}};
			{\ar@{.}(0,-3)*{\circ};(0,2)*{\circ}};
			{\ar@{-}(0,2)*{\circ};(0,7)*{\circ}};
			\endxy \right\}\mbox{\scriptsize $n+1$} $ with $(n+1)$ vertices, the associated polynomial and Eulerian coefficient are (where $(t)^k:= t(t+1)\cdots (t+k-1)$ is the rising factorial, $(t)^0:=1$)
			\[ P(T_{n+1})(t) = \frac{(t)^{\lceil \frac{n}{2}\rceil}(1-t)^{\lfloor \frac{n}{2}\rfloor} }{n!},\qquad\qquad E_{T_{n+1}} = \frac{\lfloor \frac{n}{2}\rfloor! \lceil\frac{n}{2}\rceil!}{(n+1)!}.  \]   
		\end{corollary} 
		\begin{example} For instance, \[ P(T_1)(t) = 1,\qquad P(T_2)(t)=t,\qquad P(T_3)(t)  = \frac{1}{2} t(1-t),\qquad P(T_4)(t)=\frac{1}{6} t(1-t)(1+t),  \]
			\[  P(T_5)(t)=\frac{1}{24} t(1-t)(1+t)(2-t),\qquad  P(T_6)(t)=\frac{1}{120} t(1-t)(1+t)(2-t)(2+t),\quad \ldots  \]
			\[E_{T_1}=1,\quad E_{T_2}=\frac{1}{2},\quad E_{T_3}=\frac{1}{6},\quad E_{T_4}=\frac{1}{12},\quad E_{T_5}=\frac{1}{30},\quad E_{T_6}=\frac{1}{60},\quad\ldots \]
		\end{example}
		\begin{proof} The statement about the polynomials follows inductively from \ref{th: product rule planar}, by showing that $\overrightarrow{\Delta}P(T_{n+1})(t)=P(T_n)(-t)$, where the $P(T_{n+1})(t)$ are as in the claim of the corollary and $\overrightarrow{\Delta}:\K[t]\to\K[t]$, $\overrightarrow{\Delta}p(t):=p(t+1)-p(t)$, is the forward difference operator. This is a direct computation, left to the reader (it is convenient to consider the cases $n=2m$ and $n=2m+1$ separately).
			
			To prove the claim about the coefficients, we look at the coefficient of $t$ in the polynomial \[P(T_{2m+2})(t)=\frac{1}{(2m+1)!}t(1-t)(1+t)\cdots(m-t)(m+t)=\frac{1}{(2m+1)!}t(1-t^2)\cdots(m^2-t^2).\]
By a straightforward computation, this is $\frac{(m!)^2}{(2m+1)!}$. On the other hand, Lemma \ref{lem:recuplanar} (or rather, its translation in the context of planar rooted trees) implies that this coefficient is precisely the Eulerian coefficient of the tall tree $T_{2m+1}$\footnote{More generally, Lemma \ref{lem:recuplanar} implies that the coefficient of $t$ in the polynomial $P\left(\xy {\ar@{-}(0,-3)*{\bul};(0,1)};{(0,3)*{\slab{T}}};
				\endxy\right)(t)$ is precisely $E_T$: this follows from the fact that $P(\widehat{T_v})(t)$ is either $1$ if $\widehat{T_v}=\xy {\ar@{-}(0,-2);(0,2)}\endxy$ or it has no constant term if $\widehat{T_v}\neq \xy {\ar@{-}(0,-2);(0,2)}\endxy\,$.}, thus $E_{T_{2m+1}}=\frac{m!m!}{(2m+1)!}$, as desired. On the other hand, the above computation of $P(T_{2m+2})(t)$	shows that only odd powers of $t$ appear in it, and since the odd Bernoulli numbers with the exception of $B_1$ vanish, the Eulerian coefficient $E_{T_{2m+2}}=\left<\left. \frac{D}{1-e^{-D}}\right|P(T_{2m+2})(t)\right>$ is given by
			\[ E_{T_{2m+2}}= \frac{1}{2}\times\frac{m! m!}{(2m+1)!}=\frac{m!(m+1)!}{(2m+2)!}. \]
		\end{proof}
		
		We conclude this section by giving a combinatorial interpretation of the polynomial $P(T)(t)$ associated to a planar (although, as we know, the polynomial is independent of the planar structure) rooted tree $T\in\pT$, in the spirit of Theorem \ref{th: combinatorial} from Subsection \ref{subsection: log in sT}.
		
		\begin{definition}\label{def:alternating decoration}
			We call a decreasing decoration $l$ of $T$ (Definition \ref{def:decreasing decoration}) {\em alternating}
			if for all edges of $T$ the labels of the two endpoints have opposite parity.
			
			We denote the set of decreasing, alternating decorations of $T$, which associate the number $i$ to the root, by $\Adm_\pol(T,i)$.
		\end{definition}
		
		\begin{example}
			The following two decreasing decorations are alternating,
			$$
			
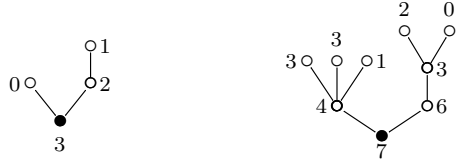
\xy     {(-6,1)*{_0}};
			{(0,-7)*{_3}};
			{(6,1)*{_2}};
			{(6,6)*{_1}};
			{\ar@{-}(0,-4)*{\bul};(-4,1)*{\circ}};
			{\ar@{-}(0,-4)*{\circ};(4,1)*{\circ}};
			{\ar@{-}(4,1)*{\circ};(4,6)*{\circ}};
			\endxy \qquad \qquad \qquad
			\xy
			{(0,-8)*{_7}};
			{(-8,-2)*{_4}};
			{(-12,4)*{_3}};
			{(-6,7)*{_3}};
			{(0,4)*{_1}};
			{(8,3)*{_3}};
			{(8,-2)*{_6}};
			{(3,11)*{_2}};
			{(9,11)*{_0}};		
			{\ar@{-}(0,-6)*{\bul};(-6,-2)*{\circ}};
			{\ar@{-}(0,-6)*{\bul};(6,-2)*{\circ}};
			{\ar@{-}(-6,-2)*{\circ};(-6,4)*{\circ}};
			{\ar@{-}(-6,-2)*{\circ};(-10,4)*{\circ}};
			{\ar@{-}(-6,-2)*{\circ};(-2,4)*{\circ}};
			{\ar@{-}(6,-2)*{\circ};(6,3)*{\circ}};
			{\ar@{-}(6,3)*{\circ};(3,8)*{\circ}};
			{\ar@{-}(6,3)*{\circ};(9,8)*{\circ}};
			\endxy
			$$
			while the following ones are not
			$$
			
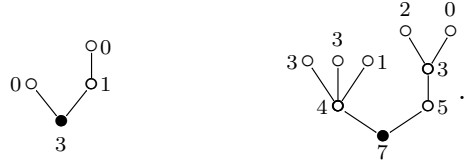
\xy     {(-6,1)*{_0}};
			{(0,-7)*{_3}};
			{(6,1)*{_1}};
			{(6,6)*{_0}};
			{\ar@{-}(0,-4)*{\bul};(-4,1)*{\circ}};
			{\ar@{-}(0,-4)*{\circ};(4,1)*{\circ}};
			{\ar@{-}(4,1)*{\circ};(4,6)*{\circ}};
			\endxy \qquad \qquad \qquad
			\xy
			{(0,-8)*{_7}};
			{(-8,-2)*{_4}};
			{(-12,4)*{_3}};
			{(-6,7)*{_3}};
			{(0,4)*{_1}};
			{(8,3)*{_3}};
			{(8,-2)*{_5}};
			{(3,11)*{_2}};
			{(9,11)*{_0}};		
			{\ar@{-}(0,-6)*{\bul};(-6,-2)*{\circ}};
			{\ar@{-}(0,-6)*{\bul};(6,-2)*{\circ}};
			{\ar@{-}(-6,-2)*{\circ};(-6,4)*{\circ}};
			{\ar@{-}(-6,-2)*{\circ};(-10,4)*{\circ}};
			{\ar@{-}(-6,-2)*{\circ};(-2,4)*{\circ}};
			{\ar@{-}(6,-2)*{\circ};(6,3)*{\circ}};
			{\ar@{-}(6,3)*{\circ};(3,8)*{\circ}};
			{\ar@{-}(6,3)*{\circ};(9,8)*{\circ}};
			\endxy.
			$$
		\end{example}
		
		Recall that the distance of a vertex $v$ of $T$ from the root is the number of edges in the directed path connecting $v$ to the root.
		
		\begin{definition} Given a (planar) rooted tree $T\in\pT$, we denote by $r_T$ (resp.: $l_T$) the number of vertices of $T$ having even (resp.: odd) distance from the root.
			
			The notation is justified by looking at the corresponding planar, binary, rooted tree $\Phi^{-1}(T)\in\pbT$: then $r_T$ (resp.: $l_T$) is precisely the number of right (resp.: left) pointing leaves of $\Phi^{-1}(T)$. 
		\end{definition}

		\begin{proposition}\label{prop: solution for Eulerian}
			Given $T\neq \bullet$, the polynomial $\altbin{t}{T}$ is the unique polynomial that, when evaluated at $\pm k$, $k\in \mathbb{N}\setminus \{0\}$, yields
			\begin{eqnarray*}
				\altbin{k}{T}&=&(-1)^{r_T-1}|\Adm_\pol(T,2k-1)|, \\
				\altbin{-k}{T}&=& (-1)^{l_T}|\Adm_\pol(T,2k)|.
			\end{eqnarray*}
			
		\end{proposition}
		
		\begin{proof}
			We prove the claim 
			by inductively establishing that the assignment
			$$ \beta: \T\times \mathbb{Z} \mapsto \mathbb{Z}, \quad (T,k)\mapsto \beta(T,k):= \begin{cases} (-1)^{r_T-1}|\Adm_\pol(T,2k-1)| & \textrm{for } k>0, \\
			(-1)^{l_T}|\Adm_\pol(T,-2k)| & \textrm{for } k\le 0
			\end{cases}$$  
			satisfies the following recursion:
			\begin{itemize}
				\item[(1)] $\beta(\bullet,k)=1$ for all $k\in\mathbb{Z}$.

				\item[(2)] for a rooted tree $T=\xy {\ar@{-}(0,-4)*{\bul};(-6,2)};{(-6,4)*{\slab{T_1}}};
				{\ar@{-}(0,-4)*{\bul};(6,2)};{(6,4)*{\slab{T_j}}};{\ar@{.}(-3,4);(3,4)};{\ar@{.}(-2.5,0);(2.5,0)}   \endxy$,  with $j\geq1$ (or in other words, $T\neq\bul$), one has
				$$\beta\left(\xy {\ar@{-}(0,-4)*{\bul};(-6,2)};{(-6,4)*{\slab{T_1}}};
				{\ar@{-}(0,-4)*{\bul};(6,2)};{(6,4)*{\slab{T_j}}};{\ar@{.}(-3,4);(3,4)};{\ar@{.}(-2.5,0);(2.5,0)}   \endxy,k\right)=\prod_{i=1}^j \beta\left(\xy {\ar@{-}(0,-3)*{\bul};(0,1)};{(0,3)*{\slab{T_i}}};
				\endxy,k\right),$$
				\item[(3)] for every rooted tree $T$ one has
				$\beta\left(\xy {\ar@{-}(0,-3)*{\bul};(0,1)};{(0,3)*{\slab{T}}};
				\endxy,0\right)=0$, and for all $k>0$
				\[ \beta\left(\xy {\ar@{-}(0,-3)*{\bul};(0,1)};{(0,3)*{\slab{T}}};
				\endxy,k+1\right) = \beta(T,-k)+\beta(T,-k+1)+\cdots+\beta(T,-1), \]
				\[ \beta\left(\xy {\ar@{-}(0,-3)*{\bul};(0,1)};{(0,3)*{\slab{T}}};
				\endxy,-k\right)= -\beta(T,k)-\beta(T,k-1)-\cdots-\beta(T,1). \] 
			\end{itemize}
			Using Theorem \ref{th: product rule planar}, it is easy to check that the assignment $\T\times \mathbb{Z} \to \mathbb{Z}: (T,k)\to\altbin{k}{T}$ obeys the same recursion, and the claimed equality follows.
			
			In order to establish the above recursion, we use induction on $|T|$, the case $T=\bullet$ being the starting point.
			The only non-trivial thing to check is Item (3). We know by induction that the signs on the right-hand sides of each of the claimed
			identities are constant, so we can separately check their validity for the absolute values, followed by a comparison of the signs.
			The identity for the absolute values is a direct consequence of the definitions and the fact that the decreasing decorations are required to be alternating. The identity of the signs is a simple check, using the fact that for $Z(T):=\xy {\ar@{-}(0,-3)*{\bul};(0,1)};{(0,3)*{\slab{T}}};
			\endxy$ (this notation is borrowed from the proof of Theorem \ref{th: product rule planar}) we have $r_{Z(T)}=l_T+1$ and $l_{Z(T)}=r_T$.

		\end{proof}
		
		\begin{remark}\label{rem: solution for Eulerian} We could use the previous proposition to deduce combinatorial formulas for the Eulerian coefficients, always in the spirit of Theorem \ref{th: combinatorial}. For instance we may use our knowledge of the integers $P(T)(k)$, $k\in\mathbb{N}$, to deduce the expansion $P(T)(t)=\sum_{k=0}^{|T|-1}\lambda_k\frac{(t)^k}{k!}$ of the polynomial $P(T)(t)$ with respect to the basis $\left\{\frac{(t)^k}{k!}\right\}$ of $\K[t]$, where we denote by $(t)^k:=t(t+1)\cdots(t+k-1)$ the rising factorials. We can compute the coefficients $\lambda_k$ in a similar way as we did in the proof of \ref{th: combinatorial}, since the rising and the falling factorials are related by the identity $(t)_k=(-1)^k(-t)^k$. It is known that $\left<\left.\frac{D}{1-e^{-D}}\right|  \frac{(t)^k}{k!}\right>=\frac{1}{k+1}$, hence we would obtain closed formulas for $E_T=\left<\left.\frac{D}{1-e^{-D}}\right| P(T)(t)\right>$. 
			
			Another approach would be to try to expand the polynomial $P(T)(t)$ with respect to the basis $\{P(T_{k+1})(t)\}_{k\geq0}$ of $\K[t]$ given by the polynomials associated to the tall trees, as in Corollary \ref{cor:coefficients of tall trees}. It is not hard to check that, given $p(t)\in\K[t]$, its expansion $p(t)=\sum_{k=0}^{|p|}\mu_k P(T_{k+1})(t)$ can be computed via
			\begin{eqnarray*}\mu_0&=&p(0)\\ \mu_1& =& p(1) - p(0),\\ \mu_2& =& p(1) -2p(0)+ p(-1),\\ \mu_3&=& p(2) - 3p(1) + 3p(0) - p(-1),\\  \mu_4&=&p(2) - 4 p(1) + 6p(0) -4p(-1)+p(-2),\\ \mu_5 &=& p(3)-5 p(2) +10p(1)-10p(0)+5 p(-1) - p(-2),\\&\ldots & 			
			\end{eqnarray*} 
			Using the previous proposition, together with Corollary \ref{cor:coefficients of tall trees}, once again we would obtain closed formulas for the coefficient $E_T$.
			
			Neither of the above approaches is entirely satisfying, though, as we don't know a combinatorial interpretation for the resulting (integral, by the previous proposition) coefficients $\lambda_k$, $\mu_k$, in the expansion of $P(T)(t)$, as we did in Theorem \ref{th: combinatorial}. A more satisfying combinatorial formula for $E_T$ will be given in the next section, Proposition \ref{prop:combinatorial eulerian}.

		\end{remark}
		
		\section{{\bf The Eulerian idempotent in Dynkin's basis}}\label{section: Eulerian in Dynkin} In this section we associate certain numbers to trees, generalizing the classical Eulerian numbers (cf.~\cite{Petersen}). We then proceed to generalize classical identities involving Eulerian numbers, such as Worpitzki's identity. In the final subsection, we apply these results to prove Theorem \ref{th:eulerian in Dynkin basis} on the expansion of $E(x_1\cdots x_n)$ in Dynkin's basis from the Introduction.
		
		\subsection{From Dynkin's basis to the PBW basis}

		We consider the pre-Lie algebra $(\pblT(\underline{n}),\rhd)$ of planar, binary, rooted trees, equipped with an admissible labeling $\ell$ by the totally ordered set $(\underline{n},\leq)=\{1\leq\cdots\leq n\}$, cf.~Definition \ref{def:admissible}. We shall denote by $\Delta:\pblT(\underline{n})\to \pblT(\underline{n})\otimes\pblT(\underline{n})$ the pre-Lie coproduct dual to $\rhd$: this is given by the formula \eqref{eq:copr} in Remark \ref{rem: pre-lie copr}, cf.~also the proof of Lemma \ref{lem:labeling}. We shall also denote by $\Delta^{k-1}:\pblT(\underline{n})\to \pblT(\underline{n})^{\otimes k}$ the iterated coproduct, defined recursively by (notice that $\Delta$ is not coassociative) $\Delta^1:=\Delta$, $\Delta^{k-1}:=(\id\otimes\Delta^{k-2})\circ\Delta$ for $k\geq3$. 
		
		\begin{definition}\label{def:associated permutations} Given a labeled tree $(T,\ell)\in\pblT(\underline{n})$, where $T$ has precisely $n$ leaves, we define a subset $S(T,\ell)\subset S_{n-1}$ of the symmetric group by the identity
			\[ \Delta^{n-1}(T,\ell) = (-1)^{r_T-1}\sum_{\sigma\in S(T,\ell)} \, \xy  {\ar@{-}(0,-2);(0,2)};(0,4)*{\text{\footnotesize{$\sigma(1)$}}}\endxy \,\otimes\cdots\otimes\, \xy  {\ar@{-}(0,-2);(0,2)};(0,4)*{\text{\footnotesize{$\sigma(n-1)$}}}\endxy \,\otimes \,\xy  {\ar@{-}(0,-2);(0,2)};(0,4)*{\text{\footnotesize{$n$}}}\endxy. \]
			where, as usual, we denote by $r_T$ the number of right pointing leaves of $T$.
			
		\end{definition}
		
		\begin{remark}\label{rem:removable} Using the discussion from Remark \ref{rem: pre-lie copr}, we may give a more explicit description of the subset $S(T,\ell)$. 
			
			First of all, for any planar, binary, rooted tree $T'$, we say that a leaf $l$ of $T'$ is \emph{removable} if (using the notations from Remark \ref{rem: pre-lie copr}):
			\begin{itemize} \item $l$ is the left child of a vertex $v$ of $T'$ such that $\varepsilon(v)=1$; or\item $l$ is the right child of a vertex $v$ of $T'$ such that $\varepsilon(v)=-1$.
			\end{itemize}
			
			With this terminology, it follows from the definition of $S(T,\ell)$ and the explicit description of the coproduct $\Delta$, that a permutation $\sigma\in S_{n-1}$ is in $S(T,\ell)$ if and only if the following condition is satisfied:
			\begin{itemize} \item[(\#)] For all $1\leq i\leq n-1 $, the leaf labeled by $\sigma(i)$ is removable in the tree $T\setminus\left\{\xy  {\ar@{-}(0,-3);(0,1)};(0,3)*{\text{\footnotesize{$\sigma(1)$}}}\endxy \,,\cdots\,, \xy  {\ar@{-}(0,-3);(0,1)};(0,3)*{\text{\footnotesize{$\sigma(i-1)$}}}\endxy\right\}$, where we have removed the leaves labeled by $\sigma(1),\ldots,\sigma(i-1)$ from $T$. 
			\end{itemize}
		\end{remark}
		
		\begin{remark} We shall represent a permutation $\sigma\in S_{n-1}$ by the corresponding string of numbers $\sigma(1)\cdots\sigma(n-1)$. For instance, the permutation $\sigma=3421\in S_4$ is the one given by $\sigma(1)=3$, $\sigma(2)=4$, $\sigma(3)=2$, $\sigma(4)=1$.
		\end{remark}
		
		\begin{example}\label{ex:removable} Given the labeled tree
			\[ (T,\ell) = \xy {\ar@{-}(0,-4);(0,-2)};
			{\ar@{-}(0,-2);(-6,4)};{(-6,6)*{\slab{1}}};
			{\ar@{-}(-4,2);(-2,4)};{(-2,6)*{\slab{3}}};
			{\ar@{-}(0,-2);(6,4) };{(6,6)*{\slab{4}}};
			{\ar@{-}(4,2);(2,4)}; 	{(2,6)*{\slab{2}}};
			\endxy,   \]
			we can use the previous Remark \ref{rem:removable} to check that $S(T,\ell)=\{231,321,312\}\subset S_3$. For instance, to see that $231\in S(T,\ell)$, we check that the leaf labeled by $2$ is removable in $T$, the leaf labeled by $3$ is removable in $	\xy
			{\ar@{-}(0,-2);(0,0)};
			{\ar@{-}(0,0);(4,4)};{(4,6)*{\slab{4}}};
			{\ar@{-}(0,0);(-4,4)};{(-4,6)*{\slab{1}}};
			{\ar@{-}(-2,2);(0,4)};{(0,6)*{\slab{3}}};
			\endxy=T\setminus\left\{\xy  {\ar@{-}(0,-3);(0,1)};(0,3)*{\text{\footnotesize{$2$}}} \endxy\right\}$, and the leaf labeled by $1$ is removable in $ \quad\xy
			{\ar@{-}(0,-2);(0,0)};
			{\ar@{-}(0,0);(-2,2)};{(-2,4)*{\slab{1}}};
			{\ar@{-}(0,0);(2,2)};{(2,4)*{\slab{4}}};
			\endxy=T\setminus\left\{\xy  {\ar@{-}(0,-3);(0,1)};(0,3)*{\text{\footnotesize{$2$}}}\endxy \,, \xy  {\ar@{-}(0,-3);(0,1)};(0,3)*{\text{\footnotesize{$3$}}}\endxy\right\}$.
			
			As another example, given the labeled tree
			\[ (T,\ell)= 
			\xy {\ar@{-}(0,-4);(0,-2)};{\ar@{-}(0,-2);(-8,6)};{(-8,8)*{\slab{1}}};
			{\ar@{-}(-6,4);(-4,6)};{(-4,8)*{\slab{3}}};
			{\ar@{-}(0,-2);(8,6) };{(8,8)*{\slab{5}}};
			{\ar@{-}(4,2);(0,6)};{(0,8)*{\slab{2}}};
			{\ar@{-}(2,4);(4,6)};{(4,8)*{\slab{4}}};
			\endxy \]
			we have $S(T,\ell)=\{3142,3412,3421,4312,4321,4231\}\subset S_4$.
			
			As a final example, if $(T,\ell)$ is a left pointing comb with $(n+1)$ leaves and its unique admissible labeling, $T=\xy {\ar@{-}(0,-4);(0,-2)};
			{\ar@{-}(0,-2);(-8,6)};{(-8,8)*{\slab{1}}};
			{\ar@{-}(0,-2);(8,6)};{(11,8.25)*{\slab{n+1}}};
			{\ar@{-}(6,4);(4,6) }; {(4,8)*{\slab{n}}};
			{\ar@{-}(2,0);(-4,6) };{(-4,8)*{\slab{2}}};
			{\ar@{.}(1.5,1.5);(4.5,4.5)}
			\endxy$, then every leaf different from the rightmost one is removable, and $S(T,\ell)=S_n$.
		\end{example}
		
		The following Proposition will allow us to pass from the PBW pasis $\basis_n$ of $\Lie_n$, considered in the previous Section \ref{section: Eulerian coefficients}, to Dynkin's basis 
		\[ \mathcal{D}_n=\{ [x_{\sigma(1)},\cdots[x_{\sigma(n-1)},x_n] \cdots] \}_{\sigma\in S_{n-1}}. \]
		\begin{proposition}\label{prop: from dynkin to pbw} Given an element $[x_{\sigma(1)},\cdots[x_{\sigma(n-1)},x_n] \cdots]$ in Dynkin's basis, its expansion in the PBW basis reads
			\[ [x_{\sigma(1)},\cdots[x_{\sigma(n-1)},x_n] \cdots] = \sum_{(T,\ell)\,\operatorname{s.t.}\,\sigma\in S(T,\ell)}(-1)^{r_T-1} r(T,\ell), \]
			where $r:\pblT(\underline{n})\to\Lie_{\leq n}$ is the realization isomorphism from Lemma \ref{lem:realization}, and the tree $T$ has precisely $n$ leaves.
		\end{proposition}
		\begin{proof} It follows directly from Definition \ref{def:associated permutations} that $\sigma$ lies in $S(T,\ell)$ if and only if $(T,\ell)$ appears, with the coefficient $(-1)^{r_T-1}$, in the expansion of
			\[ \xy  {\ar@{-}(0,-3);(0,1)};(0,3)*{\slab{\sigma(1)}}\endxy \rhd\left(\cdots\left( \xy  {\ar@{-}(0,-3);(0,1)};(0,3)*{\slab{\sigma(n-1)}}\endxy \rhd \xy  {\ar@{-}(0,-3);(0,1)};(0,3)*{\slab{n}}\endxy\right)\cdots\right) \]
			inside the pre-Lie algebra $\pblT(\underline{n})$. According to Remark \ref{remboh} (which implies $(T_1,\ell_1)\rhd(T_2,\ell_2) =[(T_1,\ell_1),(T_2,\ell_2)]$ if the maximum label of $T_2$ is greater than the maximum label of $T_1$), the identity 
			\[\left[ \xy  {\ar@{-}(0,-3);(0,1)};(0,3)*{\slab{\sigma(1)}}\endxy, \cdots\left[ \xy  {\ar@{-}(0,-3);(0,1)};(0,3)*{\slab{\sigma(n-1)}}\endxy , \xy  {\ar@{-}(0,-3);(0,1)};(0,3)*{\slab{n}}\endxy\right]\cdots\right] = \sum_{(T,\ell)\,\operatorname{s.t.}\,\sigma\in S(T,\ell)}(-1)^{r_T-1} (T,\ell) \]
			holds inside $\pblT(\underline{n})$, equipped with the commutator bracket associated to $\rhd$. Since $r$ is an isomorphism of Lie algebras by Lemma \ref{lem:realization}, the thesis follows.
		\end{proof}
		\subsection{A generalization of Worpitzki's identity} The discussion in the previous subsection leads us to consider the following generalization of the classical Eulerian numbers \cite{Petersen}.
		\begin{definition} Given a labeled (planar, binary, rooted) tree $(T,\ell)\in\pblT(\underline{n})$ with $n$ leaves, we denote by $E(T,\ell, d)$, $0\leq d\leq n-2$, (only temporarily, soon we will drop the labeling $\ell$ from the notation), the cardinality of the set
			\[ E(T,\ell,d):= |\{\sigma\in S(T,\ell)\,\operatorname{s.t.}\, d_\sigma = d   \}|,  \]
			where $d_\sigma$ is the descent number of the permutation $\sigma$ (cf.~the Introduction). We call the number $E(T,\ell, d)$ the \emph{$d$-th Eulerian number associated to $(T,\ell)$}.
		\end{definition}
		\begin{remark}\label{rem:eulnumb} As we said in Example \ref{ex:removable}, if $(T,\ell)$ is a left pointing comb with $(n+1)$ leaves and its unique admissible labeling, $T=\xy {\ar@{-}(0,-4);(0,-2)};
			{\ar@{-}(0,-2);(-8,6)};{(-8,8)*{\slab{1}}};
			{\ar@{-}(0,-2);(8,6)};{(11,8.25)*{\slab{n+1}}};
			{\ar@{-}(6,4);(4,6) }; {(4,8)*{\slab{n}}};
			{\ar@{-}(2,0);(-4,6) };{(-4,8)*{\slab{2}}};
			{\ar@{.}(1.5,1.5);(4.5,4.5)}
			\endxy$, then $S(T,\ell)=S_n$, and the associated Eulerian numbers are the classical ones (see \cite{Petersen}), which we denote by $E(n,d)$.
		\end{remark} 
		Our first objective is to show that
		\begin{proposition}\label{prop: eulerian independet of labeling} The Eulerian numbers $E(T,\ell,d)$ associated to $(T,\ell)$ are independent of the (admissible) labeling $\ell$.
		\end{proposition}

		\begin{proof} We proceed by induction on the number of leaves $|T|$. The thesis is empty when $|T|=1$, and straightforward when $|T|=2$.
			
			We consider a tree of the form $T=\xy  {\ar@{-}(0,-4);(0,-2)};
			{\ar@{-}(0,-2);(-3,1)};{(-3,3)*{\slab{T_1}}};
			{\ar@{-}(0,-2);(10,8)};
			{\ar@{-}(6,4);(3,7) };   {(3,9)*{\slab{T_k}}};
			{\ar@{.}(-1,1);(3,5)}
			\endxy  $, and equip $T_i$, $1\leq i\leq k$, with the labeling $\ell_i$ induced from $\ell$. For every $1\leq i\leq k$, let $\{j^i_1\leq\cdots\leq j^i_{|T_i|}\}$ be the set of labels appearing in $\ell_i$. We define a subset $S(T,T_i)\subset S_{|T_i|}$ in a similar way as in Remark \ref{rem:removable}, by saying that $\tau\in S_{|T_i|}$ is in $S(T,T_i)$ if, for all $1\leq r \leq |T_i|$, the leaf labeled by $j^i_{\tau(r)}$ is removable in the tree $T\setminus\left\{\xy  {\ar@{-}(0,-3);(0,1)};(0,4)*{\slab{j^i_{\tau(1)}}}\endxy \,,\cdots\,, \xy  {\ar@{-}(0,-3);(0,1)};(0,4)*{\slab{j^i_{\tau(r-1)}}}\endxy\right\}$, where we have removed the leaves labeled by $j^i_{\tau(1)},\ldots,j^i_{\tau(r-1)}$ from $T$. For every $\tau\in S(T,T_i)$, we denote by $w_\tau$ the corresponding word $w_\tau:=j^i_{\tau(1)}\cdots j^i_{\tau(|T_i|)}$ (we notice that $\tau(|T_i|)=1$: in fact, the leftmost leaf of $T_i$, which is labeled by $j^i_1$ since $\ell$ is admissible, becomes removable only after all the other leaves of $T_i$ have been removed). It is now obvious from Remark \ref{rem:removable} that the set $S(T,\ell)$ is in bijective correspondence with the set of choices of a permutation $\tau_i\in S(T,T_i)$ for all $1\leq i \leq k$, and of a shuffle of the corresponding words $w_{\tau_1},\ldots,w_{\tau_k}$.

			The proof of the proposition is completed by the following two claims.\\

			{\bf Claim:} the number of $\tau\in S(T,T_i)$ such that $d_{\tau}=d_i$, for a certain $0\leq d_i\leq |T_i|-1$, is independent of the labeling $\ell_i$, and only depends on the planar, binary, rooted tree $T_i$.\\
			
			{\bf Claim:} Given permutations $\tau_i\in S(T,T_i)$ having descent numbers $d_{\tau_i}=d_i$, for all $1\leq i\leq k$, the number of shuffles of $w_{\tau_1},\ldots,w_{\tau_k}$ having descent number a certain $d$, $0\leq d\leq |T|-2$, is independent of the particular choices of $\tau_i$, and only depends on the numbers $|T_1|,\ldots,|T_k|,d_1,\ldots,d_k$.\\
			
			The first claim follows by the inductive hypothesis. To see this, we consider the specular tree $\Sigma(T_i)$ (cf.~Subsection \ref{subsec: specular}), equipped with the labeling $\ell'_i$ defined as follows: given a leaf $l$ of $T_i$, labeled by the number $j^i_r$ under $\ell_i$, the corresponding leaf of $\Sigma(T_i)$ is labeled by $|T_i|-r+1$ under $\ell'_i$. It is clear that $\ell'_i$ is an admissible labeling of $\Sigma(T_i)$ by the set $\{1\leq\cdots\leq |T_i|\}$. We notice that the `` specular'' involution $\Sigma:S_{|T_i|}\to S_{|T_i|}$ from Definition \ref{def: specular on permutations} induces a bijective correspondence $\widetilde{\Sigma}:S(T,T_i)\to S(\Sigma(T_i),\ell'_i)$. In fact, we already noticed that $\tau(|T_i|)=1$ whenever $\tau\in S(T,T_i)$, thus $\Sigma\tau(|T_i|)=|T_i|$, and $\widetilde{\Sigma}\tau$ is just the restriction of $\Sigma\tau$ to $\{1,\ldots,|T_i|-1\}$: it is easy to see that $\tau$ satisfies the requirement to be in $S(T,T_i)$ if and only if $\widetilde{\Sigma}\tau$ satisfies the corresponding requirement ($\#$) from Remark \ref{rem:removable}. Finally, comparing with Lemma \ref{lem: specular on permutations}, we see that $d_{\tau}=|T_i| - d_{\Sigma\tau}-1=|T_i| - d_{\widetilde{\Sigma}\tau}-1$, and we can apply the inductive hypothesis to deduce the first claim.\\

			The second claim is a consequence of Stanley's shuffling Theorem \cite{Stanley, Goulden}. More precisely, this is really a statement about shuffling of words, and using an obvious induction, it is sufficient to consider the case of two words. In this case, the claim follows from \cite[Theorem 1.2]{Goulden}, after substituting $q=1$ in loc.~cit.. 
		\end{proof}
		
		\begin{remark}According to the previous proposition, we can drop the labeling from the notation for the Eulerian numbers associated to a tree $T\in\pbT$, which from here on will be denoted by $E(T,d)$, $0\leq d\leq |T|-2$.
		\end{remark}
		
		Our next objective is to relate these numbers to the polynomial $P(T)(t)$ from the previous Subsections \ref{subsection:umbral in planar} and \ref{subsection: from pbT to pT}.
		
		\begin{remark}\label{rem:removable planar} In order to make use of the results from Subsection \ref{subsection: from pbT to pT}, we switch to planar rooted trees using the ($\Sigma$-twisted) rotation correspondence $\Phi$. Under $\Phi$, an admissible labeling $\ell:\{\mbox{leaves of $T$} \}\to\{1,\ldots,n\}$ corresponds to a labeling $\Phi(\ell):\{\mbox{vertices of $\Phi(T)$}\}\to\{1,\ldots,n\}$ satisfying the following two conditions:
			\begin{itemize} 
				\item For every vertex $v$ of $\Phi(T)$ at an even (resp.: odd) level, i.e., having even (resp.: odd) distance from the root, the label of $v$ is a maximum (resp.: minimum) for the set of labels of the descendants of $v$.
				\item The labels of the children of an even (resp.: odd) level vertex are increasing from left to right (resp.: from right to left).
			\end{itemize}
			For instance, the following two labelings correspond to admissible labelings under the rotation correspondence,
			$$
			
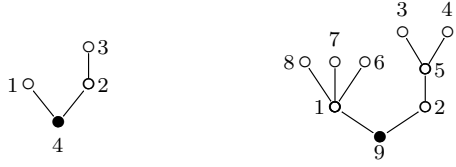
\xy     {(-6,1)*{_1}};
			{(0,-7)*{_4}};
			{(6,1)*{_2}};
			{(6,6)*{_3}};
			{\ar@{-}(0,-4)*{\bul};(-4,1)*{\circ}};
			{\ar@{-}(0,-4)*{\circ};(4,1)*{\circ}};
			{\ar@{-}(4,1)*{\circ};(4,6)*{\circ}};
			\endxy \qquad \qquad \qquad
			\xy
			{(0,-8)*{_9}};
			{(-8,-2)*{_1}};
			{(-12,4)*{_8}};
			{(-6,7)*{_7}};
			{(0,4)*{_6}};
			{(8,3)*{_5}};
			{(8,-2)*{_2}};
			{(3,11)*{_3}};
			{(9,11)*{_4}};		
			{\ar@{-}(0,-6)*{\bul};(-6,-2)*{\circ}};
			{\ar@{-}(0,-6)*{\bul};(6,-2)*{\circ}};
			{\ar@{-}(-6,-2)*{\circ};(-6,4)*{\circ}};
			{\ar@{-}(-6,-2)*{\circ};(-10,4)*{\circ}};
			{\ar@{-}(-6,-2)*{\circ};(-2,4)*{\circ}};
			{\ar@{-}(6,-2)*{\circ};(6,3)*{\circ}};
			{\ar@{-}(6,3)*{\circ};(3,8)*{\circ}};
			{\ar@{-}(6,3)*{\circ};(9,8)*{\circ}};
			\endxy
			$$
			while the following two do not (the left one fails to satisfy the second item above, while the right one fails to satisfy the first):
			$$
			
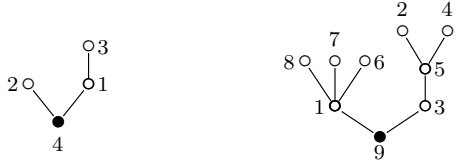
\xy     {(-6,1)*{_2}};
			{(0,-7)*{_4}};
			{(6,1)*{_1}};
			{(6,6)*{_3}};
			{\ar@{-}(0,-4)*{\bul};(-4,1)*{\circ}};
			{\ar@{-}(0,-4)*{\circ};(4,1)*{\circ}};
			{\ar@{-}(4,1)*{\circ};(4,6)*{\circ}};
			\endxy \qquad \qquad \qquad
			\xy
			{(0,-8)*{_9}};
			{(-8,-2)*{_1}};
			{(-12,4)*{_8}};
			{(-6,7)*{_7}};
			{(0,4)*{_6}};
			{(8,3)*{_5}};
			{(8,-2)*{_3}};
			{(3,11)*{_2}};
			{(9,11)*{_4}};		
			{\ar@{-}(0,-6)*{\bul};(-6,-2)*{\circ}};
			{\ar@{-}(0,-6)*{\bul};(6,-2)*{\circ}};
			{\ar@{-}(-6,-2)*{\circ};(-6,4)*{\circ}};
			{\ar@{-}(-6,-2)*{\circ};(-10,4)*{\circ}};
			{\ar@{-}(-6,-2)*{\circ};(-2,4)*{\circ}};
			{\ar@{-}(6,-2)*{\circ};(6,3)*{\circ}};
			{\ar@{-}(6,3)*{\circ};(3,8)*{\circ}};
			{\ar@{-}(6,3)*{\circ};(9,8)*{\circ}};
			\endxy
			$$
			
			We can use the bijection $\Phi$ to associate Eulerian numbers $E(T,d)$, $0\leq d\leq|T|-2$, to every planar rooted tree $T\in\pT$. We shall give a more explicit description of these numbers, in the spirit of Remark \ref{rem:removable}. First of all, we already observed that a vertex $v$ of $T$ corresponds to a removable leaf in $\Phi^{-1}(T)$ if and only if $v$ is a leaf of $T$, different from the root. For instance, for the tree 
			
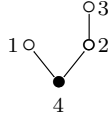
\[ \xy     {(-6,1)*{_1}};
			{(0,-7)*{_4}};
			{(6,1)*{_2}};
			{(6,6)*{_3}};
			{\ar@{-}(0,-4)*{\bul};(-4,1)*{\circ}};
			{\ar@{-}(0,-4)*{\circ};(4,1)*{\circ}};
			{\ar@{-}(4,1)*{\circ};(4,6)*{\circ}};\endxy\] 
			the removable vertices are the ones labeled by $1$ and $3$. Given an admissible labeling of $T$, that is, a bijective correspondence $\ell:\{\mbox{vertices of $T$}\}\to\{1,\ldots,|T|\}$ satisfying the two conditions above, we can define a subset $S(T,\ell)\subset S_{|T|-1}$ as in Remark \ref{rem:removable}. Namely, we say that $\sigma\in S(T,\ell)$ if for all $1\leq i\leq |T|-1$, the vertex labeled by $\sigma(i)$ is removable in the tree $T\setminus\left\{ \xy {(0,-1.5)*{\bul}};{(0,1.5)*{\slab{\sigma(1)}}} \endxy,\ldots,\xy {(0,-1.5)*{\bul}};{(0,1.5)*{\slab{\sigma(i-1)}}} \endxy \right\}$, where we have removed the vertices labeled by $\sigma(1),\ldots,\sigma(i-1)$ from $T$. The Eulerian number $E(T,d)$ is the number of permutations $\sigma\in S(T,\ell)$, for any admissible labeling $\ell$ of $T$, having descent number $d_\sigma=d$. 
			
			We notice that a vertex of $T$ becomes removable only after all its descendants have already been removed. This shows that there is a bijective correspondence between the set $S(T,\ell)$ and the set $\Adm^c(T,|T|-1)$ from Definition \ref{def:decreasing decoration}: given a complete, decreasing decoration $\delta\in\Adm^c(T,|T|-1)$, we remove the vertices of $T$ in the order prescribed by $\delta$. More precisely, we shift the labels in $\delta$ by $1$, in order to get a bijective correspondence $\underline{\delta}:\{\mbox{vertices of $T$}\}\to\{1,\ldots,|T|\}$: then the desired $\Adm^c(T,|T|-1)\xrightarrow{\cong} S(T,\ell)$ sends $\delta$ to the permutation $\ell\circ\underline{\delta}^{-1}$, restricted to $\{1,\ldots,|T|-1\}$. For instance, for the above labeled tree, the complete decreasing decoration
			\[ \delta = \xy     {(-6,1)*{_2}};
			{(0,-7)*{_3}};
			{(6,1)*{_1}};
			{(6,6)*{_0}};
			{\ar@{-}(0,-4)*{\bul};(-4,1)*{\circ}};
			{\ar@{-}(0,-4)*{\circ};(4,1)*{\circ}};
			{\ar@{-}(4,1)*{\circ};(4,6)*{\circ}};\endxy\in \Adm^c(T,|T|-1)\]
			goes into the permutation $321\in S(T,\ell)$ under this bijection. By Remark \ref{rem:}, we find the following formula for the sum of the Eulerian numbers associated to $T$:
			\[ \sum_{d=0}^{|T|-2} E(T,d) = |S(T,\ell)| = |\Adm^c(T,|T|-1)| = \frac{|T|!}{T!}.  \]
			
		\end{remark}

		\begin{theorem}\label{th:worpitzki} The polynomial $P(T)(t)$ and the Eulerian numbers $E(T,d)$ associated to a planar rooted tree $T\in\pT$, are related by the identity
			\begin{equation}\label{eq:worpitzki} (-1)^{r_T-1}P(T)(k+1) = \sum_{d=0}^{|T|-2} \binom{|T|-1 + k-d}{|T|-1}E(T,d),
			\end{equation}
			for all integers $k\geq0$ (by interpreting the binomial coefficients as generalized ones, in the above formula we have $\binom{|T|-1 + k-d}{|T|-1}=0$ whenever $k<d$).
		\end{theorem}
		\begin{remark} When $T$ is a corolla with $n$ leaves, $T=\overbrace{\xy 
				{\ar@{-}(0,-2.5)*{\bul};(-5,2.5)*{\circ}};
				{\ar@{-}(0,-2.5)*{\bul};(5,2.5)*{\circ}};
				{\ar@{.}(-3.5,2.5);(3.5,2.5)};
				{\ar@{.}(-2,0);(2,0)};
				\endxy}^n$, we have $P(T)(t)= t^n$, $E(T,d)=E(n,d)$ are the usual  Eulerian numbers (by the previous Remark \ref{rem:eulnumb}, since $T$ corresponds to a left pointing comb under $\Phi$), and the previous identity \eqref{eq:worpitzki} becomes the classical Worpitzki's identity
			\[ (k+1)^n = \sum_{d=0}^{n-1} \binom{n + k-d}{n}E(n,d), \]
			cf.~\cite[\S 1.5, Corollary 1.2]{Petersen}. 
		\end{remark}
		\begin{proof} We adapt a bijective proof of Worpitzki's identity from Knuth's book \cite{Knuth}.

			First of all, we shall equip $T$ with a particular admissible labeling $\ell$, defined as follows (we remark that the following proof does not work for a general admissible labeling). The labeling $\ell$ associates (necessarily) the number $|T|$ to the root. Next, assume there are $i_1$ vertices at level one: we label them with the numbers $1,\ldots,i_1$, in increasing order from left to right. If there are $i_2$ vertices at level two, we label them with the numbers $|T|-i_2,\ldots,|T|-1$, in increasing order from right to left. In general, if there are $i_{2j}$ (resp.: $i_{2j+1}$) vertices at the even (resp.: odd) level $2j$ (resp.: $2j+1$), we label them with the numbers $|T|+1-\sum_{h=0}^{j}i_{2h},\ldots,|T|-\sum_{h=0}^{j-1}i_{2h}$ (resp.: $1+\sum_{h=0}^{j-1}i_{2h+1},\ldots,\sum_{h=0}^{j}i_{2h+1}$), in increasing order from right to left (resp.: from left to right). For instance, the admissible labelings depicted in the previous Remark \ref{rem:removable planar} are of this type.
			
			Having fixed the above choice of an admissiblee labeling $\ell$ of $T$, it is defined the corresponding set of permutations $S(T,\ell)\subset S_{|T|-1}$. Now, according to Proposition \ref{prop: solution for Eulerian}, the left hand side in the claimed identity \eqref{eq:worpitzki} is precisely the cardinality of the set $\Adm_\pol(T,2k+1)$ of alternating decreasing decorations of $T$ (Definition \ref{def:alternating decoration}) associating the number $(2k+1)$ to the root. In order to prove the theorem, we shall consider a certain map $\theta:\Adm_\pol(T,2k+1)\to S(T,\ell)$, and count the cardinality of its fibers.
			
			The map $\theta$ is defined as follows. Given an alternating decoration $\delta\in\Adm_\pol(T,2k+1)$, the corresponding element in $\theta(\delta)\in S(T,\ell)$ is defined by removing the leaves of $(T,\ell)$ in the order prescribed by $\delta$, cf.~the discussion in the previous Remark \ref{rem:removable planar}. When there are repetitions in $\delta$, we remove the leaves according to the order prescribed by the labeling $\ell$. For instance, given the labeled tree
			\[ (T,\ell) = 
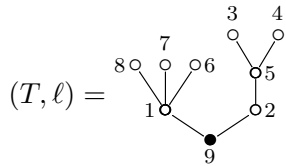
\xy
			{(0,-8)*{_9}};
			{(-8,-2)*{_1}};
			{(-12,4)*{_8}};
			{(-6,7)*{_7}};
			{(0,4)*{_6}};
			{(8,3)*{_5}};
			{(8,-2)*{_2}};
			{(3,11)*{_3}};
			{(9,11)*{_4}};		
			{\ar@{-}(0,-6)*{\bul};(-6,-2)*{\circ}};
			{\ar@{-}(0,-6)*{\bul};(6,-2)*{\circ}};
			{\ar@{-}(-6,-2)*{\circ};(-6,4)*{\circ}};
			{\ar@{-}(-6,-2)*{\circ};(-10,4)*{\circ}};
			{\ar@{-}(-6,-2)*{\circ};(-2,4)*{\circ}};
			{\ar@{-}(6,-2)*{\circ};(6,3)*{\circ}};
			{\ar@{-}(6,3)*{\circ};(3,8)*{\circ}};
			{\ar@{-}(6,3)*{\circ};(9,8)*{\circ}};
			\endxy
			\]
			and the alternating decreasing decoration 
			\[ \delta = 			
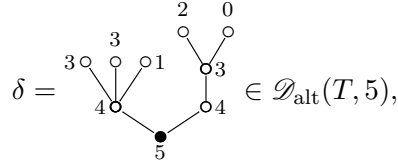
\xy
			{(0,-8)*{_5}};
			{(-8,-2)*{_4}};
			{(-12,4)*{_3}};
			{(-6,7)*{_3}};
			{(0,4)*{_1}};
			{(8,3)*{_3}};
			{(8,-2)*{_4}};
			{(3,11)*{_2}};
			{(9,11)*{_0}};		
			{\ar@{-}(0,-6)*{\bul};(-6,-2)*{\circ}};
			{\ar@{-}(0,-6)*{\bul};(6,-2)*{\circ}};
			{\ar@{-}(-6,-2)*{\circ};(-6,4)*{\circ}};
			{\ar@{-}(-6,-2)*{\circ};(-10,4)*{\circ}};
			{\ar@{-}(-6,-2)*{\circ};(-2,4)*{\circ}};
			{\ar@{-}(6,-2)*{\circ};(6,3)*{\circ}};
			{\ar@{-}(6,3)*{\circ};(3,8)*{\circ}};
			{\ar@{-}(6,3)*{\circ};(9,8)*{\circ}};
			\endxy\in\Adm_\pol(T,5), \]
			to get the corresponding permutation $\theta(\delta)\in S(T,\ell)$, we start by removing from $(T,\ell)$ the leaf labeled by $0$ under $\delta$, thus $\theta(\delta)=4\cdots$. Next, we remove, in order, the leaves labeled by $1$ and $2$ under $\delta$, thus $\theta(\delta)=463\cdots$. Now we remove the leaves labeled by $3$ under $\delta$, according to the order prescribed by $\ell$, thus $\theta(\delta)=463578\cdots$. Finally, we remove the leaves labeled by $4$ under $\delta$, according to the order prescribed by $\ell$, and in the end we found $\theta(\delta)=46357812\in S(T,\ell)$.
			
			The theorem will follow from the claim below.\\
			
			{\bf Claim:} Given $\sigma\in S(T,\ell)$ and $k\geq0$, the cardinality of the fiber $\theta^{-1}(\sigma)\subset \Adm_\pol(T,2k+1)$ is zero if $k<d_\sigma$, and is $\binom{|T|-1 + k-d_\sigma}{|T|-1}$ otherwise.\\
			
			To prove the claim, first of all we shall try to construct, given $\sigma\in S(T,\ell)$, an alternating decreasing decoration $\delta_\sigma$ such that $\theta(\delta_\sigma)=\sigma$, and such that the odd number corresponding to the root under $\delta_\sigma$ is the smallest possible. To distinguish between the labels under $\ell$ and the ones under $\delta_\sigma$, we shall call the first ones $\ell$-labels and the second ones $\delta_\sigma$-labels.  
			
			{\bf Observation:} We notice that, by construction of $\ell$, the $\ell$-label of an even level vertex of $T$ is always greater than the $\ell$-label of an odd level vertex.  
			
			To start constructing $\delta_\sigma$, we consider the vertex $v_1$ of $T$ having $\ell$-label $\sigma(1)$. The $\delta_\sigma$-label of $v_1$ will have to be the smallest possible, which is $0$ if $v_1$ is an odd level vertex and $1$ if $v_1$ is an even level vertices. In fact, since $\delta_\sigma$ has to be alternating, and since it has to associate an odd number to the root, even level vertices will have odd $\delta_\sigma$-labels, and odd level vertices will have even $\delta_\sigma$-labels. Next, we consider the vertex $v_2$ of $T$ having $\ell$-label $\sigma(2)$. There are several cases to consider. If the levels of $v_1$ and $v_2$ have the same parity, and if $\sigma(1)<\sigma(2)$, by the definition of $\theta$ we can put $\delta_\sigma(v_2)=\delta_\sigma(v_1)$. On the other hand, if the levels of $v_1$ and $v_2$ have the same parity, and if $\sigma(1)>\sigma(2)$, we will have to put $\delta_\sigma(v_2)=\delta_\sigma(v_1)+2$. If the levels of $v_1$ and $v_2$ have opposite parity, we have to put $\delta_\sigma(v_2)=\delta_\sigma(v_1)+1$. Notice that in the last case, by the previous observation, we will have $\sigma(1)<\sigma(2)$ if the level of $v_1$ was odd and the level of $v_2$ was even, and we will have $\sigma(1)>\sigma(2)$ otherwise. In all cases, we see by the previous discussion that $\delta_\sigma(v_2)$ is $0$ or $1$ if $\sigma(1)<\sigma(2)$, and is $2$ or $3$ if $\sigma(1)>\sigma(2)$.  
			
			Proceeding like this, we consider the vertex $v_j$ of $T$ having $\ell$-label $\sigma(j)$. Assume that the $\delta_\sigma$-label of $v_j$ is either $2h$ or $2h+1$. The same case by case argument as before, shows that the $\delta_\sigma$-label of $v_{j+1}$ (the vertex having $\ell$-label $\sigma(j+1)$) will have to be either $2h,2h+1$ if $\sigma(j)<\sigma(j+1)$, or $2h+2,2h+3$ if $\sigma(j)>\sigma(j+1)$. This implies that the $\delta_\sigma$-label of the root will have to be precisely $2d_\sigma+1$. Furthermore, since at every step we had only one possible choice of smallest $\delta_\sigma$-label, this will be the only alternating decoration in $\Adm_\pol(T,2d_\sigma+1)$ such that its image under $\theta$ is $\sigma$. This proves the previous claim when $k\leq d_\sigma$.
			
			It remains to consider the case $k>d_\sigma$. In this case, we claim that the set $\theta^{-1}(\sigma)$ is in bijective correspondence with the set of ordered partitions $k-d_{\sigma}=p_1+\cdots+p_{|T|}$ of the number $(k-d_\sigma)$ by $|T|$ non-negative integers. Since there are precisely $\binom{|T|-1 + k-d_\sigma}{|T|-1}$ such partitions, this will conclude the proof of the theorem. As before, we denote by $v_j$ the vertex of $T$ having $\ell$-label $\sigma(j)$ (and by $v_{|T|}$ the root). The desired bijective correspondence is defined as follows: given the ordered partition $k-d_{\sigma}=p_1+\cdots+p_{|T|}$, the corresponding decoration $\delta\in\Adm_\pol(T,2k+1)$ is given by $\delta(v_j)=\delta_\sigma(v_j)+2\sum_{h=1}^{j}p_h$.
		\end{proof}
		
		\begin{corollary} The Eulerian numbers associated to $T\in\pT$ are independent of the planar structure on $T$, and only depend on the underlying rooted tree.
		\end{corollary}
		\begin{proof} This follows from the previous theorem and Corollary \ref{lemma: Euler planar structure}.
		\end{proof}
		
		Other consequences of the previous theorem are the following generalizations of classical identities concerning the Eulerian numbers. 
		
		\begin{corollary}\label{cor:eulerian numbers} Given a rooted tree $\bul\neq T\in\T$, its associated Eulerian numbers $E(T,d)$ can be extracted from the polynomial $P(T)(t)$ via the identity
			\[ E(T,d) = \sum_{k=0}^{d}(-1)^{k+r_T-1} \binom{|T|}{k}P(T)(d+1-k),\qquad \forall \;0\leq d\leq |T|-2.  \]
		\end{corollary}
		\begin{proof} This is obtained by inverting the previous formula \eqref{eq:worpitzki}, as can be seen by following the proof of \cite[\S 1.5, Corollary 1.3]{Petersen}.
		\end{proof}
		To state the last result of this subsection, we introduce the polynomials \[ E(T)(t):=\sum_{d=0}^{|T|-2} E(T,d) t^d,\] 
		which can be regarded as generalized Eulerian polynomials associated to the  rooted tree $T\in\T$. When $T=\bullet$, we put $E(T)(t)=1$. The following corollary generalizes the classical Carlitz identity \cite[\S 1.5, Corollary 1.1]{Petersen}.
		\begin{corollary} For any rooted tree $T\in \T$, 
			\[ \frac{E(T)(t)}{(1-t)^{|T|}} = (-1)^{r_T-1}\sum_{k\geq0} P(T)(k+1)t^k.  \]\end{corollary}
		\begin{proof} Follow the proof of \cite[\S 1.5, Corollary 1.2]{Petersen}, using the identity \eqref{eq:worpitzki}.
		\end{proof}
		
		\subsection{The Eulerian idempotent in Dynkin's basis}\label{subsection: eulerian in dynkin} In this subsection we prove Theorem \ref{th:eulerian in Dynkin basis} from the Introduction. First of all, we shall use the results from the previous subsection to deduce a combinatorial formula for the Eulerian coefficient $E_T$ of a given rooted tree $T\in\T$ in terms of the associated Eulerian numbers $E(T,d)$.
		
		\begin{proposition}\label{prop:combinatorial eulerian} For any rooted tree $\bul\neq T\in\T$, 
			\[ E_T = \frac{(-1)^{r_T-1}}{|T|}\sum_{d=0}^{|T|-2} \frac{(-1)^{d}}{\binom{|T|-1}{d}} E(T,d).  \]
		\end{proposition}
		\begin{proof} We denote by $\lambda_j$ the coefficients of $P(T)(t)$ in the expansion with respect to the binomial basis, that is $P(T)(t)=\sum_{j=0}^{|T|-1}\lambda_j\binom{t}{j}$. These can be computed as in the proof of Theorem \ref{th: combinatorial}, by the formula (where in the first identity we also use the fact that $P(T)(0)=0$ for $T\neq\bullet$)
			\begin{eqnarray*}  \lambda_j&=&
			\sum_{k=0}^{j-1}(-1)^{j+k+1}\binom{j}{k+1}P(T)(k+1) \\ &=& (-1)^{r_T-1}\sum_{k=0}^{j-1}(-1)^{j+k+1}\binom{j}{k+1}\sum_{d=0}^k\binom{|T|-1+k-d}{|T|-1}E(T,d),
			\end{eqnarray*}
			where in the last identity we used the generalized Worpitzki's identity \eqref{eq:worpitzki}. It is known that $\left<\left. \frac{D}{1-e^{-D}}\right|\binom{t}{j}\right>= \frac{(-1)^{j-1}}{j(j+1)}$: for completeness, we give a simple umbral proof of this fact. First of all, since the operator $e^D-1$ on polynomials is the forward difference operator, $\left<\left. \frac{D}{1-e^{-D}}\right|\binom{t}{j}\right>=\left<\left. \frac{De^D}{e^{D}-1}\right|(e^D-1)\binom{t}{j+1}\right>=\left<\left. D e^D \right|\binom{t}{j+1}\right>$, where we used \cite[Theorem 2.2.5]{Roman}. According to \cite[Formula (2.1.4)]{Roman}, the functional $\left<\left. D e^D \right|\#\right>$ sends a polynomial $p(t)$ to $p'(1)$, thus, by a straightfoward computation, it sends the binomial coefficient $\binom{t}{j+1}$ to $\frac{(-1)^{j-1}}{j(j+1)}$. Going back to the proof of the proposition, we found that
			\begin{eqnarray*} E_T &=& \left<\left. \frac{D}{1-e^{-D}}\right|\sum_{j=0}^{|T|-1}\lambda_j\binom{t}{j}\right> \\ &=&(-1)^{r_T-1}\sum_{j=0}^{|T|-1} \frac{(-1)^{j-1}}{j(j+1)}\sum_{k=0}^{j-1}(-1)^{j+k+1}\binom{j}{k+1}\sum_{d=0}^k\binom{|T|-1+k-d}{|T|-1}E(T,d)\\
			&=&(-1)^{r_T-1}\sum_{d=0}^{|T|-2}\left(\sum_{k=d}^{|T|-2}\sum_{j=k+1}^{|T|-1} \frac{(-1)^{k}}{j(j+1)}\binom{j}{k+1}\binom{|T|-1+k-d}{|T|-1}\right)E(T,d).
			\end{eqnarray*}	
We denote by $K_{|T|,d}$ the sum between parentheses in the last line, and by $C_{|T|,d}$ the numbers $C_{|T|,d}:=\frac{(-1)^d}{|T|\binom{|T|-1}{d}}$. We want to show $K_{|T|,d}=C_{|T|,d}$ for all $0\leq d\leq |T|-2$: this is straightforward when $d= |T|-2$. Since the numbers $C_{|T|,d}$ obey the recursion $C_{|T|,d}=C_{|T|-1,d} + C_{|T|,d+1}$, the proof is complete if we show that for $0\leq d< |T|-2$ the numbers $K_{|T|,d}$ obey the same recursion. We compute
			\begin{eqnarray*} K_{|T|,d}-K_{|T|-1,d} &=& \sum_{k=d}^{|T|-2}\sum_{j=k+1}^{|T|-1} \frac{(-1)^{k}}{j(j+1)}\binom{j}{k+1}\binom{|T|-1+k-d}{|T|-1} \\
				&&- \sum_{k=d}^{|T|-3}\sum_{j=k+1}^{|T|-2} \frac{(-1)^{k}}{j(j+1)}\binom{j}{k+1}\binom{|T|-2+k-d}{|T|-2} \\ 
				&=& \sum_{k=d+1}^{|T|-3}\sum_{j=k+1}^{|T|-2} \frac{(-1)^{k}}{j(j+1)}\binom{j}{k+1}\binom{|T|-2+k-d}{|T|-1} \\
				&&+ \sum_{k=d}^{|T|-2} \frac{(-1)^{k}}{|T|(|T|-1)}\binom{|T|-1}{k+1}\binom{|T|-1+k-d}{|T|-1}\\
				&=& K_{|T|,d+1}- \sum_{k=d+1}^{|T|-2} \frac{(-1)^{k}}{|T|(|T|-1)}\binom{|T|-1}{k+1}\binom{|T|-2+k-d}{|T|-1}\\
				&&+ \sum_{k=d}^{|T|-2} \frac{(-1)^{k}}{|T|(|T|-1)}\binom{|T|-1}{k+1}\binom{|T|-1+k-d}{|T|-1} \\
				&=& K_{|T|,d+1} + \frac{1}{|T|(|T|-1)}\sum_{k=d}^{|T|-2} (-1)^{k}\binom{|T|-1}{k+1}\binom{|T|-2+k-d}{|T|-2}. \\
			\end{eqnarray*}
Finally, to complete the proof we have to show that for all $0\leq d< m:=|T|-2$ 
			\[ \sum_{k=d}^{m} (-1)^k\binom{m+1}{k+1}\binom{m+k-d}{m}=0. \]
To this end, we compute
\begin{eqnarray*}
\sum_{k=d}^{m} (-1)^k\binom{m+1}{k+1}\binom{m+k-d}{m} &=& \sum_{k=0}^{m-d} (-1)^{d+k}\binom{m+1}{k+d+1}\binom{m+k}{m} \\
&=& \sum_{k=0}^{m-d} (-1)^{d+k}\binom{m+1}{m-d-k}\binom{m+k}{k}\\
&=&  (-1)^d \sum_{k=0}^{m-d}\binom{m+1}{m-d-k}\binom{-m-1}{k}\\
&=& (-1)^d\binom{0}{m-d} =0,
\end{eqnarray*}
whenever $0\leq d< m$. Here we are interpreting $\binom{-m-1}{k}$ and $\binom{0}{m-d}$ as  generalized binomial coefficients, and relied on the Chu-Vandermonde identity  $\binom{s+t}{i}=\sum_{j=0}^i\binom{s}{i-j}\binom{t}{j}$ in the passage to the last line.

		\end{proof}
		After all these preparations, the proof of Theorem \ref{th:eulerian in Dynkin basis} is rather straightforward.
		
		\begin{theorem}\label{th:eulerian dynkin} With the notations from the Introduction, the element $E(x_1\cdots x_n)$ has the following expansion in Dynkin's basis $\mathcal{D}_n$ of $\Lie_n$,
			\begin{equation}\label{eq:eulerian in dynkin}
			E(x_1\cdots x_n) = \frac{1}{n} \sum_{\sigma\in S_{n-1}}\frac{(-1)^{d_\sigma}}{\binom{n-1}{d_\sigma}}[x_{\sigma(1)},\cdots[x_{\sigma(n-1)},x_n]\cdots]. \end{equation}
		\end{theorem}
		\begin{proof} We apply Proposition \ref{prop: from dynkin to pbw} to rewrite the right hand side of the claimed identity \eqref{eq:eulerian in dynkin} in the PBW basis $\basis_n$: \begin{multline*} \frac{1}{n} \sum_{\sigma\in S_{n-1}}\frac{(-1)^{d_\sigma}}{\binom{n-1}{d_\sigma}}[x_{\sigma(1)},\cdots[x_{\sigma(n-1)},x_n]\cdots] = \sum_{(T,\ell)} \left(\frac{(-1)^{r_T-1}}{n}\sum_{\sigma\in S(T,\ell)} \frac{(-1)^{d_\sigma}}{\binom{n-1}{d_\sigma}}\right) r(T,\ell) = \\ = \sum_{(T,\ell)} \left(\frac{(-1)^{r_T-1}}{n}\sum_{d=0}^{n-2} \frac{(-1)^{d}}{\binom{n-1}{d}} E(T,d) \right) r(T,\ell)  = \sum_{(T,\ell)} E_T\,r(T,\ell) = E(x_1\cdots x_n),
			\end{multline*}
			where in the last two identities we used the previous Proposition \ref{prop:combinatorial eulerian} and Corollary \ref{corollary:independence on labeling}.
		\end{proof}
		\begin{remark}\label{rem: direct proof} Of course, a more direct proof of the previous theorem than the one given here should be possible. For instance, we may expand the brackets in the right hand side of \eqref{eq:eulerian in dynkin}, and try to check that the resulting element in the free associative algebra $A(x_1,\ldots,x_n)$ coincides with $E(x_1\cdots x_n)$, as given in Formula \eqref{equation: Eulerian} from the Introduction. Given a word $w$ in $A(x_1,\ldots,x_n)$ such that each generator appears exactly once, let $w_1$ be the subword of $w$ to the left of $n$, and let $w_2$ be the subword to the right of $n$ in reverse order. Then $w$ appears in the expansion of $[x_{\sigma(1)},\cdots[x_{\sigma(n-1)},x_n]\cdots]$ if and only if $\sigma$ is a shuffle of $w_1$ and $w_2$. Using Stanley's shuffling Theorem \cite{Stanley,Goulden}, we can count for how many $\sigma$ having a fixed descent number $d$ the word $w$ appears in the expasion of $[x_{\sigma(1)},\cdots[x_{\sigma(n-1)},x_n]\cdots]$. Following this approach, the proof of \eqref{eq:eulerian in dynkin} would be reduced to the proof of a certain identity involving binomial coefficients.
		\end{remark}
		
	\newpage

		\appendix

		\section{Table of Eulerian coefficients}\label{appendix: table}
		
		Below we list all (non-planar, unrooted) trees with $n \le 8$ vertices, together with their Eulerian coefficient, see Corollary \ref{cor: independence on root} from  Subsection \ref{subsection: from pbT to pT}.
		
		\vspace{0.5cm}
		\begin{center}
			{\bf Table of Eulerian coefficients} \\ \vspace{2ex}
			\begin{tabular}{lcccccc}
				\vspace{2ex} \underline{n=1}: &  $\circ$  & $+1 \quad\quad$ \\ \vspace{2ex}
				\underline{n=2}: &  $\xy {\ar@{-}(0,0)*{\circ};(3,0)*{\circ}}; \endxy$ & $+\frac{1}{2} \quad\quad$\\ \vspace{2ex}
				\underline{n=3}: & $ \xy {\ar@{-}(0,0)*{\circ};(3,0)*{\circ}};{\ar@{-}(3,0)*{\circ};(6,0)*{\circ}}; \endxy $ & $+\frac{1}{6}\quad\quad$\\ \vspace{2ex}
				
				\underline{n=4}: &  $ \xy {\ar@{-}(0,0)*{\circ};(3,0)*{\circ}};{\ar@{-}(3,0)*{\circ};(6,0)*{\circ}};{\ar@{-}(6,0)*{\circ};(9,0)*{\circ}};  \endxy $ & $+\frac{1}{12}\quad\quad$&  
				$ \xy {\ar@{-}(0,0)*{\circ};(3,0)*{\circ}};{\ar@{-}(3,0)*{\circ};(6,0)*{\circ}};{\ar@{-}(3,0)*{\circ};(3,3)*{\circ}}; \endxy $& $0\quad\quad$\\ \vspace{2ex}
				
				\underline{n=5}: &$ \xy {\ar@{-}(0,0)*{\circ};(3,0)*{\circ}};{\ar@{-}(3,0)*{\circ};(6,0)*{\circ}};{\ar@{-}(6,0)*{\circ};(9,0)*{\circ}}; {\ar@{-}(9,0)*{\circ};(12,0)*{\circ}}; \endxy $&$+\frac{1}{30}\quad\quad$&$ \xy {\ar@{-}(0,0)*{\circ};(3,0)*{\circ}};{\ar@{-}(3,0)*{\circ};(6,0)*{\circ}};{\ar@{-}(6,0)*{\circ};(9,0)*{\circ}}; {\ar@{-}(6,0)*{\circ};(6,3)*{\circ}}; \endxy$&
				$+\frac{1}{60}\quad\quad$& $ \xy {\ar@{-}(0,0)*{\circ};(3,0)*{\circ}};{\ar@{-}(3,0)*{\circ};(6,0)*{\circ}};{\ar@{-}(3,0)*{\circ};(1,3)*{\circ}}; {\ar@{-}(3,0)*{\circ};(6,0)*{\circ}};{\ar@{-}(3,0)*{\circ};(5,3)*{\circ}};\endxy $& $-\frac{1}{30}\quad\quad$\\ 
				\underline{n=6}: & 
				$ \xy {\ar@{-}(0,0)*{\circ};(3,0)*{\circ}};{\ar@{-}(3,0)*{\circ};(6,0)*{\circ}};{\ar@{-}(6,0)*{\circ};(9,0)*{\circ}}; {\ar@{-}(9,0)*{\circ};(12,0)*{\circ}}; {\ar@{-}(12,0)*{\circ};(15,0)*{\circ}}; \endxy $ & $+\frac{1}{60}\quad\quad$ &
				$ \xy {\ar@{-}(0,0)*{\circ};(3,0)*{\circ}};{\ar@{-}(3,0)*{\circ};(6,0)*{\circ}};{\ar@{-}(6,0)*{\circ};(9,0)*{\circ}}; {\ar@{-}(9,0)*{\circ};(12,0)*{\circ}}; {\ar@{-}(9,0)*{\circ};(9,3)*{\circ}}; \endxy $& $0\quad\quad$ 			&  $ \xy {\ar@{-}(0,0)*{\circ};(3,0)*{\circ}};{\ar@{-}(3,0)*{\circ};(6,0)*{\circ}};{\ar@{-}(6,0)*{\circ};(9,0)*{\circ}}; {\ar@{-}(9,0)*{\circ};(12,0)*{\circ}}; {\ar@{-}(6,0)*{\circ};(6,3)*{\circ}}; \endxy $ & $+\frac{1}{60}\quad\quad $ \\ 
				
				& 		$ \xy {\ar@{-}(0,0)*{\circ};(3,0)*{\circ}};{\ar@{-}(3,0)*{\circ};(6,0)*{\circ}};{\ar@{-}(6,0)*{\circ};(9,0)*{\circ}}; {\ar@{-}(3,0)*{\circ};(3,3)*{\circ}}; {\ar@{-}(6,0)*{\circ};(6,3)*{\circ}}; \endxy $& $+\frac{1}{60}\quad\quad$  & 
				$ \xy {\ar@{-}(0,0)*{\circ};(3,0)*{\circ}};{\ar@{-}(3,0)*{\circ};(6,0)*{\circ}};{\ar@{-}(6,0)*{\circ};(9,0)*{\circ}}; {\ar@{-}(6,0)*{\circ};(4,3)*{\circ}}; {\ar@{-}(6,0)*{\circ};(8,3)*{\circ}}; \endxy $&$-\frac{1}{60}\quad\quad$ & 
				$ \xy {\ar@{-}(2,0)*{\circ};(6,0)*{\circ}};{\ar@{-}(6,0)*{\circ};(10,0)*{\circ}}; {\ar@{-}(6,0)*{\circ};(3.5,3)*{\circ}}; {\ar@{-}(6,0)*{\circ};(8.5,3)*{\circ}}; {\ar@{-}(6,0)*{\circ};(6,4)*{\circ}}; \endxy $& $0\quad\quad$ \vspace{2ex}\\
				\underline{n=7}:

				& 
				$ \xy {\ar@{-}(0,0)*{\circ};(3,0)*{\circ}};{\ar@{-}(3,0)*{\circ};(6,0)*{\circ}};{\ar@{-}(6,0)*{\circ};(9,0)*{\circ}}; {\ar@{-}(9,0)*{\circ};(12,0)*{\circ}}; {\ar@{-}(12,0)*{\circ};(15,0)*{\circ}};{\ar@{-}(15,0)*{\circ};(18,0)*{\circ}}; \endxy $ &$+\frac{1}{140}\quad\quad$ &
				
				$ \xy {\ar@{-}(0,0)*{\circ};(3,0)*{\circ}};{\ar@{-}(3,0)*{\circ};(6,0)*{\circ}};{\ar@{-}(6,0)*{\circ};(9,0)*{\circ}}; {\ar@{-}(9,0)*{\circ};(12,0)*{\circ}}; {\ar@{-}(12,0)*{\circ};(15,0)*{\circ}};{\ar@{-}(12,0)*{\circ};(12,3)*{\circ}}; \endxy $& $+\frac{1}{420}\quad\quad$&
				$ \xy {\ar@{-}(0,0)*{\circ};(3,0)*{\circ}};{\ar@{-}(3,0)*{\circ};(6,0)*{\circ}};{\ar@{-}(6,0)*{\circ};(9,0)*{\circ}}; {\ar@{-}(9,0)*{\circ};(12,0)*{\circ}}; {\ar@{-}(12,0)*{\circ};(15,0)*{\circ}};{\ar@{-}(9,0)*{\circ};(9,3)*{\circ}}; \endxy $& $+\frac{1}{210}\quad\quad$ \\
				
				&		$ \xy {\ar@{-}(0,0)*{\circ};(3,0)*{\circ}};{\ar@{-}(3,0)*{\circ};(6,0)*{\circ}};{\ar@{-}(6,0)*{\circ};(9,0)*{\circ}}; {\ar@{-}(9,0)*{\circ};(12,0)*{\circ}}; {\ar@{-}(6,0)*{\circ};(6,3)*{\circ}};{\ar@{-}(9,0)*{\circ};(9,3)*{\circ}}; \endxy $&$+\frac{1}{140}\quad\quad$ &
				
				$ \xy {\ar@{-}(0,0)*{\circ};(3,0)*{\circ}};{\ar@{-}(3,0)*{\circ};(6,0)*{\circ}};{\ar@{-}(6,0)*{\circ};(9,0)*{\circ}}; {\ar@{-}(9,0)*{\circ};(12,0)*{\circ}}; {\ar@{-}(3,0)*{\circ};(3,3)*{\circ}};{\ar@{-}(9,0)*{\circ};(9,3)*{\circ}}; \endxy $ & $-\frac{1}{210}\quad\quad$  &
				$ \xy {\ar@{-}(0,0)*{\circ};(3,0)*{\circ}};{\ar@{-}(3,0)*{\circ};(6,0)*{\circ}};{\ar@{-}(6,0)*{\circ};(9,0)*{\circ}}; {\ar@{-}(9,0)*{\circ};(12,0)*{\circ}}; {\ar@{-}(9,0)*{\circ};(7,3)*{\circ}};{\ar@{-}(9,0)*{\circ};(11,3)*{\circ}}; \endxy $ & $-\frac{1}{105}\quad\quad$ \\
				\vspace{2ex}
				&		$ \xy {\ar@{-}(0,0)*{\circ};(3,0)*{\circ}};{\ar@{-}(3,0)*{\circ};(6,0)*{\circ}};{\ar@{-}(6,0)*{\circ};(9,0)*{\circ}}; {\ar@{-}(9,0)*{\circ};(12,0)*{\circ}}; {\ar@{-}(6,0)*{\circ};(4,3)*{\circ}};{\ar@{-}(6,0)*{\circ};(8,3)*{\circ}}; \endxy $ &$-\frac{1}{420}\quad\quad$ &
				$ \xy {\ar@{-}(0,0)*{\circ};(3,0)*{\circ}};{\ar@{-}(3,0)*{\circ};(6,0)*{\circ}};{\ar@{-}(6,0)*{\circ};(9,0)*{\circ}}; {\ar@{-}(9,0)*{\circ};(12,0)*{\circ}}; {\ar@{-}(6,0)*{\circ};(6,3)*{\circ}}; {\ar@{-}(6,3)*{\circ};(6,6)*{\circ}}; \endxy $& $+\frac{1}{105} \quad\quad$ &
				$ \xy {\ar@{-}(6,0)*{\circ};(6,3)*{\circ}};{\ar@{-}(3,0)*{\circ};(6,0)*{\circ}};{\ar@{-}(6,0)*{\circ};(10,0)*{\circ}}; {\ar@{-}(10,0)*{\circ};(13,0)*{\circ}}; {\ar@{-}(10,0)*{\circ};(8,3)*{\circ}};{\ar@{-}(10,0)*{\circ};(12,3)*{\circ}}; \endxy $&  $+\frac{1}{420}\quad\quad $ \\
				& $ \xy {\ar@{-}(0,0)*{\circ};(3,0)*{\circ}};{\ar@{-}(3,0)*{\circ};(6,0)*{\circ}};{\ar@{-}(6,0)*{\circ};(9,0)*{\circ}}; {\ar@{-}(6,0)*{\circ};(3.5,3)*{\circ}}; {\ar@{-}(6,0)*{\circ};(8.5,3)*{\circ}}; {\ar@{-}(6,0)*{\circ};(6,4)*{\circ}}; \endxy $ &$-\frac{1}{84}\quad\quad$ &
				$ \xy {\ar@{-}(6,0)*{\circ};(2.5,2)*{\circ}};{\ar@{-}(2,0)*{\circ};(6,0)*{\circ}};{\ar@{-}(6,0)*{\circ};(10,0)*{\circ}}; {\ar@{-}(6,0)*{\circ};(4.5,3)*{\circ}}; {\ar@{-}(6,0)*{\circ};(7.5,3)*{\circ}}; {\ar@{-}(6,0)*{\circ};(9.5,2)*{\circ}}; \endxy $& $+\frac{1}{42}\quad\quad$ \\

				\underline{n=8}: 
				&
				$ \xy {\ar@{-}(0,0)*{\circ};(3,0)*{\circ}};{\ar@{-}(3,0)*{\circ};(6,0)*{\circ}};{\ar@{-}(6,0)*{\circ};(9,0)*{\circ}}; {\ar@{-}(9,0)*{\circ};(12,0)*{\circ}}; {\ar@{-}(12,0)*{\circ};(15,0)*{\circ}};{\ar@{-}(15,0)*{\circ};(18,0)*{\circ}};
				{\ar@{-}(18,0)*{\circ};(21,0)*{\circ}};			
				\endxy $ &$+\frac{1}{280}\quad\quad$ &
				
				$ \xy {\ar@{-}(-3,0)*{\circ};(0,0)*{\circ}};{\ar@{-}(0,0)*{\circ};(3,0)*{\circ}};{\ar@{-}(3,0)*{\circ};(6,0)*{\circ}};{\ar@{-}(6,0)*{\circ};(9,0)*{\circ}}; {\ar@{-}(9,0)*{\circ};(12,0)*{\circ}}; {\ar@{-}(12,0)*{\circ};(15,0)*{\circ}};{\ar@{-}(12,0)*{\circ};(12,3)*{\circ}}; \endxy $& $0\quad\quad$
				
				&
				$ \xy {\ar@{-}(-3,0)*{\circ};(0,0)*{\circ}};{\ar@{-}(0,0)*{\circ};(3,0)*{\circ}};{\ar@{-}(3,0)*{\circ};(6,0)*{\circ}};{\ar@{-}(6,0)*{\circ};(9,0)*{\circ}}; {\ar@{-}(9,0)*{\circ};(12,0)*{\circ}}; {\ar@{-}(12,0)*{\circ};(15,0)*{\circ}};{\ar@{-}(9,0)*{\circ};(9,3)*{\circ}}; \endxy $& $+\frac{1}{280}\quad\quad$ \\
				& 
				$ \xy {\ar@{-}(0,0)*{\circ};(3,0)*{\circ}};{\ar@{-}(3,0)*{\circ};(6,0)*{\circ}};{\ar@{-}(6,0)*{\circ};(9,0)*{\circ}}; {\ar@{-}(9,0)*{\circ};(12,0)*{\circ}}; {\ar@{-}(12,0)*{\circ};(15,0)*{\circ}};{\ar@{-}(9,0)*{\circ};(9,3)*{\circ}};{\ar@{-}(15,0)*{\circ};(18,0)*{\circ}}; \endxy $& $0\quad\quad$ &
				
				$ \xy {\ar@{-}(0,0)*{\circ};(3,0)*{\circ}};{\ar@{-}(3,0)*{\circ};(6,0)*{\circ}};{\ar@{-}(6,0)*{\circ};(9,0)*{\circ}}; {\ar@{-}(9,0)*{\circ};(12,0)*{\circ}}; {\ar@{-}(12,0)*{\circ};(15,0)*{\circ}};{\ar@{-}(9,0)*{\circ};(9,3)*{\circ}};{\ar@{-}(12,0)*{\circ};(12,3)*{\circ}}; \endxy $& $+\frac{1}{210}\quad\quad$ &
				$ \xy {\ar@{-}(0,0)*{\circ};(3,0)*{\circ}};{\ar@{-}(3,0)*{\circ};(6,0)*{\circ}};{\ar@{-}(6,0)*{\circ};(9,0)*{\circ}}; {\ar@{-}(9,0)*{\circ};(12,0)*{\circ}}; {\ar@{-}(12,0)*{\circ};(15,0)*{\circ}};{\ar@{-}(12,0)*{\circ};(12,3)*{\circ}};{\ar@{-}(6,0)*{\circ};(6,3)*{\circ}}; \endxy $& $-\frac{1}{420}\quad\quad$ \\
				
				& $ \xy {\ar@{-}(0,0)*{\circ};(3,0)*{\circ}};{\ar@{-}(3,0)*{\circ};(6,0)*{\circ}};{\ar@{-}(6,0)*{\circ};(9,0)*{\circ}}; {\ar@{-}(9,0)*{\circ};(12,0)*{\circ}}; {\ar@{-}(12,0)*{\circ};(15,0)*{\circ}};{\ar@{-}(12,0)*{\circ};(12,3)*{\circ}};{\ar@{-}(3,0)*{\circ};(3,3)*{\circ}}; \endxy $& $+\frac{1}{420}\quad\quad$
				
				& $ \xy {\ar@{-}(0,0)*{\circ};(3,0)*{\circ}};{\ar@{-}(3,0)*{\circ};(6,0)*{\circ}};{\ar@{-}(6,0)*{\circ};(9,0)*{\circ}}; {\ar@{-}(9,0)*{\circ};(12,0)*{\circ}}; {\ar@{-}(12,0)*{\circ};(15,0)*{\circ}};{\ar@{-}(6,0)*{\circ};(6,3)*{\circ}};{\ar@{-}(9,0)*{\circ};(9,3)*{\circ}}; \endxy $& $+\frac{1}{168}\quad\quad$ &
				$ \xy {\ar@{-}(-3,0)*{\circ};(0,0)*{\circ}};{\ar@{-}(0,0)*{\circ};(3,0)*{\circ}};{\ar@{-}(3,0)*{\circ};(6,0)*{\circ}};{\ar@{-}(6,0)*{\circ};(9,0)*{\circ}}; {\ar@{-}(9,0)*{\circ};(12,0)*{\circ}}; {\ar@{-}(9,0)*{\circ};(7,3)*{\circ}};{\ar@{-}(9,0)*{\circ};(11,3)*{\circ}}; \endxy $ & $-\frac{1}{210}\quad\quad$ 
				\\ &

				$ \xy {\ar@{-}(-3,0)*{\circ};(0,0)*{\circ}};{\ar@{-}(0,0)*{\circ};(3,0)*{\circ}};{\ar@{-}(3,0)*{\circ};(6,0)*{\circ}};{\ar@{-}(6,0)*{\circ};(9,0)*{\circ}}; {\ar@{-}(9,0)*{\circ};(12,0)*{\circ}}; {\ar@{-}(6,0)*{\circ};(4,3)*{\circ}};{\ar@{-}(6,0)*{\circ};(8,3)*{\circ}}; \endxy $ & $-\frac{1}{210}\quad\quad$ & $ \xy {\ar@{-}(-3,0)*{\circ};(0,0)*{\circ}};{\ar@{-}(0,0)*{\circ};(3,0)*{\circ}};{\ar@{-}(3,0)*{\circ};(6,0)*{\circ}};{\ar@{-}(6,0)*{\circ};(9,0)*{\circ}}; {\ar@{-}(9,0)*{\circ};(12,0)*{\circ}}; {\ar@{-}(6,0)*{\circ};(6,3)*{\circ}}; {\ar@{-}(6,3)*{\circ};(6,6)*{\circ}}; \endxy $& $+\frac{1}{210} \quad\quad$ &

				$ \xy {\ar@{-}(3,0)*{\circ};(6,0)*{\circ}};{\ar@{-}(6,0)*{\circ};(9,0)*{\circ}}; {\ar@{-}(9,0)*{\circ};(12,0)*{\circ}}; {\ar@{-}(12,0)*{\circ};(15,0)*{\circ}};{\ar@{-}(6,0)*{\circ};(6,3)*{\circ}};{\ar@{-}(9,0)*{\circ};(9,3)*{\circ}};{\ar@{-}(12,0)*{\circ};(12,3)*{\circ}}; \endxy $& $-\frac{1}{420}\quad\quad$
				\\ 
				
				&
				$ \xy {\ar@{-}(-5,0)*{\circ};(-2,0)*{\circ}};{\ar@{-}(-2,0)*{\circ};(1,0)*{\circ}};{\ar@{-}(1,0)*{\circ};(6,0)*{\circ}};{\ar@{-}(6,0)*{\circ};(9,0)*{\circ}}; {\ar@{-}(1,0)*{\circ};(1,3)*{\circ}}; {\ar@{-}(6,0)*{\circ};(4,3)*{\circ}};{\ar@{-}(6,0)*{\circ};(8,3)*{\circ}}; \endxy $ & $-\frac{1}{210}\quad\quad$
				&
				$ \xy {\ar@{-}(-3,0)*{\circ};(0,0)*{\circ}};{\ar@{-}(0,0)*{\circ};(3,0)*{\circ}};{\ar@{-}(3,0)*{\circ};(6,0)*{\circ}};{\ar@{-}(6,0)*{\circ};(9,0)*{\circ}}; {\ar@{-}(0,0)*{\circ};(0,3)*{\circ}}; {\ar@{-}(6,0)*{\circ};(4,3)*{\circ}};{\ar@{-}(6,0)*{\circ};(8,3)*{\circ}}; \endxy $ & $0\quad\quad$
				&
				$ \xy {\ar@{-}(9,0)*{\circ};(12,0)*{\circ}};{\ar@{-}(-2,0)*{\circ};(1,0)*{\circ}};{\ar@{-}(1,0)*{\circ};(6,0)*{\circ}};{\ar@{-}(6,0)*{\circ};(9,0)*{\circ}}; {\ar@{-}(1,0)*{\circ};(1,3)*{\circ}}; {\ar@{-}(6,0)*{\circ};(4,3)*{\circ}};{\ar@{-}(6,0)*{\circ};(8,3)*{\circ}}; \endxy $ & $+\frac{1}{140}\quad\quad$
				\\
				
				& $ \xy {\ar@{-}(0,0)*{\circ};(3,0)*{\circ}};{\ar@{-}(3,0)*{\circ};(6,0)*{\circ}};{\ar@{-}(6,0)*{\circ};(9,0)*{\circ}}; {\ar@{-}(9,0)*{\circ};(12,0)*{\circ}}; {\ar@{-}(6,0)*{\circ};(6,3)*{\circ}}; {\ar@{-}(6,3)*{\circ};(6,6)*{\circ}}; {\ar@{-}(9,0)*{\circ};(9,3)*{\circ}}; \endxy $& $0 \quad\quad$
				&
				$ \xy {\ar@{-}(0,0)*{\circ};(3,0)*{\circ}};{\ar@{-}(3,0)*{\circ};(6,0)*{\circ}};{\ar@{-}(6,0)*{\circ};(9,0)*{\circ}}; {\ar@{-}(9,0)*{\circ};(12,0)*{\circ}}; {\ar@{-}(6,0)*{\circ};(4,3)*{\circ}};{\ar@{-}(6,0)*{\circ};(8,3)*{\circ}};{\ar@{-}(8,3)*{\circ};(8,6)*{\circ}}; \endxy $ &$+\frac{1}{210}\quad\quad$
				
				& $ \xy {\ar@{-}(-3,0)*{\circ};(0,0)*{\circ}};{\ar@{-}(0,0)*{\circ};(3,0)*{\circ}};{\ar@{-}(3,0)*{\circ};(6,0)*{\circ}};{\ar@{-}(6,0)*{\circ};(9,0)*{\circ}}; {\ar@{-}(6,0)*{\circ};(3.5,3)*{\circ}}; {\ar@{-}(6,0)*{\circ};(8.5,3)*{\circ}}; {\ar@{-}(6,0)*{\circ};(6,4)*{\circ}}; \endxy $ &$0\quad\quad$ 
				
				\\
				&
				$ \xy {\ar@{-}(9,0)*{\circ};(12,0)*{\circ}};{\ar@{-}(0,0)*{\circ};(3,0)*{\circ}};{\ar@{-}(3,0)*{\circ};(6,0)*{\circ}};{\ar@{-}(6,0)*{\circ};(9,0)*{\circ}}; {\ar@{-}(6,0)*{\circ};(3.5,3)*{\circ}}; {\ar@{-}(6,0)*{\circ};(8.5,3)*{\circ}}; {\ar@{-}(6,0)*{\circ};(6,4)*{\circ}}; \endxy $ &$-\frac{1}{84}\quad\quad$ 
				&
				
				$ \xy {\ar@{-}(-2,0)*{\circ};(1,0)*{\circ}};{\ar@{-}(1,0)*{\circ};(6,0)*{\circ}};{\ar@{-}(6,0)*{\circ};(9,0)*{\circ}}; {\ar@{-}(1,0)*{\circ};(-1,3)*{\circ}};{\ar@{-}(1,0)*{\circ};(2.5,3)*{\circ}}; {\ar@{-}(6,0)*{\circ};(4.5,3)*{\circ}};{\ar@{-}(6,0)*{\circ};(8,3)*{\circ}}; \endxy $ & $+\frac{1}{84}\quad\quad$

				& $ \xy {\ar@{-}(6,0)*{\circ};(6,3)*{\circ}};{\ar@{-}(-3,0)*{\circ};(1,0)*{\circ}};{\ar@{-}(1,0)*{\circ};(6,0)*{\circ}};{\ar@{-}(6,0)*{\circ};(9,0)*{\circ}}; {\ar@{-}(6,0)*{\circ};(3.5,3)*{\circ}}; {\ar@{-}(6,0)*{\circ};(8.5,3)*{\circ}}; {\ar@{-}(1,0)*{\circ};(1,3)*{\circ}}; \endxy $ &$-\frac{1}{84}\quad\quad$

				\\
				
				&
				$ \xy {\ar@{-}(6,0)*{\circ};(2.5,2)*{\circ}};{\ar@{-}(2,0)*{\circ};(6,0)*{\circ}};{\ar@{-}(6,0)*{\circ};(10,0)*{\circ}}; {\ar@{-}(6,0)*{\circ};(4.5,3)*{\circ}}; {\ar@{-}(6,0)*{\circ};(7.5,3)*{\circ}}; {\ar@{-}(6,0)*{\circ};(9.5,2)*{\circ}};{\ar@{-}(-1,0)*{\circ};(2,0)*{\circ}}; \endxy $& $+\frac{1}{84}\quad\quad$ 
				&
				
				$ \xy {\ar@{-}(0,0)*{\circ};(0,4)*{\circ}};{\ar@{-}(0,0)*{\circ};(4,0)*{\circ}};{\ar@{-}(0,0)*{\circ};(-4,0)*{\circ}}; {\ar@{-}(0,0)*{\circ};(3.46,2)*{\circ}}; {\ar@{-}(0,0)*{\circ};(2,3.46)*{\circ}}; {\ar@{-}(-3.46,2)*{\circ};(0,0)*{\circ}};
				{\ar@{-}(-2,3.46)*{\circ};(0,0)*{\circ}}; \endxy $& $0\quad\quad$
				
			\end{tabular}
			\begin{remark} When tabulating the Eulerian coefficients, there are certain tricks one can use to aid the computation. We do not try to make this assertion precise, but rather illustrate it through examples. As a first example, whenever we consider a tree which looks like
				$$ 
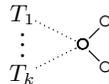
\xy 
				{\ar@{.}(4,0)*{\circ};(-4,4)*{\slab{T_1}}};
				{\ar@{.}(4,0)*{\circ};(-4,-4)*{\slab{T_k}}};
				{\ar@{-}(4,0)*{\circ};(7,3)*{\circ}};
				{\ar@{-}(4,0)*{\circ};(7,-3)*{\circ}};
				{(-4,1)*{\vdots}};
				\endxy
				$$
				the results established in Subsection \ref{subsection: from pbT to pT}, together with the identity
				\[ P\left(\xy {\ar@{-}(0,-1.5)*{\bul};(-3,1.5)*{\circ} };{\ar@{-}(0,-1.5)*{\bul};(3,1.5)*{\circ} };\endxy\right)(t) = t^2 = -2 \frac{t(1-t)}{2} + t= -2P\left(\xy {\ar@{-}(0,-3)*{\bul};(0,0)*{\circ} };{\ar@{-}(0,0)*{\circ};(0,3)*{\circ} };\endxy\right)(t)+P\left(\xy {\ar@{-}(0,-1.5)*{\bul};(0,1.5)*{\circ} };\endxy\right)(t)\] readily imply
				$$
				E\left( \,\, \xy 
				{\ar@{.}(4,0)*{\circ};(-4,4)*{\slab{T_1}}};
				{\ar@{.}(4,0)*{\circ};(-4,-4)*{\slab{T_k}}};
				{\ar@{-}(4,0)*{\circ};(7,3)*{\circ}};
				{\ar@{-}(4,0)*{\circ};(7,-3)*{\circ}};
				{(-4,1)*{\vdots}};
				\endxy \,\,\right) = -2 E\left( \,\,\xy 
				{\ar@{.}(4,0)*{\circ};(-4,4)*{\slab{T_1}}};
				{\ar@{.}(4,0)*{\circ};(-4,-4)*{\slab{T_k}}};
				{(-4,1)*{\vdots}};
				{\ar@{-}(4,0)*{\circ};(8,0)*{\circ}};
				{\ar@{-}(8,0)*{\circ};(12,0)*{\circ}};
				\endxy \,\, \right) + E\left( \,\, \xy
				{\ar@{.}(4,0)*{\circ};(-4,4)*{\slab{T_1}}};
				{\ar@{.}(4,0)*{\circ};(-4,-4)*{\slab{T_k}}};
				{(-4,1)*{\vdots}};
				{\ar@{-}(4,0)*{\circ};(8,0)*{\circ}}; \endxy \,\, \right).
				$$
				For instance, we have
				\begin{eqnarray*}
					E\left(\xy {\ar@{-}(0,-1.5)*{\circ};(3,-1.5)*{\circ}};{\ar@{-}(3,-1.5)*{\circ};(6,-1.5)*{\circ}};{\ar@{-}(6,-1.5)*{\circ};(9,-1.5)*{\circ}}; {\ar@{-}(6,-1.5)*{\circ};(6,1.5)*{\circ}}; \endxy \right) &=& -2 E(\xy {\ar@{-}(0,0)*{\circ};(3,0)*{\circ}};{\ar@{-}(3,0)*{\circ};(6,0)*{\circ}};{\ar@{-}(6,0)*{\circ};(9,0)*{\circ}}; {\ar@{-}(9,0)*{\circ};(12,0)*{\circ}}; \endxy  ) + E(\xy {\ar@{-}(0,0)*{\circ};(3,0)*{\circ}};{\ar@{-}(3,0)*{\circ};(6,0)*{\circ}};{\ar@{-}(6,0)*{\circ};(9,0)*{\circ}}; \endxy  ) = -2 \times \frac{1}{30}+\frac{1}{12} = \frac{1}{60},  \\
					E\left(\xy {\ar@{-}(-3,-1.5)*{\circ};(0,-1.5)*{\circ}};{\ar@{-}(0,-1.5)*{\circ};(3,-1.5)*{\circ}};{\ar@{-}(3,-1.5)*{\circ};(6,-1.5)*{\circ}};{\ar@{-}(3,-1.5)*{\circ};(1,1.5)*{\circ}}; {\ar@{-}(3,-1.5)*{\circ};(6,-1.5)*{\circ}};{\ar@{-}(3,-1.5)*{\circ};(5,1.5)*{\circ}};\endxy\right) &=& -2 E\left(\xy {\ar@{-}(9,-1.5)*{\circ};(12,-1.5)*{\circ}};{\ar@{-}(0,-1.5)*{\circ};(3,-1.5)*{\circ}};{\ar@{-}(3,-1.5)*{\circ};(6,-1.5)*{\circ}};{\ar@{-}(6,-1.5)*{\circ};(9,-1.5)*{\circ}}; {\ar@{-}(6,-1.5)*{\circ};(6,1.5)*{\circ}}; \endxy \right) + E\left(\xy {\ar@{-}(-3,-1.5)*{\circ};(0,-1.5)*{\circ}};{\ar@{-}(0,-1.5)*{\circ};(3,-1.5)*{\circ}};{\ar@{-}(3,-1.5)*{\circ};(6,-1.5)*{\circ}};{\ar@{-}(3,-1.5)*{\circ};(3,1.5)*{\circ}}; \endxy  \right) = -2 \times \frac{1}{60} + \frac{1}{60} = -\frac{1}{60}.
				\end{eqnarray*}
				Similarly, if we consider a tree of the form 
				
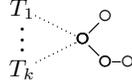
\[  \xy 
				{\ar@{.}(4,0)*{\circ};(-4,4)*{\slab{T_1}}};
				{\ar@{.}(4,0)*{\circ};(-4,-4)*{\slab{T_k}}};
				{\ar@{-}(4,0)*{\circ};(7,3)*{\circ}};
				{\ar@{-}(4,0)*{\circ};(7,-3)*{\circ}};{\ar@{-}(7,-3)*{\circ};(10,-3)*{\circ}};
				{(-4,1)*{\vdots}};
				\endxy \]
				the identity
				$ P\left(\xy {\ar@{-}(0,-4)*{\bul};(-4,0)*{\circ} };{\ar@{-}(0,-4)*{\bul};(4,0)*{\circ} };{\ar@{-}(4,0)*{\circ};(4,4)*{\circ}}\endxy\right)(t) = \frac{t^2-t^3}{2} = 3 \frac{t-t^3}{6} - \frac{t-t^2}{2}= 3P\left(\xy {\ar@{-}(0,-6)*{\bul};(0,-2)*{\circ} };{\ar@{-}(0,-2)*{\circ};(0,2)*{\circ} };{\ar@{-}(0,2)*{\circ};(0,6)*{\circ} };\endxy\right)(t)-P\left(\xy {\ar@{-}(0,-4)*{\bul};(0,0)*{\circ} };{\ar@{-}(0,0)*{\circ};(0,4)*{\circ} };\endxy\right)(t)$ implies
				$$
				E\left( \xy 
				{\ar@{.}(4,0)*{\circ};(-4,4)*{\slab{T_1}}};
				{\ar@{.}(4,0)*{\circ};(-4,-4)*{\slab{T_k}}};
				{\ar@{-}(4,0)*{\circ};(7,3)*{\circ}};
				{\ar@{-}(4,0)*{\circ};(7,-3)*{\circ}};{\ar@{-}(7,-3)*{\circ};(10,-3)*{\circ}};
				{(-4,1)*{\vdots}};
				\endxy\right) = 3 E\left( \,\,\xy 
				{\ar@{.}(4,0)*{\circ};(-4,4)*{\slab{T_1}}};
				{\ar@{.}(4,0)*{\circ};(-4,-4)*{\slab{T_k}}};
				{(-4,1)*{\vdots}};
				{\ar@{-}(4,0)*{\circ};(8,0)*{\circ}};
				{\ar@{-}(8,0)*{\circ};(12,0)*{\circ}};
				{\ar@{-}(12,0)*{\circ};(16,0)*{\circ}};
				\endxy \,\, \right) - E\left( \,\,\xy 
				{\ar@{.}(4,0)*{\circ};(-4,4)*{\slab{T_1}}};
				{\ar@{.}(4,0)*{\circ};(-4,-4)*{\slab{T_k}}};
				{(-4,1)*{\vdots}};
				{\ar@{-}(4,0)*{\circ};(8,0)*{\circ}};
				{\ar@{-}(8,0)*{\circ};(12,0)*{\circ}};
				\endxy \,\, \right).
				$$
				For instance,
				\[E\left( \xy {\ar@{-}(0,-3)*{\circ};(3,-3)*{\circ}};{\ar@{-}(3,-3)*{\circ};(6,-3)*{\circ}};{\ar@{-}(6,-3)*{\circ};(9,-3)*{\circ}}; {\ar@{-}(9,-3)*{\circ};(12,-3)*{\circ}}; {\ar@{-}(6,-3)*{\circ};(4,0)*{\circ}};{\ar@{-}(6,-3)*{\circ};(8,0)*{\circ}};{\ar@{-}(8,0)*{\circ};(8,3)*{\circ}}; \endxy\right) = 3E\left(\xy {\ar@{-}(-3,-3)*{\circ};(0,-3)*{\circ}};{\ar@{-}(0,-3)*{\circ};(3,-3)*{\circ}};{\ar@{-}(3,-3)*{\circ};(6,-3)*{\circ}};{\ar@{-}(6,-3)*{\circ};(9,-3)*{\circ}}; {\ar@{-}(9,-3)*{\circ};(12,-3)*{\circ}}; {\ar@{-}(6,-3)*{\circ};(6,0)*{\circ}}; {\ar@{-}(6,0)*{\circ};(6,3)*{\circ}}; \endxy \right) - E\left(\xy {\ar@{-}(0,-3)*{\circ};(3,-3)*{\circ}};{\ar@{-}(3,-3)*{\circ};(6,-3)*{\circ}};{\ar@{-}(6,-3)*{\circ};(9,-3)*{\circ}}; {\ar@{-}(9,-3)*{\circ};(12,-3)*{\circ}}; {\ar@{-}(6,-3)*{\circ};(6,0)*{\circ}}; {\ar@{-}(6,0)*{\circ};(6,3)*{\circ}}; \endxy \right) = 3\times \frac{1}{210} - \frac{1}{105}=\frac{1}{210}.\]
				As a final example, using the identity $P\left(\xy {\ar@{-}(0,-3)*{\bul};(0,0)*{\circ} };{\ar@{-}(0,0)*{\circ};(0,3)*{\circ} };\endxy\right)(t) =\frac{t-t^2}{2}=-\frac{1}{2}P\left(\xy {\ar@{-}(0,-1.5)*{\bul};(-3,1.5)*{\circ} };{\ar@{-}(0,-1.5)*{\bul};(3,1.5)*{\circ} };\endxy\right)(t)+\frac{1}{2}P\left(\xy {\ar@{-}(0,-1.5)*{\bul};(0,1.5)*{\circ} };\endxy\right)(t)$, we can compute
				\[ E\left(\xy {\ar@{-}(0,-3)*{\circ};(3,-3)*{\circ}};{\ar@{-}(3,-3)*{\circ};(6,-3)*{\circ}};{\ar@{-}(6,-3)*{\circ};(9,-3)*{\circ}}; {\ar@{-}(9,-3)*{\circ};(12,-3)*{\circ}}; {\ar@{-}(6,-3)*{\circ};(6,0)*{\circ}}; {\ar@{-}(6,0)*{\circ};(6,3)*{\circ}}; \endxy \right) = -\frac{1}{2}E\left(\xy {\ar@{-}(9,-1.5)*{\circ};(12,-1.5)*{\circ}};{\ar@{-}(0,-1.5)*{\circ};(3,-1.5)*{\circ}};{\ar@{-}(3,-1.5)*{\circ};(6,-1.5)*{\circ}};{\ar@{-}(6,-1.5)*{\circ};(9,-1.5)*{\circ}}; {\ar@{-}(6,-1.5)*{\circ};(4,1.5)*{\circ}};{\ar@{-}(6,-1.5)*{\circ};(8,1.5)*{\circ}}; \endxy \right)+\frac{1}{2}E\left(\xy {\ar@{-}(9,-1.5)*{\circ};(12,-1.5)*{\circ}};{\ar@{-}(0,-1.5)*{\circ};(3,-1.5)*{\circ}};{\ar@{-}(3,-1.5)*{\circ};(6,-1.5)*{\circ}};{\ar@{-}(6,-1.5)*{\circ};(9,-1.5)*{\circ}}; {\ar@{-}(6,-1.5)*{\circ};(6,1.5)*{\circ}}; \endxy \right) = - \frac{1}{2} \times\left(-\frac{1}{420}\right) +\frac{1}{2}\times \frac{1}{60}=\frac{1}{105}. \]
				
				Notice that in all of the above examples, we are expressing the Eulerian coefficient of a given tree $T$ as a linear combination of Eulerian coefficients of trees $T',T''$, such that both $T'$ and $T''$ precede $T$ in the table above. The moral is: when tabulating the Eulerian coefficients, after choosing a convenient order on the set of trees (we do not try to make this precise), the Eulerian coefficient of a given tree can be expressed as a linear combination of Eulerian coefficients which have already been computed. The previous table was obtainted using this method.
			\end{remark}

		\end{center}
		
\newpage
		\section{Umbral calculus in magmatic algebras}\label{appendix:magmatic}
		
		The umbral calculus developed in Section \ref{section: pre-Lie} is not specific to pre-Lie algebras, and can be applied to more general situations. In this appendix we consider a simple instance of this statement.
		
		\begin{definition}
			A complete magmatic algebra $(V,\vee)$ is a vector space $V$ equipped with a bilinear operation $\vee:V\otimes V\to V$ and a filtration	
			$$ \cdots \subset F^pL \subset F^{p-1}L \subset \cdots F^2L \subset F^1L = L$$
			which is complete, i.e., $V\to\underleftarrow{\operatorname{lim}}\, V/F^pV$, and compatible with $\vee$, i.e., $F^pV\vee F^qV\subset F^{p+q}V$. 
		\end{definition}		
		
		Given a $\delta$-series $f(t)=\sum_{k\geq1} \frac{c_k}{k!} t^k\in\K[[t]]$, we define a map 
		\[ f_\vee(-)\colon V\to V\colon x\to f_\vee(x) := \sum_{k\geq1} \frac{c_k}{k!}\overbrace{x\vee (\cdots( x\vee x}^k)\cdots).\]
		Following the same reasoning as in Subsection \ref{subsection: Umbral calculus}, one sees that $f_\vee(-)$ is a bijection from $V$ to itself. In fact, the same argument given there shows that the equation $ x = f_\vee(y)$ is equivalent to 
		\[ y = \sum_{k\geq0} \frac{a_k}{k!}  \overbrace{y\vee (\cdots( y}^k\vee x)\cdots), \]
		where the $a_k$ are the coefficients of the formal power series $g(t):= \frac{t}{f(t)}=\sum_{k\ge0}\frac{a_k}{k!}t^k$, and the latter can be solved recursively in $y$. Furthermore, the same proof as the one of Proposition \ref{prop:recvsdiffeq}, shows that $y$ can be computed as $\left<\left. g(D)\right| P\right>$, where $P$ is the solution of the differential equation 
		\begin{equation}\label{diffeqmag}
		\left\{ \begin{array}{l} P ' = \langle g(D)| P \rangle\vee P \\
		P(0) = x
		\end{array}  \right.
		\end{equation}
		in the magmatic algebra $(V[t],\vee)$ of polynomials with coefficient in $V$.
		
		We focus on the universal case. As well-known,  the vector space $\pbT$ (Definition \ref{def:pbT}) equipped with the magmatic product
		\[ \vee\colon\pbT\otimes\pbT \to\pbT,\qquad T_1\vee T_2\,\,=\,\, \xy
		{\ar@{-}(0,-4);(0,-2)};
		{\ar@{-}(0,-2);(-4,2)};{(-4,4)*{\slab{T_1}}};
		{\ar@{-}(0,-2);(4,2)};{(4,4)*{\slab{T_2}}};
		\endxy  \]
		is the free (complete) magmatic algebra generated by $\xy
		{\ar@{-}(0,-2);(0,2)};
		\endxy$ . Given planar, binary, rooted trees $T_1,\ldots,T_k$, we shall depict the tree $T_1\vee(\cdots(T_k\vee\,
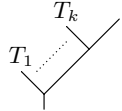
\xy
		{\ar@{-}(0,-2);(0,2)};
		\endxy\, )\cdots)$ as
		\[ \xy  {\ar@{-}(0,-4);(0,-2)};
		{\ar@{-}(0,-2);(-3,1)};{(-3,3)*{\slab{T_1}}};
		{\ar@{-}(0,-2);(10,8)};
		{\ar@{-}(6,4);(3,7) };   {(3,9)*{\slab{T_k}}};
		{\ar@{.}(-1,1);(3,5)}
		\endxy  \]
		Furthermore, we observe that every planar, binary, rooted tree $T$ has a unique decomposition  as above.
		
		The {right pointing branches} of a planar, binary rooted tree $T$ are those indicated by a dashed line in the following picture
		
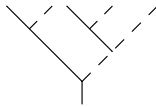
\[  \xy {\ar@{-}(0,-5);(0,-2)};
		{\ar@{-}(0,-2);(-10,8)};
		{\ar@{--}(-7,5);(-4,8)};
		{\ar@{--}(0,-2);(10,8) };
		{\ar@{-}(4,2);(-2,8)}; 	
		{\ar@{--}(1,5);(4,8)};
		\endxy \]	
		We denote the set of right pointing branches of $T$ by $B_r(T)$. Given a right pointing branch $b\in B_r(T)$, its \emph{length}, denoted by $l(b)$, is the number of edges in it. For instance, in the above picture there are two right pointing branches of length one, and one of length two.
		
		\begin{theorem} The expansion of $f_\vee^{-1}\left(\,\xy
			{\ar@{-}(0,-2);(0,2)};
			\endxy\,\right)$,  with respect to the canonical basis of $\pbT$, reads
			\[  f_\vee^{-1}\left(\,\xy
			{\ar@{-}(0,-2);(0,2)};
			\endxy\,\right) = \sum_T \left(\prod_{b\in B_r(T)}\frac{a_{l(b)}}{l(b)!}\right) T  \]
		\end{theorem}
		\begin{proof}  First of all, we shall denote by $P(T)(t)\in\K[t]$ the coefficient of $T$ in the expansion of the solution $P\in\pbT[t]$ to \eqref{diffeqmag} (with $x =\,\xy
			{\ar@{-}(0,-2);(0,2)};
			\endxy\, $). Moreover, we denote by $a_T$ the coefficient of $T$ in the expansion of $\langle g(D)| P\rangle=f_\vee^{-1}\left(\,\xy
			{\ar@{-}(0,-2);(0,2)};
			\endxy\,\right) $. It is straightforward to check that \eqref{diffeqmag} implies $P\left(\,\xy
			{\ar@{-}(0,-2);(0,2)};
			\endxy\,\right)(t)=1$. 
			
			We observe that the dual magmatic coproduct $\Delta:\pbT\to\pbT\otimes\pbT$ is given by $\Delta\left(\,\xy
			{\ar@{-}(0,-2);(0,2)};
			\endxy\,\right) =0$ and $\Delta\left(\xy
			{\ar@{-}(0,-4);(0,-2)};
			{\ar@{-}(0,-2);(-4,2)};{(-4,4)*{\slab{T'}}};
			{\ar@{-}(0,-2);(4,2)};{(4,4)*{\slab{T''}}};
			\endxy\right) =T'\otimes T''$. Thus, the analog of Lemma \ref{prop:coproduct recursion} tells us that \[ P\left(\xy
			{\ar@{-}(0,-4);(0,-2)};
			{\ar@{-}(0,-2);(-4,2)};{(-4,4)*{\slab{T'}}};
			{\ar@{-}(0,-2);(4,2)};{(4,4)*{\slab{T''}}};
			\endxy\right)(t) = a_{T'}\int_0^t P(T'')(\tau)d\tau. \]
			
			Finally, the above formula and a straightforward induction shows that for the tree 
			\[  T\,\, =\,\, 
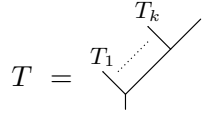
\xy  {\ar@{-}(0,-4);(0,-2)};
			{\ar@{-}(0,-2);(-3,1)};{(-3,3)*{\slab{T_1}}};
			{\ar@{-}(0,-2);(10,8)};
			{\ar@{-}(6,4);(3,7) };   {(3,9)*{\slab{T_k}}};
			{\ar@{.}(-1,1);(3,5)}
			\endxy \]	
			we have $P(T)(t)= a_{T_1}\cdots a_{T_k} \frac{t^k}{k!}$. Then, another straightforward induction, keeping in mind that a right pointing branch of $T$ is either the rightmost branch or a right pointing branch of the subtrees $T_1,\ldots,T_k$, shows that $a_T=\langle g(D)|P(T)(t)\rangle=\prod_{b\in B_r(T)}\frac{a_{l(b)}}{l(b)!}$, as desired.
		\end{proof}
		
		As a consequence, we find the following formula for the pre-Lie logarithm.
		
		\begin{corollary} Given a complete left pre-Lie algebra $(L,\rhd)$ and $x\in L$, the pre-Lie logarithm $\log_\rhd(1+x)$ is given by
			\[\log_\rhd (1+x) = \sum_T \left(\prod_{b\in B_r(T)}\frac{B_{l(b)}}{l(b)!}\right) T_\rhd(x),\]
			where the sum runs over the set of planar, binary, rooted trees, and we denote by $T_\rhd(x)\in L$ the image of $T$ under the unique morphism of magmatic algebras $(\pbT,\vee) \to (L,\rhd)\colon\,\,\xy
			{\ar@{-}(0,-2);(0,2)};
			\endxy\,\,\to x$.\end{corollary}
		\begin{example} Up to order four, the previous formula for $\log_\rhd(1+x)$ reads
			\begin{eqnarray*} \log_\rhd(1+x) &=& x - \frac{1}{2} x\rhd x + \frac{1}{12} x\rhd(x\rhd x) + \frac{1}{4} (x\rhd x)\rhd x  -\frac{1}{24} x\rhd((x\rhd x)\rhd x)\\ && - \frac{1}{24} (x\rhd x)\rhd(x\rhd x) -\frac{1}{24}(x\rhd(x\rhd x))\rhd x - \frac{1}{8} ((x\rhd x)\rhd x)\rhd x +\cdots
			\end{eqnarray*}
		\end{example}
		
\newpage
		\section{A formula for the Baker-Campbell-Hausdorff product}

		As recalled in the introduction, the Baker-Campbell-Hausdorff product $\bullet$ on a (complete) Lie algebra $\g$ and the Eulerian idempotent $E:\U(\g)\to \g$ are related by the following formula
		\begin{equation}\label{eq:eulervsBCH}
		x\bullet y = \sum_{i,j\geq0}\frac{1}{i!j!} E(x^i y^j),  \end{equation}
		see \cite{Loday}. In this appendix, we use our results on the Eulerian idempotent to deduce a new formula for the Baker-Campbell-Hausdorff product. 
		
		To state the formula, we shall introduce some notations. Given a planar, binary, rooted tree $T\in\pbT$ with $|T|$ leaves, we denote:
		\begin{itemize} \item by $r_T$ (resp.: $l_T$) the number of right (resp.: left) pointing leaves of $T$ (if $T=\,\xy {\ar@{-}(0,-2);(0,2)}\endxy\,$, by convention $r_T=1$, $l_T=0$);
			\item by $E_T$ the Eulerian coefficient of $T$;
			\item by $K_T$ the number of labelings $\ell:\{\mbox{leaves of $T$}\}\to\{1,\ldots,|T|\}$ which are admissible in the sense of Definition \ref{def:admissible}, and satisfy the additional condition: \begin{itemize} 
				\item the label of a left pointing leaf is always smaller than the label of a right pointing leaf;
			\end{itemize}
			\item given $x,y\in\g$, by $T(x,y)\in\g$ the element obtained by labeling every left pointing leaf of $T$ by $x$, every right pointing leaf of $T$ by $y$, and finally by regarding the corresponding labeled tree as an iterated bracket inside $\g$ as usual (by convention, when $T=\,\xy {\ar@{-}(0,-2);(0,2)}\endxy\,$ we put $T(x,y)=x+y$).
			
		\end{itemize}
		
		\begin{example} For instance, for the tree 
			\[  T= 
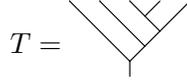
\xy {\ar@{-}(0,-4);(0,-2)};
			{\ar@{-}(0,-2);(-8,6)};
			{\ar@{-}(4,2);(0,6)};
			{\ar@{-}(0,-2);(8,6) };
			{\ar@{-}(2,0);(-4,6)}; 
			{\ar@{-}(2,4);(4,6)};	
			\endxy \]
			we have
			\[ r_T = 2,\qquad l_T=3,\qquad E_T =\frac{1}{60},\qquad K_T=1,\qquad T(x,y)=[x,[x,[[x,y],y]]],  \]
			while for the tree 
			\[ T= 
			
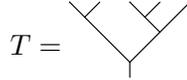
\xy {\ar@{-}(0,-4);(0,-2)};{\ar@{-}(0,-2);(-8,6)}; 
			{\ar@{-}(-6,4);(-4,6)};
			{\ar@{-}(0,-2);(8,6) };
			{\ar@{-}(4,2);(0,6)}; 
			{\ar@{-}(2,4);(4,6)}; 
			\endxy \]
			we have
			\[ r_T = 3,\qquad l_T=2,\qquad E_T =\frac{1}{30},\qquad K_T=2,\qquad T(x,y)=[[x,y],[[x,y],y]].  \]
		\end{example}
		
		Using these notations, we can state our result as follows 
		\begin{theorem} The Baker-Campbell-Hausdorff product $x\bullet y$ in $\g$ is given by the formula
			\[  x\bullet y = \sum_T \frac{E_T K_T}{l_T! r_T!} T(x,y), \]
			where the sum runs over the set of planar, binary, rooted trees.
		\end{theorem}
		\begin{proof} We apply Equation \eqref{eq:eulervsBCH}. Thus, we have to understand the element $E(x^iy^j)$ inside $\g$. By naturality, this is the image of the element $E(x_1\cdots x_{i+j})$ inside the free Lie algebra $L(x_1,\ldots, x_{i+j})$, under the morphism of Lie algebras $\psi:L(x_1,\ldots, x_{i+j})\to\g$, $x_1,\ldots,x_i\xrightarrow{\psi}x$, $x_{i+1},\ldots,x_{i+j}\xrightarrow{\psi} y$. Given a tree $T$ with precisely $(i+j)$ leaves, together with an admissible labeling $\ell:\{\mbox{leaves of $T$}\}\mapsto\{1,\ldots,i+j\}$, we denote by $T_\ell(x,y)\in\g$ the element obtained by replacing the labels $1,\ldots,i$ by $x$, the labels $i+1,\ldots,i+j$ by $y$, and finally by regarding this new labeled tree as an iterated bracket inside $\g$ as usual. According to the results from Subsection \ref{subsection: pbT}, we have $E(x^iy^j)=\sum_{(T,\ell)}E_T T_\ell(x,y)$, where the sum runs over the set of planar, binary, rooted trees with $(i+j)$ leaves and an admissible labeling $\ell$. To complete the proof, we claim that $T_\ell(x,y)=0$ unless $(T,\ell)$ has precisely $i$ left pointing leaves, labeled by the set $\{1,\ldots,i\}$, and $j$ right pointing leaves, labeled by the set $\{i+1,\ldots,i+j\}$. In fact, if this isn't the case, there has to be a right pointing leaf $l$ in $T_\ell(x,y)$ labeled by $x$. Denoting by $T'\subset T$ the subtree such that $l$ is the rightmost leaf of $T'$, by the admissiblity requirement on the labeling $\ell$, all leaves of $T'$ have to be labeled by $x$, and then the corresponding iterated bracket inside $\g$ vanishes.
		\end{proof}
		
		\begin{remark} The previous formula for the Baker-Campbell-Hausdorff product is not an optimal one, as the iterated brackets $T(x,y)$ inside $\g$ are not linearly independent among themselves. For instance, when $T=\xy {\ar@{-}(0,-4);(0,-2)};
			{\ar@{-}(0,-2);(-6,4)};
			{\ar@{-}(-4,2);(-2,4)};
			{\ar@{-}(0,-2);(6,4) };
			{\ar@{-}(4,2);(2,4)}; 	
			\endxy$ we have $T(x,y)=[[x,y],[x,y]]=0$. As another example, when $T= \xy {\ar@{-}(0,-4);(0,-2)};
			{\ar@{-}(0,-2);(-6,4)};
			{\ar@{-}(0,2);(2,4)};
			{\ar@{-}(0,-2);(6,4) };
			{\ar@{-}(2,0);(-2,4)}; 	
			\endxy$, $T'=\xy {\ar@{-}(0,-4);(0,-2)};
			{\ar@{-}(0,-2);(-6,4)};
			{\ar@{-}(0,2);(-2,4)};
			{\ar@{-}(0,-2);(6,4) };
			{\ar@{-}(-2,0);(2,4)}; 	
			\endxy$, by the Jacobi identity $T(x,y)=[x, [ [x,y] , y] ]=[ [ x, [x,y] ], y ] + [[x,y],[x,y]]=[ [ x, [x,y] ], y ]=T'(x,y)$.
			In particular, up to order four the previous formula reads 
			\[ x\bullet y= x+y+\frac{1}{2}[x,y] +\frac{1}{12}[[x,y],y] + \frac{1}{12}[x,[x,y]] + \frac{1}{48}[[x,[x,y]],y]+ \frac{1}{48}[x,[[x,y],y]]+\cdots, \]
			while an optimal formula would be 
			\[ x\bullet y= x+y+\frac{1}{2}[x,y] +\frac{1}{12}[[x,y],y] + \frac{1}{12}[x,[x,y]] + \frac{1}{24}[[x,[x,y]],y]+\cdots. \]
			It is an interesting problem to determine a subset $\mathcal{B}_{\operatorname{pb}}\subset\pbT$ of the set of planar, binary, rooted trees, such that the corresponding iterated brackets $T(x,y)\in L(x,y)$, with $T\in\mathcal{B}_{\operatorname{pb}}$, form a basis of the free Lie algebra $L(x,y)$ over $x$ and $y$, and such that in the corresponding expansion $x\bullet y = \sum_{T\in \mathcal{B}_{\operatorname{pb}}} c_T T(x,y)$, the coefficient $c_T$ can be computed in terms of purely combinatorial data associated to the tree $T$. 
			
		\end{remark}

	\end{document}